%% file: thesisvm.tex
\documentclass[
    a4paper,
    11pt
]{report} %

\usepackage[greek,english]{babel}

\usepackage{etex} 

    \usepackage[ 
    utf8
    ]{inputenc}

    \usepackage{enumitem}

    \usepackage{url} 

    \usepackage{graphicx,color} 

    \usepackage[final]{pdfpages} 

  \usepackage{hyperref}

    \usepackage[
    notcite,
    notref
    ]{}

    \usepackage[
    all
    ]{xy}

    \usepackage{tikz} 
    \usepackage{tikz-cd}
    \usetikzlibrary{shapes.geometric}
    \usetikzlibrary{arrows}
    \usetikzlibrary{positioning}

    \usepackage[dvips]{pstricks} 

    \usepackage{ragged2e}

    \usepackage{
    amsthm,
    amscd,
    amsfonts,
    amssymb,
    mathrsfs,
    amsmath
    }
    \allowdisplaybreaks[1]

    \usepackage{thmtools}

    \renewcommand\thmcontinues[1]{Continued}

    \usepackage[alphabetic]{amsrefs}

    \makeatletter
    \def\cleardoublepage{
        \clearpage
        \ifodd\c@page
        \else
        \hbox{}
        \thispagestyle{empty}
        \newpage
        \if@twocolumn
            \hbox{}
            \newpage
        \fi
        \fi
    }
    \makeatother

    \usepackage[
    refpage
    ]{nomencl}

    \usepackage{datetime}

    \newdateformat{monthyeardate}{%
        \monthname[\THEMONTH], \THEYEAR}

    \newcommand{\C}{\mathbb{C}}
    \newcommand{\R}{\mathbb{R}}
    
    \newcommand{\Z}{\mathbb{Z}}
    \newcommand{\N}{\mathbb{N}}
    \newcommand{\T}{\mathbb{T}}

    \renewcommand{\AA}{\mathfrak A}

    \newcommand{\PP}{\mathfrak P}
    \newcommand{\FF}{\mathfrak F}
    \newcommand{\GG}{\mathfrak G}
    \newcommand{\HH}{\mathfrak H}
    \newcommand{\HHb}{\bar{\mathfrak H}}

    \newcommand{\EXT}{\operatorname{Ext}}
    \newcommand{\Ho}{\operatorname{H}}
    
    \newcommand{\cohom}[3]{H^{{\raise1pt\hbox{$\scriptstyle#1$}}}(#2\>\!,#3)}
    \newcommand{\tatecohom}[3]%
        {\widehat H^{{\raise1pt\hbox{$\scriptstyle#1$}}}(#2\>\!,#3)}

    \newcommand{\Cohom}[4]%
        {\Ho^{{\raise1pt\hbox{$\scriptstyle#1$}}}_{#2}\big(#3 ;#4\big)}
    \newcommand{\GenHom}[2]%
            {\mathcal{H}_{#1}\left(#2\right)}
    \newcommand{\Tatecohom}[3]%
        {\widehat H^{{\raise1pt\hbox{$\scriptstyle#1$}}}\big(#2\>\!,#3\big)}
    \newcommand{\Ext}[4]{{\EXT}^{{\raise1pt\hbox{$\scriptstyle#1$}}}_{#2}\big(#3 ,#4\big)}
    
    \newcommand{\homol}[3]{\Ho_{{\lower1pt\hbox{$\scriptstyle#1$}}}(#2\>\!,#3)}
    \newcommand{\homolog}[2]{H_{{\lower1pt\hbox{$\scriptstyle#1$}}}(#2)}

    \newcommand{\Ab}{\mathfrak{Ab}}

    \newcommand{\Ker}[1]{\operatorname{Ker} \left( #1 \right)}

    \newcommand{\epim}{\twoheadrightarrow}

    \newcommand{\UEG}{{\underline EG}}
    \newcommand{\UEGG}[1]{{\underline E}\left(#1\right)}

    \newcommand{\EFG}[2]{\operatorname{E}_{#1} \left( #2 \right)}
    \newcommand{\BFG}[2]{\operatorname{B}_{#1} \left( #2 \right)}

    \newcommand{\All}{\mathfrak{All}}
    \newcommand{\Vcy}{\mathfrak{VC}}
    \newcommand{\Fin}{\mathfrak{Fin}}
    
    \newcommand{\RS}[3]{\mathfrak{R}_{#1}^{#2} \left(#3\right)}
    \newcommand{\dF}[2]{\mathfrak{#1}_{#2} \setminus \mathfrak{#1}_{#2-1}}
    \newcommand{\dRS}[3]{\RS{#1}{#2}{#3} \setminus \RS{#1-1}{#2}{#3}}

    \newcommand{\Nzr}[2]{N_{#1} #2}
    \newcommand{\Nzer}[2]{N_{#1}\left[#2\right]}
    \newcommand{\Nzerr}[3]{N_{#2}\left[#3\right]_{#1}}
    \newcommand{\comm}{\operatorname{Comm}}

    \DeclareMathOperator{\pd}{pd}

    \DeclareMathOperator{\cd}{cd}
    \DeclareMathOperator{\gd}{gd}
    \DeclareMathOperator{\ucd}{\underline{cd}}
    \DeclareMathOperator{\ugd}{\underline{gd}}

    \newcommand{\OFG}[2]{\mathcal{O}_{#1} #2}
    \newcommand{\RMod}[2]{\operatorname{Mod-}\OFG{#1}{#2}} 
    \newcommand{\LMod}[2]{\OFG{#1}{#2}\operatorname{-Mod}}
    \newcommand{\mor}[3]{\operatorname{mor}_{#1}\left(#2\>\!,#3\right)}
    \newcommand{\Zfirst}[2]{\Z[?\>\!,#1/#2]_{#1}}
    \newcommand{\Zsecond}[2]{\Z[#1/#2\>\!, ?]_{#1}}
    \newcommand{\Znone}[1]{\Z[?\>\!,??]_{#1}}
    \newcommand{\Zfull}[3]{\Z[#1/#2\>\!, #1/#3]_{#1}}
    \newcommand{\Ztriv}[1]{\underline{\Z}_{#1}}
    \newcommand{\modres}[1]{\operatorname{res}_{#1}}
    \newcommand{\modind}[1]{\operatorname{ind}_{#1}}
    \newcommand{\modcoind}[1]{\operatorname{coind}_{#1}}

    \newcommand{\restr}[2]{#1 \raisebox{-.5ex}{\big|}_{#2}}

    \newcommand{\mcyl}[1]{\operatorname{Cyl}\left(#1\right)}
    \newcommand{\mcone}[1]{\operatorname{C}\left(#1\right)}
    \newcommand{\dmcyl}[2]{\operatorname{Cyl}\left(#1, #2\right)}

    \newcommand{\bs}{\backslash}
    \newcommand{\st}{\,|\,}

    \numberwithin{equation}{chapter}
    \newtheorem{Theorem}{Theorem}[chapter]
    \newtheorem*{Theorem*}{Theorem}
    \newtheorem{Prop}[Theorem]{Proposition}
    \newtheorem{Proposition}[Theorem]{Proposition}
    \newtheorem{Lemma}[Theorem]{Lemma}
    \newtheorem{Corollary}[Theorem]{Corollary}
    
    \newtheorem{Remark}[Theorem]{Remark}

    \theoremstyle{definition}
        
        \newtheorem{Example}[Theorem]{Example}
        \newtheorem{Defn}[Theorem]{Definition}
        \newtheorem{Observation}[Theorem]{Observation}

    \theoremstyle{remark}

\usepackage{minitoc}
\begin{document}
\include{cover}
\setcounter{secnumdepth}{3}
\newpage
\input{declaration}
\input{Acknowledgements.tex}

\input{abstract.tex}
\dominitoc
\tableofcontents
\justify
\input{TableOfContents}
\end{document}

%% file: cover.tex
\thispagestyle{empty}

\begin{center}

    \vspace*{0.8in}

    {\LARGE\bf Classifying spaces for chains of families of subgroups\\}

    \vspace{0.3in}
    {\large V\'ictor Moreno}

    \vspace{2in}
    Thesis submitted to the University of London\\
    for the degree of Doctor of Philosophy

    \vspace{2in}
    School of Mathematics and Information Security\\
    Royal Holloway, University of London

    \vspace{0.2in}
    \monthyeardate{November 2018}

\end{center}

\newpage
\thispagestyle{empty}
\newpage

%% file: declaration.tex
\chapter*{Declaration}
\parskip=11pt

These doctoral studies were conducted under the supervision of
Professor Brita Nucinkis.

The work presented in this thesis is the result of original research I conducted, in
collaboration with others, whilst enrolled in the School of Mathematics and Information Security as a candidate for the
degree of Doctor of Philosophy. This work has not been submitted for any other degree or award in
any other university or educational establishment.

\bigskip\bigskip\bigskip

\begin{figure}[h!]
\hspace{\fill} \includegraphics[width=4cm]{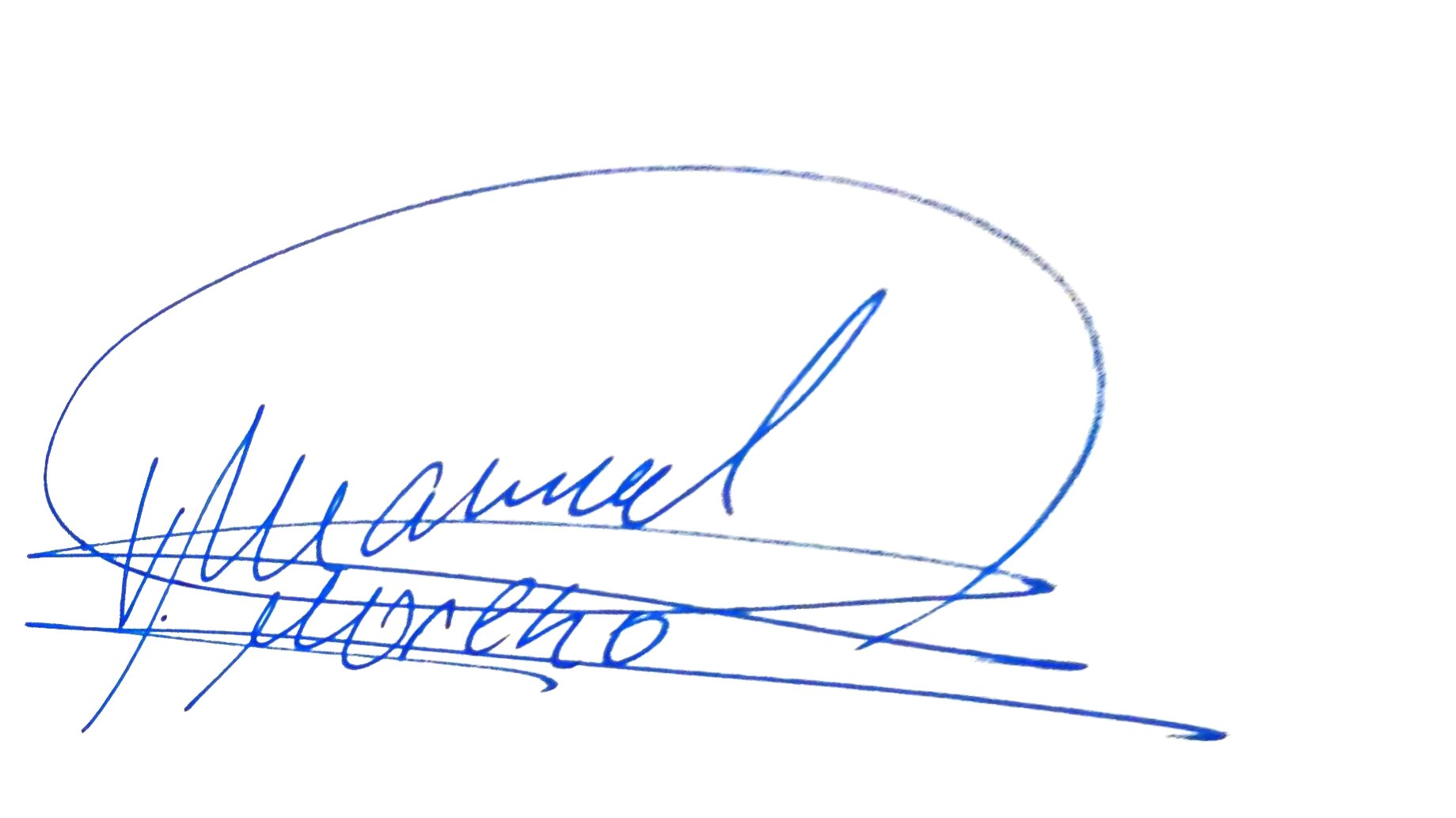}
\end{figure}

\hspace{\fill} V\'ictor Moreno\\
\smallskip
\hspace{\fill} November 25, 2018
\parskip=0pt
\newpage

%% file: Acknowledgements.tex
\chapter*{Acknowledgements}

This work would not have been possible without the aid and support
of the people that surrounded me over these years.

I am profoundly grateful to my supervisor Professor Brita Nucinkis for showing
me the ways of research and always finding what I needed the most, be it
well-timed encouraging words, mathematical insights, showing passion for
what I was learning, the recognition of progress in my work or a reminder that the
wheel was already invented. For all this and much more I learned in our
time working together, thank you.

For persuading me to pursue this opportunity and for being always a source of
kindness and energy for many different topics, I want to thank Pep Burillo.

As I shared my first steps in the field with them, I would like to thank Federico Pasini and Ged Corob-Cook.
Learning and working side by side was a great welcome for me, and the fruits of our collaboration
gave shape to and inspired the present work.

Many were those who made my stay at Royal Holloway great. Thank you Christian, Matteo,
Eugenio, Pips, Sam, George, Rachel, Naomi, Wanpeng, Amit and Alex.
Heartfelt thanks go to Thalia, Thyla and Pavlo for the friendship
we built over these years.

I would also like to thank Claudia, Camila, Maria P., Hector, Maria F., Anna, Su,
Alfonso, Stefan and Pablo for keeping me sane and becoming a second family for me
in each of the locations I lived during these years.

A mi familia, especialmente a mi madre, gracias de todo coraz\'on por
el apoyo y cari\~no que siempre me hac\'eis llegar, sin importar lo lejos
que estemos.

And finally, my dearest thank you to Marina, for being the most wonderful partner
I could have hoped for. \textgreek{Γεμίζεις τις καλές μου στιγμές με χαρά, με φως
και σταθερότητα τις σκοτεινές μου στιγμές, κάνοντας με πάντοτε καλύτερο.
Αγαπώ σε, καρδιά μου.}

%% file: abstract.tex
\chapter*{Abstract}
{

This thesis concerns the study of the Bredon cohomological and geometric
dimensions of a discrete group $G$ with respect to a family $\FF$ of subgroups of $G$.
With that purpose, we focus on building finite-dimensional models for $\EFG{\FF}{G}$.
The cases of the family $\Fin$ of finite subgroups of a group and the family $\Vcy$ of virtually
cyclic subgroups of a group have been widely studied and many tools have been developed
to relate the classifying spaces for $\Vcy$ with those for $\Fin$.

Given a discrete group $G$ and an ascending chain
$\FF_0 \subseteq \FF_1 \subseteq \ldots \subseteq \FF_n \subseteq \ldots$ of families of subgroups
of $G$, we provide a recursive methodology to build models for $\EFG{\FF_r}{G}$ and give
certain conditions under which the models obtained are finite-dimensional. We provide
upper bounds for both the Bredon cohomological and geometric dimensions of $G$ with
respect to the families $\left(\FF_r\right)_{r\in\N}$ utilising the classifying spaces
obtained.

We consider then the families $\HH_r$ of virtually polycyclic subgroups of Hirsch length
less than or equal to $r$, for $r\in\N$. We apply the results obtained for chains of families
of subgroups to the chain $\HH_0 \subseteq \HH_1 \subseteq \ldots$ for an arbitrary virtually
polycyclic group $G$, proving that the corresponding Bredon dimensions are both bounded above
by $h(G) + r$, where $h(G)$ is the Hirsch length of $G$.

Finally, we give similar results for the same chain of families of subgroups and an arbitrary locally
virtually polycyclic group as the ambient group, obtaining in this case the upper bound $h(G) + r  + 1$.

}

%% file: TableOfContents.tex
\setcounter{chapter}{-1}

\include{Introduction}

\include{ClassSpaces}

\include{BredonCohomology}

\include{BredonDimensions}

\include{RelatedFamilies}

\include{ClassForChains}

\include{ClassForPoly}

\include{Bibliography}

%% file: Introduction.tex
\chapter{Introduction}\label{ch:intro}

Given a group $G$, we say that a non-empty collection $\FF$ of subgroups of $G$ is a \emph{family} if it
is closed under conjugation and taking subgroups. In this configuration, a $G$-CW-complex $X$ is a model
for $E_\FF G$ or a classifying space for the family $\FF$ if for each subgroup $H \leq G$, the set of
points of $X$ that $H$ fixes is contractible if $H\in\FF$ and empty otherwise.

A universal property holds for such spaces, meaning that if $X$ is a model for $E_\FF G$ and $Y$
is any $G$-CW-complex with stabilizers in $\FF$, there is a $G$-map $f: Y \to X$ unique up to
homotopy. In other words, a model for $E_\FF G$ is a terminal object in the homotopy category of
$G$-CW-complexes with stabilizers in $\FF$. As such, their existence is guaranteed for any
group and family of subgroups \cite{luecksurvey}. Given their universal existence, the interest
lies in determining the least possible dimension of a model for $E_\FF G$, i.e.
the \emph{Bredon geometric dimension} of $G$ with respect to the family $\FF$, denoted by $\gd_\FF G$.

Homological methods facilitate the study of such dimensions. In the case of $G$-CW-complexes
with stabilisers in a family $\FF$, the Bredon cohomology of groups is the most suitable tool. Glen Bredon
introduced this homology theory in \cite{bredon_equiv} for finite groups and
Wolfgang L{\"u}ck extended it for arbitrary groups and families of subgroups in \cite{luecktransformation}.

The spaces $\underline{E}G = E_{\mathfrak{F}}G$, known as universal space for proper actions, for
the family $\mathfrak{F} = \Fin$ of finite subgroups and
$\underline{\underline{E}} G = E_{\mathfrak{F}} G$ for the family $\mathfrak{F} = \Vcy$ of virtually
cyclic subgroups have been widely studied given their appearance as the geometric objects in the Baum-Connes
and Farrell-Jones conjectures, respectively. For a first introduction into the subject see, for example,
the survey \cite{luecksurvey}.

In the case of $\Fin$, explicit constructions of the respective models arise in a natural way
from the geometrical origin, interpretation or properties of many classes of groups. For $\Vcy$,
however, building such spaces has proven more challenging. As such, methodologies
that could help obtain the desired models using known classifying spaces for the family of finite
subgroups and other related families have been developed.

Such methodologies made the construction of classifying spaces for families of subgroups
other than $\Fin$ and $\Vcy$ reachable. For example, in \cite{nucinkispetrosyan_hier}, the authors build
$3$-dimensional classifying spaces for the family of virtually nilpotent subgroups of any
abelian-by-infinite cyclic groups. Another example more aligned with $\Fin$ and $\Vcy$ can be
found in \cite{nucinkisetal}, where $(n+r)$-dimensional models for $E_{\FF_r} G$, where $G$
is finitely generated abelian and $\FF_r$ is the family of subgroups of torsion-free rank less
than or equal to $r$, are constructed in a recursive manner.

Let $\FF_0 \subseteq \FF_1 \subseteq \ldots \subseteq \FF_n \subseteq \ldots$ be an ascending
chain of families of subgroups of a discrete group $G$. Under certain conditions, it is possible
to build classifying spaces for all the families in the chain recursively, utilising those
for $\FF_0$ and other families that will be introduced throughout the process. The aim of this
thesis is to provide such methodology and use it to give upper bounds for the respective Bredon dimensions.
We then apply this construction process to families of virtually polycyclic subgroups.

\section*{Structure of the thesis}
In Chapter~\ref{ch:class_spaces}, we present CW-complexes and their equivariant counterparts:
$G$-CW-complexes. We also introduce some operations with such spaces that will be used throughout
the thesis and the conditions under which the resulting spaces are also ($G$-)CW-complexes. These
operations are quotients, products, joins, push-outs and, in particular, mapping cones and
mapping and double mapping cylinders. Finally, we give the definition and some basic properties
and examples of families of subgroups and classifying spaces.

In Chapter~\ref{ch:bredon_cohom}, we introduce the orbit category $\OFG{\FF}{G}$ and Bredon modules
as functors from said category to the category of abelian groups. We also give a basic overview
of free and projective objects of the category of Bredon modules, as they are key for defining
Bredon Cohomology (Chapter~\ref{ch:bredon_cohom}) and describing its relation to
classifying spaces (Chapter~\ref{ch:bredon_dim}).

In Chapter~\ref{ch:bredon_dim}, apart from specifying some results about the aforementioned
relation, we define the Bredon cohomological and geometric dimensions of a group $G$ with
respect to a family of subgroups $\FF$, $\cd_{\FF} G$ and $\gd_{\FF} G$, respectively.

In Chapter~\ref{ch:related-families}, we compile and extend a list of results that, given related
families and groups, connect their respective classifying spaces, Bredon cohomology groups and
Bredon dimensions. The particular cases we look into are: a family and its restriction to a
subgroup of the ambient group, families that are related by a functor, unions of families,
pairs of families $\FF \subseteq \GG$ such that the set $\GG\bs\FF$ admits certain structure and
families of subgroups of a direct union of groups and their restrictions to the groups appearing
in the direct union.

The main contribution of this thesis is Chapter~\ref{ch:chains}. Given an ascending chain
$\left(\FF_r\right)_{r\in\N}$ of families of subgroups of a discrete group and provided that
the chain has certain properties, we develop a methodology
based on the results in Chapter~\ref{ch:related-families} to build models for
$\EFG{\FF_r}{G}$ recursively. We find upper bounds for the Bredon dimensions with respect to
$\FF_r$ depending on those with respect to $\FF_{r-1}$ and other related families.
In the case that there are finite-dimensional classifying spaces
for the family $\FF_0$ for the ambient group and some of its subgroups, we list some further
conditions that will ensure that the models we build
have also finite dimension.

Finally, in Chapter~\ref{ch:class_poly}, we apply the results in the previous chapter to the
chain $\left(\FF_r\right)_{r\in\N}$ of families of subgroups, where $H\in\FF_r$ if and only if
$H$ is virtually polycyclic and its Hirsch length is smaller than or equal to $r$. We consider
two different classes of groups for the ambient group: virtually polycyclic groups and
locally virtually polycyclic groups.

%% file: ClassSpaces.tex
\chapter{Classifying Spaces for families of subgroups}\label{ch:class_spaces}
\minitoc

\section{CW-complexes}\label{sect_cwcomp}

In \cite{whitehead_combhomotopy1}, J.H.C. Whitehead introduced CW-complexes as a class of
topological spaces that could play the role of simplicial complexes in Homotopy
Theory and would allow the field to be studied from a different perspective.
The results we will present in this section can be found, in most of the cases,
in the aformentioned publication. Given the nature of their first appearance,
their description and study in \cite{whitehead_combhomotopy1} is based on their desired properties.
For a more constructive view of CW-complexes,
we introduce them as in modern publications such as \cite{Hatcher_algebraictop} and relate those results and definitions to the ones found in the original source.

\begin{Defn}\label{attaching_space}
    Let $X$ and $Y$ be topological spaces. Let $f : A \to Y$ be a continuous map,
    where $A \subset X$ is a subspace. Then, the \emph{attaching space} or \emph{adjunction space}
    for $f$ is
    $$ X \cup_f Y = (X \sqcup Y)/\sim,$$
    where $\sim$ is the equivalence relation generated by $f(a)\sim a$ for all $a \in A$.
\end{Defn}

\begin{Defn}\label{modernCWc} A non-empty topological space $X$ is a \emph{CW-complex} if it admits a filtration $X^0 \subseteq X^1 \subseteq \ldots \subseteq X^n \subseteq \ldots \subseteq X$ such that:
\begin{enumerate}[label= {(\roman*)}]
\item $X^0$ is a non-empty discrete set of points ($0$-cells).
\item\label{attach} For each $n \geq 1$, $X^n$ can be obtained by attaching $n$-cells $e^n_\alpha$ to $X^{n-1}$ via maps $\phi_\alpha : S^{n-1} \to X^{n-1}$.
\item $X = \bigcup_n X^n$ and the topology in $X$ coincides with the weak topology associated to such filtration.

Under these conditions, $X^n$ will be called $n$-\emph{skeleton} of $X$.
\end{enumerate}

\end{Defn}

According to Definition~\ref{attaching_space}, the condition \ref{attach} in the definition of CW-complex
means that $X^n$ is the quotient space of the disjoint union of
$X^{n-1}$ with a collection of $n$-discs under the identification $x \sim \phi_{\alpha} (x)$ for $x \in \partial D^n_{\alpha}$, i.e.
$X^n = \left(X^{n-1} \bigsqcup_\alpha D^n_\alpha \right) / \sim$ where $\sim$ is the equivalence relation generated by $x \sim \phi_\alpha (x)$ for $x\in \partial D^n_\alpha$. The $n$-cell $e^n_{\alpha}$ is the homeomorphic image of $D^n_{\alpha} \setminus \partial D^n_{\alpha}$.

\begin{Defn} If $X = X^n$ for some $n$, then $X$ is \emph{finite-dimensional}, in which case we will say that its \emph{dimension} is
$\dim (X) = \min\{ n \in \N \st X = X^n\}$. If there is no such $n$, then $\dim (X) = \infty$.
\end{Defn}

\begin{Defn}
A CW-complex $X$ is called \emph{finite} if it has only finitely many cells.
\end{Defn}

\begin{Example} The $n$-sphere is a CW-complex. For example, if we take one $0$-cell $e^0$ and one $n$-cell $e^n$ with constant attaching map $S^{n-1} \to e^0$, then $S^n = X^n$.
\end{Example}

\begin{Example}\label{RasCWcompl} $\R$ is a CW-complex. Let $X^0 = \Z$ and for each $i\in \Z$ attach to $X^0$ a $1$-cell $e_i^1$ via the map $\phi_i: S^0 = \{0,1\} \to \Z$ defined by $\phi_i (0) = i$ and $\phi_i (1) = i+1$ to obtain $X^1$. Then, $X^1 = \R$.

More generally, $\R^n$ is a CW-complex.
\end{Example}

\begin{Example} Let $\Gamma = ( V, E)$ be a graph with vertices $V$ and edges $E$. Take as the $0$-cells the vertices, $X^0 = \bigcup_{v\in V} e_v^0$, and for each edge $\{u,v\}\in E$ attach to $X^0$ a $1$-cell  $e_{\{u,v\}}^1$ via the map $\phi_{\{u,v\}}: S^0 = \{0,1\} \to X^0$ defined by $\phi_{\{u,v\}}(0) = u$ and $\phi_{\{u,v\}}(1) = v$ to obtain $X^1$. Then, $\Gamma = X^1$ is a CW-complex.
\end{Example}

\begin{Example}\label{TasCWcompl} The torus $\T$ is a CW-complex. A filtration fitting the definition would be the following:

Take $X^0 = \{*\}$. Attach to it two $1$-cells $e_1^1, e_2^1$ via the constant maps $\phi_1 = \phi_2: S^0 \to \{*\}$ to obtain $X^1$. Take $e^2$ a single $2$-cell and consider the map $\phi : S^1 \to X^1$ such that it sends each pair of opposite quarters of the $S^1$ to a different $1$-cell. Then, $\T = X^2$.
\end{Example}

These examples help us to get an idea of how important to the CW-complex structure of a topological space $X$ is the way the discs $D^n_{\alpha}$ are incorporated into $X$. That leads to the following definition, that we will use further in this section to clarify what \emph{CW} in CW-complexes stand for:

\begin{Defn} For each $n$-cell $e^n_{\alpha}$ of a CW-complex $X$ we define its \emph{characteristic map} $\Phi^n_{\alpha} : D^n_{\alpha} \to X$ by the composition $$D^n_{\alpha} \hookrightarrow X^{n-1} \bigsqcup_\alpha D^n_\alpha \epim X^n \hookrightarrow X.$$
\end{Defn}

\begin{Remark} If $X$ is a CW-complex and $e^n_{\alpha}$ is any of its $n$-cells, then $\Phi^n_{\alpha}$ is continuous and its restriction to the interior of $D^n_{\alpha}$ is a homeomorphism onto $e^n_{\alpha}$.
\end{Remark}

\begin{Remark} If $A \subseteq X$ is an open (closed) subset of a CW-complex $X$, then a map $f : A \to Y$, where $Y$ is any topological space, is continuous if and only if its restriction $\restr{f}{A \cap \bar{e}}$ is continuous for all cells of $X$.
\end{Remark}

And it is always important when we define a mathematical object (or structure) to define and study its sub-objects (sub-structures):

\begin{Defn} Given a CW-complex $X$, a subspace $A \subset X$ is a \emph{subcomplex} of $X$ if for every cell $e$ of $X$, if there is $p\in e$ such that $p\in A$, then $\bar{e} \subset A$, where $\bar{e}$ is the closure of ${e}$.
\end{Defn}

Equivalently, a subspace $A$ of a CW-complex $X$ is a subcomplex if it is a closed subset that is the
union of a set of cells of $X$.

\begin{Defn} Given a CW-complex $X$ and a set of points $P\subseteq X$, the \emph{closure} of $P$ in $X$, denoted by $X(P)$, is the smallest subcomplex of $X$ containing $P$, i.e., the intersection of all subcomplexes of $X$ that contain $P$.
\end{Defn}

Definition~\ref{modernCWc} introduces CW-complexes in a constructive way, making it easier to work with such spaces. At this point we are able to relate this definition with the original one that J. H. C. Whitehead formulated in 1949 in \cite{whitehead_combhomotopy1}. Whitehead's definition, even if less practical than the one commonly used nowadays, gives a better insight in why such spaces are called CW-complexes and which topological necessities they were defined to cover in Homotopy Theory.

\begin{Defn}\cite[Section 4]{whitehead_combhomotopy1}
A Hausdorff space $X$ is called a \emph{cell complex} if it is the union of disjoint open cells $e^n_{\alpha}$ subject to the following condition: the closure $\bar{e}^n_{\alpha}$ of each $n$-cell $e^n_{\alpha} \in X$ shall be the image of a fixed $n$-simplex $\sigma^n_\alpha$ by a map $f : \sigma^n_\alpha \to \bar{e}^n_\alpha$ such that
\begin{enumerate}[label={(\roman*)}]
\item $\restr{f}{(\sigma^n_\alpha \bs \partial\sigma^n_\alpha)}$ is a homeomorphism onto $e^n_\alpha$
\item $\partial e^n_\alpha \subset X^{n-1}$, where $\partial e^n_\alpha = f \partial \sigma^b_\alpha = \bar{e}^n_\alpha \bs e^n_\alpha$ and $X^{n-1}$ is the $(n-1)$-skeleton of $X$, consisting of all the cells whose dimensionalities do not exceed $n-1$.
\end{enumerate}
\end{Defn}

\begin{Defn}\cite[Section 5]{whitehead_combhomotopy1}
A cell complex $X$ is said to be \emph{closure finite} if $X(e)$ is a finite subcomplex of $X$ for every cell.
\end{Defn}

Note that the notion of \emph{subcomplex} refers to cell complexes and not to CW-complexes. A subspace $A$ of a cell complex $X$ is a subcomplex if it is the union of a subset of $X$'s cells such that $e \subseteq L$ implies $\bar{e} \subseteq L$ for all cells of $X$.

\begin{Defn}\cite[Section 5]{whitehead_combhomotopy1}
A cell complex $X$ has the \emph{weak topology} if a subset $U \subseteq X$ is closed provided $U \cap \bar{e}$ is closed for each cell $e\in X$.
\end{Defn}

And finally, Whitehead's definition of CW-complex:

\begin{Defn}\cite[Section 5]{whitehead_combhomotopy1}\label{originalCWc}
A cell complex $X$ is a \emph{CW-complex} if it is closure finite and has the weak topology.
\end{Defn}

Now we can see clearly that \emph{C} stands for \emph{closure finite} and \emph{W} for \emph{weak topology}.

A proof of the equivalence between Definition~\ref{modernCWc} and Definition~\ref{originalCWc} can be found, for example, in the Appendix of~\cite{Hatcher_algebraictop}:

\begin{Prop}\cite[Proposition A.2.]{Hatcher_algebraictop}
    Given a Hausdorff space $X$ and a family of maps $\Phi^n_\alpha : D^n_\alpha \to X$,
    then these maps are the characteristic maps of a CW-complex (as in Definition~\ref{modernCWc})
    structure on $X$ if and only if:
    \begin{enumerate}[label={\emph{(\roman*)}},noitemsep]
        \item each $\Phi^n_\alpha$ restricts to a homeomorphism from $\mathring{D}^n_\alpha = D^n_\alpha \bs \partial D^n_\alpha$ onto its image, a cell $e^n_\alpha \subseteq X$;
        \item for each cell $e^n_\alpha$, $\Phi^n_\alpha ( \partial D^n_\alpha )$ is contained in a finite subcomplex whose cells have dimension strictly less than $n$; and
        \item a subset of $X$ is closed if and only if it meets  the topological closure of each cell of $X$ in a closed set.
    \end{enumerate}
\end{Prop}

As one may note, previous proposition alone does not prove the equivalence of the two definitions. For its completion, CW-complexes defined as in Definition~\ref{modernCWc} have to be Hausdorff. A stronger property is true:

\begin{Prop}\cite[Proposition A.3.]{Hatcher_algebraictop}
CW-complexes (as in Definition~\ref{modernCWc}) are normal and, in particular, Hausdorff.
\end{Prop}

We shall now introduce some constructions preserving the structure of CW-complexes that will be used directly or indirectly for the results on this thesis.

\subsection{Quotients}

Let $X$ be a CW-complex and $A\subseteq X$ a subcomplex of $X$. Then the space $X / A$ inherits a CW-structure from $X$ naturally, by keeping the cells from $X \smallsetminus A$ and identifying $A$ with an extra $0$-cell.

In a more general case that will be useful when considering mapping cones and joins,
for example, we have:

\begin{Prop}\cite[(F) in Section 5]{whitehead_combhomotopy1}\label{QovermapCW}
    If $X$ is a CW-complex, $L$ is a closure finite complex and $\pi : X \to L$
    is a surjective map such that:
    \begin{enumerate}[label= {\emph{\arabic*)}}, noitemsep]
        \item $L$ has the identification topology determined by $\pi$ and
        \item $L(f(\bar{e}))$ is finite for every cell $e\in X$,
    \end{enumerate}
    then $L$ is a CW-complex.
\end{Prop}

\subsection{Product}

\begin{Defn}
    Let $X$ and $Y$ be CW-complexes with cells $e^n_\alpha$ and $\tilde{e}^n_\beta$ and characteristic
    maps $\Phi^n_{\alpha} : D^n_{\alpha} \to X$ and $\tilde{\Phi}^n_{\beta} : D^n_{\beta} \to Y$, respectively.
    We say the \emph{product cellular structure} of $X\times Y$ is the one defined
    by ${(X\times Y)}^n = \{ e^k_\alpha \times \tilde{e}^l_\beta \st 0 \leq k + l \leq n \}$ and characteristic maps
    \begin{align*}
        \Psi^{k,l}_{\alpha,\beta} : D^{k+l}_{\alpha,\beta} & \longrightarrow  X\times Y \\
        p & \longmapsto \Psi^{k,l}_{\alpha,\beta} \left( p \right) = \left(\Phi^k_{\alpha}(p_X),\tilde{\Phi}^l_{\beta}(p_Y) \right),
    \end{align*}
    where $p_X = f_{k,l}(\pi^k_1(p))$ and $p_Y = f_{k,l}(\pi^l_2(p))$ for the
    natural homeomorphism $f_{k,l}: D^{k+l} \to D^k \times D^l$ and the projections
    $\pi^k_1, \pi^l_2$ from $D^k \times D^l$ to $D^k$ and $D^l$, respectively.
\end{Defn}

Note that, for any cells $e^k_\alpha \in X$ and $e^l_\beta\in Y$,
${e^k_\alpha \times \tilde{e}^l_\beta} \subset {X(e^k_\alpha) \times Y(\tilde{e}^l_\beta)}$ holds,
being the latter a subcomplex of $X\times Y$. Then, by definition of closure of a
subset in a CW-complex, ${(X\times Y)\left(e^k_\alpha \times \tilde{e}^l_\beta\right)} \subset {X(e^k_\alpha) \times Y(\tilde{e}^l_\beta)}$
and therefore since $X$ and $Y$ a closure finite, so is $X\times Y$.

However, $X$ and $Y$ having the weak topology with respect to their CW-complex structures
doesn't generally mean $X\times Y$ has the weak topology with respect to the product cellular
structure defined above.

According to Theorem A.6 (\cite[Appendix: Topology of Cell Complexes]{Hatcher_algebraictop}) and
to Propositions (D) and (H) (\cite[Section 5]{whitehead_combhomotopy1}) we have some
conditions on $X$ and $Y$ for their product (together with the product cellular structure) to be
a CW-complex:

\begin{Theorem}\label{productCW}
    Let $X$ and $Y$ be CW-complexes. Then $X\times Y$ with the product cellular
    structure defined above is a CW-complex if any of the following is true:
    \begin{enumerate}[label={\emph{(\roman*)}},noitemsep]
        \item either $X$ or $Y$ is locally compact;
        \item either $X$ of $Y$ is locally finite;
        \item both $X$ and $Y$ have finitely many cells.
    \end{enumerate}

\end{Theorem}

\subsection{Join}

\begin{Defn} The \emph{join} of two non-empty topological spaces $X$ and $Y$, denoted by $X * Y$, is given by the quotient
$$ X * Y = X \times Y \times [0,1] / \sim , $$
where $\sim$ is the equivalence relation generated by $(x,y_1,0) \sim (x,y_2,0)$ for all $x\in X$ and $y_1,y_2 \in Y$ and $(x_1,y,1) \sim (x_2,y,1) $ for all $x_1,x_2 \in X$ and $y \in Y$.
\end{Defn}

\begin{Corollary}\label{join_is_CW}
    If $X$ and $Y$ are CW-complexes such that any of the conditions in Theorem~\ref{productCW}
    is true, then $X * Y$ admits a CW-structure inherited from
    the product cellular structure of $X \times Y \times [0,1]$ and the projections
    $\pi_X : X \times Y \times \{0\} \to X$ and $\pi_Y : X \times Y \times \{1\} \to Y$.
\end{Corollary}

\begin{proof}
    By Theorem~\ref{productCW} and since $[0,1]$ locally compact, both $X \times Y$
    and $X \times Y \times [0,1]$ are CW-complexes with respect to the corresponding
    product cellular structure.

    Let $C_X = X \times Y \times [0,1] / \sim_0$, where $\sim_0$ is the equivalence
    relation generated by $(x,y_1,0) \sim_0 (x,y_2,0)$ for all $x\in X$ and $y_1,y_2 \in Y$.
    Let ${\tilde{\pi}_X : X \times Y \times [0,1] \to C_X}$ be the corresponding
    quotient map. Then, applying Proposition~\ref{QovermapCW} to $\tilde{\pi}_X$,
    we have that $C_X$ is a CW-complex.

    We can express $X * Y$ as the quotient $C_X / \sim_1$ where $\sim_1$ is the
    equivalence relation generated by $(x_1,y,1) \sim_1 (x_2,y,1) $ for all $x_1,x_2 \in X$
    and $y \in Y$. Let ${\tilde{\pi}_Y : C_X \to X * Y}$ be the corresponding quotient
    map. Then, applying Proposition~\ref{QovermapCW} to $\tilde{\pi}_Y$,
    we have that $X * Y$ is a CW-complex,, as we wanted to see.
\end{proof}

\subsection{Attaching spaces along maps}

Some important examples of attaching spaces that we will use throughout this thesis are the following:

\begin{Defn}\label{defconcyl}
    Given $f : X \to Y$ a continous map between topological spaces, the \emph{mapping cylinder of }$f$
    is $$\mcyl{f} = \left(X\times [0,1]\right)\cup_g Y,$$
    where $g : X\times \{1\} \subset X\times [0,1]  \to Y$ is defined by $g(x,1) = f(x)$.

    The \emph{mapping cone of } $f$ is
    $$\mcone{f} = \mcyl{f} / \sim,$$
    where $\sim$ is the equivalence relation generated by $(x,0) \sim (x^\prime, 0)$ for all
    $x, x^\prime \in X$.
\end{Defn}

\begin{Defn}\label{defdoublecyl}
    Given $f : X \to Y$ and $g : X \to Z$ continuous maps between topological spaces,
    the \emph{double mapping cylinder of } $Y \xleftarrow{f} X \xrightarrow{g} Z$ is
    $$\dmcyl{f}{g} = \left(X\times [0,1]\right)\cup_k \left(Y\sqcup Z\right),$$
    where $k : X\times \{0,1\} \subset X\times [0,1]  \to \left(Y\sqcup Z\right)$ is
    defined by $k(x,0) = f(x) \in Y$ and $k(x,1) = g(x) \in Z$.
\end{Defn}

\begin{Defn}
    Let $X$ and $Y$ be CW-complexes. Then a continuous map $f : X \to Y$
    is called \emph{cellular} if $f(X^n) \subseteq Y^n$ for all $n$.
\end{Defn}

A cellular map sends $0$-cells to $0$-cells but the same is not necessarily true
for $n$-cells with $n > 0$. For example, take $f : \R \to \R$ given by $f(x) = x^2$,
where $\R$ has the CW-complex structure given in Example~\ref{RasCWcompl}. $f$ is
a cellular map since $X^0 = \Z$, $f(\Z) \subset \Z$ and $f(\R) \subseteq    \R$.
However, if we take $e^1 = (1,2)$, $f(e^1) = (1,4)$, and $(1,4)$ is not a $1$-cell
but the union of two $1$-cells and one $0$-cell.

Note that the maps $g$ and $h$ in Definition~\ref{defconcyl} are cellular if $f$
is, given that $\{0\}$ and $\{1\}$ are the $0$-cells of $[0,1]$. Analogously,
the map $k$ in Definition~\ref{defdoublecyl} is cellular if $f$ and $g$ are.

The following theorem gives us some conditions under which an adjunction space
is a CW-complex:

\begin{Theorem}\cite[Theorem 2.3.1.]{fritsch_cellularstr}\cite[Lemma 3.10]{lueck_alg_top}\label{pushoutisCW}
    Let $X$ and $Y$ be CW-complexes and $A \subseteq X$ a subcomplex of $X$. Let
    $f : A \to Y$ be a cellular map. Then, if we take $Z$ to be the topological
    push-out of the diagram formed by $f$ and $\iota : A \hookrightarrow X$, $Z$
    is a CW-complex.

    Moreover, if $\bar{\iota}$ and $\bar{f}$ are the maps that complete the push-out diagram
    and $c(X), c(Y), c(A), c(Z)$ are the sets of open cells of $X, Y, A$ and $Z$ respectively,
    the $n$-skeleton $Z^n = \bar{f}(A^n)\cup \bar{\iota}(Y^n)$ and $c(Z) = c(Y) \sqcup (c(X) \bs c(A))$.
\end{Theorem}

\begin{Corollary}\label{cone_mcyl_dcyl_CW}
    Let $X$, $Y$ and $Z$ be CW-complexes and let $f : X \to Y$ and $g : X \to Z$ be
    cellular maps. Then the mapping cone $\mcone{f}$, the mapping cylinder $\mcyl{f}$ and the double
    mapping cylinder $\dmcyl{f}{g}$ are CW-complexes with $X$ and $Y$ (and $Z$ in the case of $\dmcyl{f}{g}$)
    as subcomplexes.
\end{Corollary}
\begin{proof}
    Consequence of Proposition~\ref{QovermapCW} and Theorem~\ref{pushoutisCW}.
\end{proof}

The condition of $f$ being a cellular map is not as big a restriction as it may seem, given
the following result, that can be found for example in \cite[Theorem 8.5.4.]{tDieck_algebraic_top},
\cite[Theorem 4.8]{Hatcher_algebraictop} and \cite[Theorem 2.4.11]{fritsch_cellularstr}:

\begin{Theorem}[Cellular Approximation Theorem]\label{CellApprox}
    Every continuous map $f : X \to Y$ between CW-complexes is homotopic to a cellular
    map ${g : X \to Y}$. If $f$ is already cellular on a subcomplex $A \subset X$, the
    homotopy may be taken to be stationary on $A$.
\end{Theorem}

In \cite[Theorem 2.1]{luecktransformation}, W. L\"uck provides a version of the theorem for $G$-CW-complexes (which we will talk about in the next section).

\section{$G$-CW-Complexes}

\begin{Defn}\label{definition_GCW} Let $G$ be a discrete group and $X$ a topological space such that $G$ acts continuously on $X$. A $G$-CW-complex structure on $X$ consists of
\begin{enumerate}[label={(\roman*)},noitemsep]
\item a filtration $\emptyset = X_{-1} \subseteq X^0 \subseteq X^1 \subseteq \ldots \subseteq X^n \subseteq \ldots \subseteq X$ such that $\bigcup_{n \geq 0} X^n = X$;
\item a collection $\{ e^n_\alpha \st \alpha \in \mathcal{A}^n\}$ of $G$-subspaces $e^n_\alpha \subseteq X^n$ for each $n\in\N$
\end{enumerate}
such that

\begin{enumerate}[label={(\alph*)},noitemsep]
    \item $X$ has the weak topology with respect to the filtration $\{X^n\}_{n\in\N}$
    \item for each $n\geq 0$ $X^n$ can be obtained by attaching the $G$-subspaces $e^n_\alpha$ to $X^{n-1}$ via continuous $G$-maps $q^n_\alpha : G/H_\alpha \times S^{n-1} \to X^{n-1}$, where $H_\alpha$ are subgroups of $G$. That is, $X^n$ is the push-out of the following diagram:

    $$\begin{tikzcd}
    \underset{\alpha\in \mathcal{A}^n}{\bigsqcup} G/H_\alpha \times S^{n-1}
            \arrow[rr, "\underset{\alpha\in \mathcal{A}^n}{\bigsqcup} q^n_\alpha"]
            \arrow[hookrightarrow, dd, "\iota"]
    & & X^{n-1}
            \arrow[dotted, dd] \\
            \\
    \underset{\alpha\in \mathcal{A}^n}{\bigsqcup} G/H_\alpha \times D^n
            \arrow[dotted, rr, "\underset{\alpha\in \mathcal{A}^n}{\bigsqcup} Q^n_\alpha"]
    & & X^n
    \end{tikzcd}
    $$

\end{enumerate}
\end{Defn}

In this case the $n$-skeleton is $X^n$ again and $e^n_\alpha$ are the (open) equivariant $n$-cells and $\bar{e}^n_\alpha$ are their topological closure (we may refer to them as closed equivariant $n$-cells).

We may assume that $X$ is a Hausdorff space. In fact, in some of their first appearances, $G$-CW-complexes were defined to be Hausdorff spaces (\cite{matumoto_GCW}, \cite{illman_GCW}), as the original CW-complexes were all normal (and hence Hausdorff). Also, in \cite{tDieck_transformation} it is shown that if $X^n$ is obtained from $X^{n-1}$ as in the push-out above and $X^{n-1}$ is Hausdorff, then $X^n$ is also Hausdorff.

$G$-CW-complexes can be defined more generally for topological groups (see \cite{luecksurvey}, for example).

As for any $G$-space, the \emph{isotropy groups} of a $G$-CW-complex $X$ play an essential role when studying the relation between $G$ and $X$.
If we take the isotropy group of $x\in X$, $G_x = \{g \in G \st gx = x\}$, we can see that it is nothing than the preimage of $\{x\}$ by the action of $G$ on $X$, which is a continuous map. In the case that we focus our interest, groups are discrete, and so equipped with the discrete topology. In that case, of course, isotropy groups are open and closed.

\begin{Prop}\cite[Remark 1.3]{luecksurvey}\label{GCW-open-isotropy}
Let $X$ be a $G$-space with $G$-invariant filtration
$$\emptyset = X_{-1} \subseteq X^0 \subseteq X^1 \subseteq \ldots \subseteq X^n \subseteq \ldots \subseteq X = \bigcup_{n\geq 0} X^n.$$
Then the following assertions are equivalent:
\begin{enumerate}[label= {\emph{\roman*)}}, noitemsep]
\item Every isotropy group of $X$ is open and the filtration above yields a $G$-CW-structure on $X$.
\item The filtration above yields a CW-structure on $X$ such that for each open cell $e\subseteq X$ and each $g\in G$, if $ge \cap e \neq \emptyset$ then $g$ fixes $e$ point-wise.
\end{enumerate}
\end{Prop}

The case of the proposition above holds for discrete groups, but it is not the general case if one takes in consideration any topological group. In Section 2 of \cite{illman_rest_transf} there is an example of non-discrete equivariant (for the circle group $S^1 = \{ z \in \C \st |z| = 1\}$) CW-complex $X$ such that $X$ does not admit a CW structure compatible with the equivariant CW structure given.

\begin{Defn}
Let $G$ be a topological group acting on a CW-complex $X$. We say that the $G$-action on $X$ is \emph{cellular} (or that $G$ acts \emph{cellularly} on $X$) iff
\begin{enumerate}[label={(\roman*)},noitemsep]
\item if $e$ is an $n$-cell of $X$ and $g\in G$, then $g e$ is also an $n$-cell of $X$ and
\item if $e$ is a cell of $X$ and $g\in G$ is such that $ge \cap e \neq \emptyset$ then $gp = p$ for all $p \in e$.
\end{enumerate}
\end{Defn}

As a particular case of Proposition \ref{GCW-open-isotropy}, we have the following characterization of equivariant CW-complexes for discrete groups:

\begin{Corollary}\label{G-CW-comp-discrete}
Let $G$ be a discrete group and $X$ a topological space. Then $X$ is a $G$-CW-complex if and only if $X$ admits a CW-structure and $G$ acts on $X$ cellularly.
\end{Corollary}

\begin{Example}[continues = RasCWcompl] Let $G = \Z = \langle t \rangle$, then $\R$ is a $G$-CW-complex. The $G$-action would be defined by $t^n x = x + n$, where $n\in \Z$ and $x\in \R$, i.e. $G = \Z$ acts on $\R$ by translation of $1$ unit in the positive direction.
\
According to Corollary~\ref{G-CW-comp-discrete}, we only need to check that the
action is cellular. $0$-cells are points $m \in \Z$, and clearly $G$ sends points
in $\Z$ to $\Z$. $1$-cells are intervals $(m,m+1)$ where $m\in \Z$. Since all elements
of $g$ would act by addition of an integer, $1$-cells would also go to $1$-cells. It only
remains to check that if $g\in G$ and a $1$-cell $(m,m+1)$ are such that $g x \in (m,m+1)$
for some $x\in (m,m+1)$, then $gy = y$ for all $y \in (m,m+1)$. But that is clear given
the definition of the action.

More generally, $\R^n$ is a $\Z^k$-CW-complex for any $n,k \in \N$, where each generator of $\Z^k$ acts either trivially on $\R^n$ or by translation by any integer on one of the components of the points of $\R^n$.
\end{Example}

\begin{Example}\label{RasDinfCWcomplex} Similarly, $\R$ is also a $D_\infty$-CW-complex, where $D_\infty = \langle a , b \st ab = ba, b^2 = 1 \rangle$ is the infinite dihedral group.

In this case, $a$ acts by translation of $2$ units in the positive direction and $b$ by reflection with respect to $0$.

It is necessary for the translation induced by $a$ to be of $2$ units since the element $ab$ fixes the midpoint between $0$ and $a0$, which would need to be a $0$-cell itself and not belong to the interior of a $1$-cell. Alternatively, we could define the $0$-skeleton of $\R$ to also include the points of the form $\frac{m}{2}$ for $m\in\Z$ and the $1$-skeleton to be the segments between consecutive $0$-cells.
\end{Example}

\begin{Defn} Given a group $G$ and a $G$-space $X$ (with left-action), the \emph{quotient of $X$ by the $G$-action} is
$$ G \bs X = X / \sim$$
where $x \sim y$ if and only if there is $g\in G$ such that $x = gy$.
\end{Defn}

Note that $G\bs X$ is itself a $G$-space where $G$ acts trivially. Since $G$ and $X$ have a topology and the $G$-action on $X$ is continuous, $G \bs X$ is a topological space with the quotient topology.

\begin{Defn}
A $G$-space $X$ is \emph{cocompact} if $G\bs X$ is compact.
\end{Defn}

\begin{Defn}\label{def_finite_type}
    A $G$-CW-complex is said to be of \emph{finite type} if the indexing sets $\mathcal{A}^n$ in
    Definition~\ref{definition_GCW} are all finite, i.e., if there are finitely many $n$-cells
    for all $n\in\N$.
\end{Defn}

\begin{Remark}
    A $G$-CW-complex is of finite type if and only if it has only finitely many $G$-orbits in each dimension.
\end{Remark}

\begin{Remark}
A $G$-CW-complex $X$ is cocompact if and only if it is of finite type and finite-dimensional.
\end{Remark}

\begin{Example}[continues = TasCWcompl]
Consider $G = \Z^2 = \langle s,t \rangle$. As seen in the continuation of Example~\ref{RasCWcompl}, $\R^2$ is a $\Z^2$-CW-complex. We can take as action the one generated by $s (0,0) = (1,0)$ and $t(0,0) = (0,1)$.

Note that in this case $\Z^2 \bs \R^2 = \T$. And so, $\T$ is a $\Z^2$-CW-complex on which $\Z^2$ acts trivially.
\end{Example}

\subsection{Operations with $G$-CW-complexes}

In Section~\ref{sect_cwcomp} we showed constructions with CW-complexes that result
in CW-complexes that we will use throughout. Let us extend some of those results to
$G$-CW-complexes for discrete groups, using Corollary~\ref{G-CW-comp-discrete}. In the
case of the quotient and product of $G$-CW-complexes, the properties $(F)$ and $(H)$
in \cite{matumoto_GCW}, respectively, provide more general results than the ones in this
section, but we are only interested in the case where $G$ is a discrete group. In the case
of the join, mapping cylinder and double mapping cylinder, similar results can be found
in \cite{luecktransformation}. For this reasons, one should read the proofs we provide
as a way of obtaining useful information about the $G$-actions for the discrete case, since the results
were proved in previously cited sources.

\begin{Corollary}
    Given a $G$-CW-complex $X$ and $\pi : X \twoheadrightarrow L$ a surjective map
    as in Proposition~\ref{QovermapCW}. Assume moreover that $gx = gy$ for all $g\in G$
    and $x,y\in X$ such that $\pi(x) = \pi(y)$ and that $\pi$ is cellular. Then $L$ is a $G$-CW-complex.
\end{Corollary}
\begin{proof}
    Given $g\in G$, $g$ acts on $l\in L$ by $l \mapsto \pi(gx)$, where $\pi(x) = l$.
    This action is cellular as both $\pi$ and the action of $G$ on $X$ are cellular.
\end{proof}

\begin{Corollary}\label{productGCW}
    Let $X$ and $Y$ be $G$-CW-complexes. Then, under any of the conditions stated in
    Theorem~\ref{productCW}, $X \times Y$ is a $G$-CW-complex.
\end{Corollary}
\begin{proof}
    By Theorem~\ref{productCW}, $X \times Y$ with the product cellular structure
    is a $CW$-complex.

    Define the action of $g\in G$ on $X\times Y$ by $(x,y) \mapsto (gx,gy)$. We need to
    see that this action is cellular:
    \begin{enumerate}[label={(\roman*)},noitemsep]
        \item If $e$ is an $n$-cell of $X\times Y$ and $g\in G$, then $g e$ is also
        an $n$-cell of $X\times Y$:\\
        By definition of the product cellular structure on $X\times Y$,
        there is $e_1$ a $k$-cell of $X$ and there is $e_2$ a $(n-k)$-cell of $Y$ such
        that $e = e_1 \times e_2$. Also, $g e = (g e_1) \times (g e_2)$, and since
        the $G$-actions on $X$ and $Y$ are cellular, we have $g e_1$ is a $k$-cell
        of $X$ and $g e_2$ is an $(n-k)$-cell of $Y$. Hence, $g e$ is the $n$-cell
        of $X\times Y$.
        \item If $e$ is a cell of $X\times Y$ and $g\in G$ is such that $ge \cap e \neq \emptyset$
        then $gp = p$ for all $p \in e$:\\
        Let $e_1$ and $e_2$ as above. Then if $ge \cap e \neq \emptyset$ then we have
        $g e_1 \cap e_1 \neq \emptyset$ and $g e_2 \cap e_2 \neq \emptyset$. And as
        the $G$-actions on $X$ and $Y$ are both cellular, that means that $g p_1 = p_1$
        for all $p_1 \in e_1$ and $g p_2 = p_2$ for all $p_2 \in e_2$. Therefore,
        for every $p\in e$, $gp = p$.
    \end{enumerate}

    By Corollay~\ref{G-CW-comp-discrete}, we are done, as $G$ is discrete.
\end{proof}

\begin{Corollary}\label{joinGCW}
    Let $X$ and $Y$ be $G$-CW-complexes. Then, under any of the conditions stated in
    Theorem~\ref{productCW}, $X * Y$ is a $G$-CW-complex.
\end{Corollary}
\begin{proof}
    $X\times Y$ is a $G$-CW-complex by Corollary~\ref{productGCW}. Hence,
    $X \times Y \times [0,1]$ is also a $G$-CW-complex by the same result, taking
    the trivial $G$-action on $[0,1]$.
    Let $\pi : X \times Y \times [0,1] \to X * Y$ be the quotient map. Note that
    the restriction $\pi$ to $X \times Y \times (0,1)$ is injective. Note that
    the cells of $X * Y$ are of the form $\pi(e_1 \times e_2 \times (0,1))$,
    $\pi(e_1 \times e_2 \times {0}) = e_1$ or $\pi(e_1 \times e_2 \times {1}) = e_2$
    for $e_1$ and $e_2$ cell of $X$ and $Y$ respectively.

    Let $g\in G$, then we define the action of $g$ on $X * Y$ by $g\pi(x,y,t) = \pi(g(x,y,t)) = \pi(gx,gy,t)$.
    We need to see now that this $G$-action is cellular:
    \begin{enumerate}[label={(\roman*)},noitemsep]
        \item If $e$ is an $n$-cell of $X * Y$ and $g\in G$, then $g e$ is also
        an $n$-cell of $X * Y$:\\
        If $e$ is of the form $e_i$ for $i \in \{1 , 2\}$, then since $X$ and $Y$
        are $G$-CW-complexes, $g e$ is a cell of $X * Y$ of the same form. If $e$
        is of the form $\pi(e_1 \times e_2 \times (0,1))$, then
        $g e = \pi(g e_1 \times g e_2 \times (0,1))$ is also a cell of $X * Y$, as
        $g e_1$, $g e_2$ and $(0,1)$ are cells of $X$, $Y$ and $[0,1]$ respectively.
        \item If $e$ is a cell of $X * Y$ and $g\in G$ is such that $ge \cap e \neq \emptyset$
        then $gp = p$ for all $p \in e$:\\
        If $e = e_i$ for $i \in \{1 , 2\}$, then we are done since
        $G$ acts cellularly on $X$ and $Y$. If $e$ is of the form
        $\pi(e_1 \times e_2 \times (0,1))$ and $g e\cap e \neq \emptyset$, then we
        have $g (e_1 \times e_2 \times (0,1)) \cap e_1 \times e_2 \times (0,1) \neq \emptyset$.
        Let $p \in \cap e$, since $\pi$ is injective in $X \times Y \times (0,1)$,
        $\pi^{-1}(p) \in e_1 \times e_2 \times (0,1)$
        is a single point. Since $G$ acts cellularly in $X\times Y\times [0,1]$,
        then $g\pi^{-1}(p) = \pi^{-1}(p)$. Hence, $g p = p$, as we wanted to see.
    \end{enumerate}

    By Corollay~\ref{G-CW-comp-discrete}, we are done, as $G$ is discrete.
\end{proof}

\begin{Corollary}\label{mcylGCW}
    Let $X$ and $Y$ be $G$-CW-complexes and let $f : X \to Y$ be a $G$-map.
    Then, $\mcyl{f}$ is a $G$-CW-complex.
\end{Corollary}

\begin{proof}

    Let $G$ act trivially on $[0,1]$. Then, by Corollary~\ref{productGCW}, $X\times [0,1]$ is a $G$-CW-complex, and so
    is $\left(X\times [0,1]\right) \sqcup Y$. Let $\pi : \left(X \times [0,1]\right) \sqcup Y \to \mcyl{f}$ be the quotient map. Then, given $g\in G$
    we can define the action of $g$ on $\mcyl{f}$ as $g \pi(x,t) = (gx, t)$ for $x \in X$ and $t \in [0,1]$ and
    $g \pi(y) = gy$ for $y \in Y$. Since $f$ is a $G$-map, the action is well-defined, i.e., for all $x \in X$ $g \pi(x, 1)$ and $g \pi(f(x))$
    correspond to the same point in $\mcyl{f}$. By definition, $\pi$ is bijective when restricted to $X\times [0,1)$ and when restricted to $Y$.

    Hence, the cells of $\mcyl{f}$ are of the form $\pi(e_1 \times \{0\})$, $\pi(e_1 \times (0, 1))$
    or $\pi(e_2)$ for $e_1$ and $e_2$ cells of $X$ and $Y$ respectively.

    \begin{enumerate}[label={(\roman*)},noitemsep]
        \item If $e$ is an $n$-cell of $\mcyl{f}$ and $g\in G$, then $g e$ is also
        an $n$-cell of $\mcyl{f}$:\\
        Let $e$ be an $n$-cell of $\mcyl{f}$ of the form $\pi(e_1 \times A)$ where $e_1$ is a $k$-cell of $X$ and $A$ is either $\{0\}$ or $(0,1)$.
        Since $\pi$ is bijective when restricted to $X\times [0,1)$, we have $k + dim(A) = n$. By definition of the $G$-action
        on $\mcyl{f}$, $g e = \pi((g e_1)\times A)$. Since $X$ is a $G$-CW-complex, if $e_1$ is a $k$-cell of $X$, so is $g e_1$.
        And again as $\pi$ is bijective when restricted to $X\times [0, 1)$, that means $g e$ is a $k + dim(A)$-cell of $\mcyl{f}$.
        In the case $e$ is an $n$-cell of $\mcyl{f}$ of the form $\pi(e_2)$ for $e_2$ and $n$-cell of $Y$, as $g \pi(e_2) = \pi(e_2)$
        and $Y$ is a $G$-CW-complex, $g \pi(e_2)$ is an $n$-cell of $\mcyl{f}$.
        \item If $e$ is a cell of $\mcyl{f}$ and $g\in G$ is such that $ge \cap e \neq \emptyset$
        then $gp = p$ for all $p \in e$:\\
        Let again $e$ be of the form $\pi(e_1 \times A)$ where $e_1$ is a cell of $X$ and $A$
        is either $\{0\}$ or $(0,1)$. Then, since $\pi$ is bijective when restricted to $X\times [0,1)$
        and by definition of the action of $G$ on $\mcyl{f}$,
        $ge \cap e \neq \emptyset$ if and only if $g (e_1 \times A) \cap e_1 \times A \neq \emptyset$.
        But $G$ is acting trivially on $[0,1]$, so the second condition is equivalent to
        $g e_1 \cap e_1 \neq \emptyset$. And since $X$ is a $G$-CW-complex, in that case
        for all $p \in e_1$, $gp = p$, which by analogous reasoning is equivalent to $gq = q$ for
        all $q \in e$.
        In the case $e$ is of the form $\pi(e_2)$, the proof is analogous, taking into consideration
        that $f$ is a $G$-map (and hence $f(gx) = gf(x)$) for the points in $\mcyl{f}$ of the form $\pi(f(x))$ for $x\in X$.
    \end{enumerate}

    By Corollay~\ref{G-CW-comp-discrete}, we are done, as $G$ is discrete.
\end{proof}

\begin{Corollary}\label{dmcylGCW}
    Let $X$, $Y$ and $Z$ be $G$-CW-complexes and let $f : X \to Y$ and $g : X \to Z$
    be $G$-maps.
    Then, $\dmcyl{f}{g}$ is a $G$-CW-complex.
\end{Corollary}
\begin{proof}
    It is only necessary to apply Corollary~\ref{mcylGCW} to each of the maps and
    identify the two copies of $X\times \{0\}$.
\end{proof}

\section{Families of subgroups}

Given two $G$-spaces $X$ and $Y$, and a $G$-map $g : X \to Y$, we denote by $[g]$ the equivalence class of all
$G$-maps from $X$ to $Y$ that are homotopic to $g$ and we denote by $[X,Y]^G$ the set of $G$-homotopy classes
of $G$-maps from $X$ to $Y$.

The following theorem is stated and proved for topological
groups in \cite{luecksurvey}, but we include a reduced version restricted to discrete groups.
With it, families of subgroups are introduced from their relation to Homotopy Theory.

\begin{Theorem}[Whitehead Theorem for Families]\cite[Theorem 1.6]{luecksurvey}\label{WTFF}
Let $f : Y \to Z$ be a $G$-map of $G$-spaces for $G$ a discrete group. Let $\FF$ be a set of subgroups of $G$ which is closed under conjugation. Then the following assertions are equivalent:
\begin{enumerate}[label= {\emph{\roman*)}}, noitemsep]
\item for any $G$-CW-complex $X$, whose isotropy groups belong to $\FF$, the map induced by $f$ $$f_* : [X , Y]^G \to [X , Z]^G , \hspace{0.5cm} [g] \mapsto [f \circ g]$$
between the set of $G$-homotopy classes of $G$-maps is bijective;
\item for any $H\in\FF$ the map $f^H : Y^H \to Z^H$ is a weak homotopy equivalence, where $A^H$ represents the subset of fixed points by $H$ of a $G$-space $A$.
\end{enumerate}
\end{Theorem}

\begin{Defn}
    Let $G$ be a group. A non-empty collection $\FF$ of subgroups of $G$ is a \emph{family} if it is closed under conjugation.
    If it is also closed under finite intersections, we say $\FF$ is a \emph{semi-full} family. In the case it is closed under
    taking subgroups, we call $\FF$ a \emph{full} family.
\end{Defn}

All results produced in this thesis will refer to full families. Therefore, if not indicated otherwise, we will assume that families are full families.

\begin{Example} The following are full families of subgroups:
\begin{enumerate}[label= {\arabic*)}, noitemsep]
\item $\FF = \{ \{ 1\} \}$, the \emph{trivial} family;
\item $\All$, the family of all subgroups of $G$;
\item $\Fin$, the family of all finite subgroups of $G$;
\item $\Vcy$, the family of all virtually cyclic subgroups of $G$;
\item $\PP$, the family of all $p$-subgroups of $G$;
\end{enumerate}
\end{Example}

\begin{Example}\label{bounded_rank_subgroups_of_a_class_is_family}
    Let $\mathfrak{X}$ be a non-empty class of groups closed under taking subgroups and let
    $r: \mathfrak{X} \to \N \cup \{\infty\}$ be a rank such that:
    \begin{enumerate}[label={(\roman*)}, noitemsep]
        \item if $H,K \in \mathfrak{X}$ are such that there is an injective homomorphism
        $f : H \to K$, then $r(H) \leq r(K)$ and
        \item if $H,K \in \mathfrak{X}$ are such that $H \cong K$, then $r(H) = r(K)$.
    \end{enumerate}
    Then, given a group $G$ and $n \in \N \cup \{\infty\}$, $$\mathfrak{X}_n (G) =
    \{ H \leq G \st H \in \mathfrak{X} \text{ and } r(H) \leq n \}$$ is a full
    family of subgroups of $G$.
\end{Example}

Here we expose some ways to obtain new families from given families and subgroups:

\begin{Remark}\label{thingswithfam} Let $\FF$ and $\GG$ be two families of subgroups of a group $G$ and let $K\leq G$ and $N \lhd G$. Then, the following holds:
\begin{enumerate}[label={\emph{\arabic*)}},noitemsep]
\item $\FF \cap \GG$ is a family of subgroups of $G$;
\item $\FF \cup \GG$ is a family of subgroups of $G$;
\item the restriction of $\FF$ to $K$, $\FF \cap K = \{ H \cap K \st H \in \FF \}$, is a family of subgroups of $K$;
\item $\FF/N = \{ HN/N \leq G / N \st H \in \FF \}$ is a family of subgroups of $G/N$.
\end{enumerate}
\end{Remark}

\begin{Remark}
    If the families $\FF$ and $\GG$ in Remark~\ref{thingswithfam} are
    full, so are $\FF \cap \GG$, $\FF \cup \GG$, $\FF \cap K$ and $\FF/N$.

    If the families $\FF$ and $\GG$ are semi-full, so are $\FF \cap \GG$, $\FF \cap K$ and $\FF/N$.
    However, $\FF \cup \GG$ is not semi-full in general, given that the subgroups of the form $H\cap K$,
    where $H\in\FF$ and $K\in\GG$, do not necessarily belong to any of the two original families.
\end{Remark}

\section{Classifying spaces}

In this section, we present some basic results about classifying spaces for families of subgroups. More information can be
found in \cite{luecksurvey}, for example.

\begin{Defn} Let $G$ be a topological group and let $\FF$ be a semi-full family of subgroups of $G$. A $G$-CW-complex $X$ is a \emph{classifying space of $G$ for the family $\FF$} if it satisfies the following conditions:
\begin{enumerate}[label={(\roman*)},noitemsep]
\item All isotropy groups of $X$ belong to $\FF$.
\item\label{universal_prop} If $Y$ is a $G$-CW-complex with isotropy groups in $\FF$, then there exists $G$-map $f : Y \to X$, unique up to $G$-homotopy.
\end{enumerate}
\end{Defn}

Equivalently, we may say $X$ is a \emph{model for} $\EFG{\FF}{G}$ when $X$ is a classifying space for the family
$\FF$ of subgroups of $G$.

We will refer to condition \ref{universal_prop} as \emph{universal property of classifying spaces}, and it is equivalent to
$X$ being a terminal object in the $G$-homotopy category of $G$-CW-complexes with isotropy groups in the family $\FF$.

One of the first questions that arises is whether the existence of such spaces is conditional or universal. The following theorem shows that their existence is universal:

\begin{Theorem}[Existence of models for $\EFG{\FF}{G}$]\cite[Proposition 2.3]{luecktransformation}\label{EFG_existence}
Let $G$ be a topological group and $\FF$ a semi-full family of (closed) subgroups of $G$. Then, there is a model for $\EFG{\FF}{G}$.
\end{Theorem}

And as a consequence of Whitehead Theorem for Families (\ref{WTFF}), we can characterize classifying spaces homotopically as follows:

\begin{Theorem}[Homotopy characterization of $\EFG{\FF}{G}$]\cite[Theorem 1.9]{luecksurvey}
Let $\FF$ be a semi-full family of subgroups of a topological group $G$. Then, a $G$-CW-complex $X$ is a model for $\EFG{\FF}{G}$ if and only if all its isotropy groups belong to $\FF$ and for each $H\in\FF$ the set $X^H$ of points fixed by $H$ is weakly contractible.
\end{Theorem}

In the case of a discrete group $G$, the condition of $X^H$ being weakly contractible can be substituted by $X^H$ being contractible, as weak homotopy equivalences between CW-complexes are homotopy equivalences (\cite[Theorem 3.5]{whitehead_elements_hpy}).

\begin{Prop}
Let $\FF$ be a semi-full family of subgroups of a discrete group $G$. Then, a $G$-CW-complex $X$ is a model for $\EFG{\FF}{G}$ if and only if all its isotropy groups belong to $\FF$ and for each $H\in\FF$ the set $X^H$ is contractible.
\end{Prop}

Finally, when considering full families of subgroups of discrete groups, we can conclude the following characterization of classifying spaces:

\begin{Corollary}\cite[Corollary 2.5]{fluchthesis}\label{other_def_of_class_space}
Let $\FF$ be a full family of subgroups of a discrete group $G$. Then, a $G$-CW-complex $X$ is a model for $\EFG{\FF}{G}$ if and only if for every $H\leq G$ we have
\begin{enumerate}[label={\emph{(\roman*)}},noitemsep]
\item $X^H = \emptyset$ if $H \notin \FF$;
\item $X^H$ is contractible if $H\in\FF$.
\end{enumerate}
\end{Corollary}

\begin{Example}
Let $G$ be any discrete group and $\All$ the family of all subgroups of $G$, then $G/G = \{ * \}$ is a model for $\EFG{\All}{G}$. Moreover, $\EFG{\FF}{G}$ admits a $0$-dimensional model if and only if $G \in \FF$ (\cite[Proposition 3.19]{fluchthesis}).
\end{Example}

\begin{Example}[continues = RasCWcompl]\label{RasEGforZ}
$\R^n$ is a classifying space for the family of finite subgroups $\Fin$ of $\Z^n$ (which in this case coincides with the trivial family).
Given a set $\{g_1, \ldots, g_n\}$ of generators of $\Z^n$ and a base $\{e_1, \ldots, e_n\}$ of $\R^n$ as a $\R$-vector space,
define the action of $\Z^n$ on $\R^n$ as follows:
For $i\in \{1, \ldots, n\}$, $g_i$ acts on $\R^n$ by translation by the vector $e_i$.

Then, it is clear that only the subgroup $\{1\}$ fixes points (all $\R^n$, which is contractible).
\end{Example}

\begin{Example}[continues = RasDinfCWcomplex]
$\R$ is a model for $\EFG{\Fin}{D_\infty}$.

First, note that a subgroup $H$ of $D_\infty$ is finite if and only if there is $i \in \Z$ such that $H = H_i = \langle ba^i \rangle$. It is easy to see that $\R^{H_i} = \{-i\}$. Moreover, if $K \leq D_\infty$ is not finite, then it contains an element of the form $a^j$, which doesn't fix any element in $\R$. Hence, if $K \notin \Fin$, $\R^K = \emptyset$, as we needed to see.
\end{Example}

We will visit examples of classifying spaces for other families than $\Fin$ and the trivial family in Section~\ref{sec:LW} of Chapter 3.

%% file: BredonCohomology.tex
\chapter{Bredon Cohomology}\label{ch:bredon_cohom}
\minitoc

\section{Bredon Modules}
Let $G$ be a group. If $H$ is a subgroup of $G$, then $G/H$ is a $G$-space. Moreover, the action of $G$
on $G/H$ is transitive, so $G/H$ is a homogeneous $G$-space.

Let $H,K \leq G$ and consider $G/H$ and $G/K$ as $G$-spaces. Then, we denote the set of all $G$-maps from $G/H$ to $G/K$ as $[G/H, G/K]_G$.

Given $f \in [G/H, G/K]_G$, since $f(gH)=gf(H)$, $f$ is fully characterised by $f(H)$. Assume $g\in G$ is such that $f(H) = gK$. Then, given $h\in H$, since $hH = H$, $hgK = gK$. That means $g^{-1}Hg \leq K$, so $gK \in (G/K)^H$.
In addition, given $gK \in (G/K)^H$, we can define a $G$-map $f_g : G/H \to G/K$ by $f_g(xH) = xgK$ and $f_g$ is the unique $G$-map such that the image of the coset $H$ is $gK$.

Then, $gK\mapsto f_g$ is a bijection between $(G/K)^H$ and $[G/H, G/K]_G$.

\begin{Defn}
Let $G$ be a group and $\FF$ a family of subgroups of $G$. The \emph{orbit category} $\OFG{\FF}{G}$ is the small category whose objects are homogeneous $G$-spaces $G/H$ for $H\in\FF$ and whose morphisms are $G$-maps between such $G$-spaces.
\end{Defn}

\begin{Defn}\label{bredonmodule} A \emph{Bredon module over the orbit category} $\OFG{\FF}{G}$ is a functor $M : \OFG{\FF}{G} \to \Ab$ where $\Ab$ is the category of abelian groups.

In the case that $M$ is a contravariant functor, we will say that $M$ is a \emph{right $\OFG{\FF}{G}$-module} and if it is a covariant functor, we will say that $M$ is a \emph{left $\OFG{\FF}{G}$-module}.

If $M$ and $N$ are contravariant $\OFG{\FF}{G}$-modules, a \emph{morphism} $\Phi : M \to N$ is a natural transformation from the functor $M$ to the functor $N$. That is, $\Phi$ is given by a family of homomorphisms of abelian groups $\Phi(G/H) : M(G/H) \to N(G/H)$ such that for every $f \in [G/H,G/K]_G$ the following diagram commutes:

$$\begin{tikzcd}
M(G/H)
        \arrow[rr, "\Phi(G/H)"]
& & N(G/H)
        \\
        \\
M(G/K)
        \arrow[rr, "\Phi(G/K)"]
        \arrow[uu, "M(f)"]
& & N(G/K).
        \arrow[uu, "N(f)"]
\end{tikzcd}
$$
\\

In the case $M$ and $N$ are covariant $\OFG{\FF}{G}$-modules, morphisms are defined in the analogous way, taking into account that the vertical arrows in the diagram have to be reversed.
\end{Defn}

\begin{Example}
The trivial $\OFG{\FF}{G}$-module $\Ztriv{\FF}$ is defined by the functor that associates any element of the orbit category with $\Z$
and any morphism between elements of the orbit category with the identity homomorphism in $\Z$. That is, if $G/H, G/K \in \OFG{\FF}{G}$
and $f : G/H \to G/K$ is a $G$-map, then $\Ztriv{\FF}(G/H) = \Z$ and $\Ztriv{\FF}(f): \Z \to \Z$ is defined by $\Ztriv{\FF}(f)(n) = n$.

More generally, given an abelian group $A$, we define a constant $\OFG{\FF}{G}$ module $\underline{A}$ as $\underline{A}(G/H) = A$ for each object $G/H$ of $\OFG{\FF}{G}$ and $\underline{A}(f)$ is the identity homomorphism in $A$ for every morphism $f$ of $\OFG{\FF}{G}$.
\end{Example}

\begin{Example}\label{ZfirstEx} Given $K\in\FF$, we define the contravariant (right) Bredon module $\Zfirst{G}{K}$ as follows:
\begin{enumerate}[label={(\roman*)},noitemsep]
\item for $G/H \in \OFG{\FF}{G}$, take $\Zfull{G}{H}{K}$ as the free abelian group with basis $[G/H, G/K]_G$;
\item for $G/H, G/L \in \OFG{\FF}{G}$ and $f\in [G/H, G/L]_G$, $\Z[f, G/K]_G$ sends $g \in [G/L, G/K]_G$ to  $g \circ f \in [G/H, G/K]_G$, and then extend linearly to a homomorphism $\Zfull{G}{L}{K} \to \Zfull{G}{H}{K}$.
\end{enumerate}

Analogously, we can define  the covariant (left) Bredon module $\Zsecond{G}{H}$ for a given $G/H \in \OFG{\FF}{G}$.
\end{Example}

\begin{Defn}
$\RMod{\FF}{G}$ is the category of contravariant Bredon modules over $\OFG{\FF}{G}$ with morphisms as defined in \ref{bredonmodule}.

$\LMod{\FF}{G}$ is the category of covariant Bredon modules over $\OFG{\FF}{G}$ with morphisms as defined in \ref{bredonmodule}.
\end{Defn}

\section{Free and projective Bredon modules}

We will briefly construct the free objects of $\RMod{\FF}{G}$. For a more detailed and rigorous view on this matter, see \cite[Chapter 1. Section 5]{fluchthesis}.

\begin{Defn}
An $\FF$-\emph{set} $\Delta$ is a pair $\Delta = (\Delta, \phi)$ consisting of a set $\Delta$ and a function $\phi : \Delta \to \FF$. We denote $\Delta_H = \phi^{-1}(\{H\})$ the $H$-\emph{component} of $\Delta$ for each $H \in \FF$. A map $f : (\Delta, \phi) \to (\Delta^\prime, \phi^\prime)$ of $\FF$-sets is a map $f$ between the sets $\Delta$ and $\Delta^\prime$ such that the diagram formed by $f$, $\phi$ and $\phi^\prime$ commutes.

The category described by the objects and maps defined above is denoted by $\FF\text{-}\mathfrak{Set}$.
\end{Defn}

A Bredon module $M$ can also be seen as an $\FF$-set, taking $M_H = M(G/H)$. We can consider the forgetful functor
$$U : \RMod{\FF}{G} \to \FF\text{-}\mathfrak{Set}$$
that sends the Bredon module $M$ to its underlying $\FF$-set $UM$, also denoted by $M$.

\begin{Defn}
We say an $\FF$-set $X$ is a subset of a Bredon module $M$ if $X_H \subseteq M_H$ for all $H\in\FF$. The submodule of $M$ \emph{generated by} $X$ is the smallest submodule of $M$ containing the $\FF$-set $X$ and denoted by $\langle X \rangle$.
\end{Defn}

The singleton $\FF$-sets are those with $\Delta_K = \{\delta\}$ for a particular $K\in\FF$ and $\Delta_H = \emptyset$ for all $H\in\FF$ different from $K$. We will denote this singleton by $\Delta_\delta$, and $K = \phi(\delta)$. These $\FF$-sets give rise to the Bredon modules $\Zfirst{G}{K}$:

\begin{Lemma}\cite[Lemma 1.12]{fluchthesis}\label{ZfirstFromSingleton}
Let $K\in\FF$. Then $\Zfirst{G}{K} = \langle \Delta_\delta \rangle$, where $K = \phi(\delta)$.
\end{Lemma}

We can write any $\FF$-set $\Delta$ as the coproduct of the singleton $\FF$-sets of its elements, i.e., $\Delta = \coprod_{\delta\in\Delta} \Delta_\delta$. For this reason, and given Lemma \ref{ZfirstFromSingleton}, we can now define a left adjoint for $U$:

\begin{Prop}\label{forgetful_free_functs}\cite[Proposition 1.13]{fluchthesis}
    The forgetful functor $U$ has
    a left adjoint $F : \FF\text{-}\mathfrak{Set} \to \RMod{\FF}{G}$.
\end{Prop}

It follows that given an $\FF$-set $\Delta$ its image by $F$ is $$F\Delta = \coprod_{\delta\in\Delta} \Zfirst{G}{\phi(\delta)}.$$

\begin{Defn} Given $M\in\RMod{\FF}{G}$, we say that $M$ is \emph{free} if there is an $\FF$-set $\Delta$ such that $M = F\Delta$.
\end{Defn}

\begin{Defn} Given $P\in \RMod{F}{G}$, we say $P$ is a \emph{projective Bredon module} if for every $M,N\in \RMod{F}{G}$ and morphisms $\phi: P \to M$ and $\pi: N \to M$ such that $N \xrightarrow{\pi} M \to 0$ is exact there is a morphism $\psi: P \to N$ such that the following diagram commutes
$$\begin{tikzcd}
& P \arrow[dl, "\psi"', dashrightarrow] \arrow[d, "\phi"] \\
N
\arrow[r, "\pi"]
& M \arrow[r]
& 0
\end{tikzcd}
$$

\end{Defn}

$\RMod{\FF}{G}$ is an abelian category in which kernels and images are calculated component wise. In particular,
a sequence of right $\OFG{\FF}{G}$-modules $L \to M \to N$ is exact at $M$ if and only if the corresponding
sequences of abelian groups $N(G/H) \to M(G/H) \to L(G/H)$ are exact at $M(G/H)$ for all $G/H\in\OFG{\FF}{G}$.
For those reasons, the following characterization of projective Bredon modules, analogous to that for modules
over a ring that can be found for example in \cite[section 2.2]{weibel_homalg} or \cite[section 3.1]{rotman_homalg}, holds true.

\begin{Prop}\label{equiv_cond_proj_mod} Let $P$ be a Bredon module over the orbit category $\OFG{\FF}{G}$. Then the following statements for $P$ are equivalent:
\begin{enumerate}[label={\emph{(\arabic*)}},noitemsep]
\item $P$ is projective;
\item every exact sequence $0 \to M \to N \to P \to 0$ splits;
\item $\mor{\FF}{P}{?}$ is an exact functor;
\item $P$ is a direct summand of a free $\OFG{\FF}{G}$-module.
\end{enumerate}
\end{Prop}

Note that the right Bredon modules $\Zfirst{G}{K}$ are projective Bredon modules,
as they are free $\OFG{\FF}{G}$-modules. This, together with Proposition~\ref{forgetful_free_functs},
leads us to the following crucial result:

\begin{Theorem}\label{RModHasEnoughProj}
$\RMod{\FF}{G}$ has enough projectives, i.e., for every $M\in\RMod{\FF}{G}$ there is a projective $P\in\RMod{\FF}{G}$ and an epimorphism $\Phi : P \to M$.
\end{Theorem}

In \cite[Section 3]{mislin_equivariant}, the author gives a constructive approach to projective Bredon modules that may offer more practical insight on these modules. In particular, reading the proof of Theorem~\ref{RModHasEnoughProj} (\cite[Pg. 9-10]{mislin_equivariant}) may be a good exercise to get familiar with working with Bredon modules.

\begin{Defn}
    Given a right $\OFG{\FF}{G}$-module $M$, a \emph{resolution of} $M$ \emph{in} $\RMod{\FF}{G}$ is
    a long exact sequence $$ \ldots \to N_n \to N_{n-1} \to \ldots \to N_0 \to M \to 0$$
    such that $N_k\in\RMod{\FF}{G}$ for all $k$.

    In the case that $N_k$ are projective $\OFG{\FF}{G}$-modules, we say that the sequence is a \emph{projective resolution}.
\end{Defn}

Note that Theorem~\ref{RModHasEnoughProj} implies that there is a projective resolution of every contravariant $\OFG{\FF}{G}$-module in $\RMod{\FF}{G}$.

\section{Bredon Cohomology}

By Theorem~\ref{RModHasEnoughProj}, every $M\in\RMod{\FF}{G}$ admits a projective resolution $P_*(M) \epim M$. Therefore, for every $M,N\in\RMod{\FF}{G}$ we can define a cochain complex $\mor{\FF}{P_*(M)}{N}$, which allows us to define the derived functors of the morphism functor $\mor{\FF}{?}{??}$.

\begin{Defn} Given $N\in\RMod{\FF}{G}$ and $n\in\N$, we define $\Ext{n}{\FF}{?}{N}$ to be the $n$-th right derived functor of $\mor{\FF}{?}{N}$. That is, for every $M\in\RMod{\FF}{G}$
$$\Ext{n}{\FF}{M}{N} = \Ho_{n}(\mor{\FF}{P_*(M)}{N}).$$
\end{Defn}

Analogously to the case of Proposition~\ref{equiv_cond_proj_mod}, the results \cite[2.2.3]{weibel_homalg} and
\cite[2.5.2]{weibel_homalg} remain true for $\RMod{\FF}{G}$:

\begin{Prop}\label{projective_equivalence}
Let $M\in\RMod{\FF}{G}$. Then the following statements for $M$ are equivalent:
\begin{enumerate}[label={\emph{(\arabic*)}},noitemsep]
\item $M$ is projective;
\item $\mor{\FF}{M}{?}$ is an exact functor;
\item $\Ext{n}{\FF}{M}{N} = 0$ for every $n \geq 1$ and every $N\in\RMod{\FF}{G}$;
\item $\Ext{1}{\FF}{M}{N} = 0$ for every $N\in\RMod{\FF}{G}$.
\end{enumerate}
\end{Prop}

\begin{Defn} Let $M\in\RMod{\FF}{G}$. Then the \emph{Bredon cohomology groups} $\Cohom{n}{\FF}{G}{M}$ \emph{of} $G$ \emph{with coefficients in} $M$ are
$$\Cohom{n}{\FF}{G}{M} = \Ext{n}{\FF}{\Ztriv{\FF}}{M}. $$
\end{Defn}

%% file: BredonDimensions.tex
\chapter{Bredon Dimensions}\label{ch:bredon_dim}
\minitoc

\section{Bredon cohomological dimension}

\begin{Defn}\label{proj_resolution} Let $G$ be a discrete group, $\FF$ a full family of subgroups of $G$ and
    $M \in \RMod{\FF}{G}$. Let $n$ be the smallest natural number such that there
    is a projective resolution of $M$
    $$ 0 \to P_n \to P_{n-1} \to \ldots \to P_0 \to M \to 0$$
    of length $n$. Then, we say that $n$ is the \emph{projective dimension of} $M$ and denote
    it by $\pd_{\FF} M = n$. In the case there is no such $n$, we say $\pd_{\FF} M = \infty$.
\end{Defn}

As in the case of Proposition~\ref{projective_equivalence}, the results \cite[Lemma 4.1.6]{weibel_homalg} and \cite[Proposition 8.6]{rotman_homalg}
hold for $\RMod{\FF}{G}$:

\begin{Prop}\label{pd_zeros}
    Let $M$ be a right $\OFG{\FF}{G}$-module. Then the following statements are equivalent:
    \begin{enumerate}[label={\emph{(\arabic*)}},noitemsep]
        \item $\pd_{\FF} M \leq d$;
        \item $\Ext{n}{\FF}{M}{N} = 0$ for every $N\in\RMod{\FF}{G}$ and every $n > d$ ;
        \item $\Ext{d+1}{\FF}{M}{N} = 0$ for every $N\in\RMod{\FF}{G}$;
        \item given any projective resolution of $M$ $$\ldots \to P_2 \to P_1 \to P_0 \to 0,$$
        the kernel $\Ker{P_d \to P_{d-1}}$ is projective.
    \end{enumerate}
\end{Prop}

\begin{Defn} Let $G$ be a discrete group and $\FF$ a full family of subgroups of $G$. The \emph{Bredon cohomological dimension of} $G$ \emph{with respect to} $\FF$ is the projective dimension of the trivial $\OFG{\FF}{G}$-module $\Ztriv{\FF}$ and we denote it $\cd_{\FF} G$.
\end{Defn}

Note that if the group $G$ belongs to the family $\FF$, then the trivial $\OFG{\FF}{G}$-module $\Ztriv{\FF}$ is free (and hence
projective). That means $\cd_{\FF} G = 0$. The reciprocal is true in the case of $\FF$ being semi-full:

\begin{Proposition}\label{G_in_FF_cd_0}\cite[Proposition 3.20]{fluchthesis}
    Let $G$ be a group and $\FF$ a semi-full family of subgroups of $G$. Then, $\cd_{\FF} G = 0$ if and only if $G\in\FF$.
\end{Proposition}

\begin{Corollary}\label{cd_max_min}
    Given a full family $\FF$ of subgroups of a discrete group $G$, we have
    \begin{enumerate}[label={\emph{(\roman*)}}, noitemsep]
        \item $\cd_{\FF} G = \max\{ d \st \text{ there is } M\in\RMod{\FF}{G} \text{ with }\Cohom{d}{\FF}{G}{M} \neq 0\}$ and
        \item $\cd_{\FF} G = \min\{ d \st \Cohom{d+1}{\FF}{G}{M} = 0 \text{ for all } M\in\RMod{\FF}{G}\}$.
    \end{enumerate}
    Also, $\cd_{\FF} G = \infty$ if and only if the maximum in $(i)$ doesn't exist, which is
    equivalent to the set over which we take the minimum in $(ii)$ being empty.
\end{Corollary}

\begin{Defn}
    Let $G$ be a discrete group and $\FF$ a full family of subgroups of $G$.
    Then, the \emph{Bredon geometric dimension of} $G$ \emph{for the family} $\FF$
    is the smallest possible dimension of a model for $\EFG{\FF}{G}$.
\end{Defn}

Since $\{*\}$ is a model for $\EFG{\FF}{G}$ if $G\in\FF$ and since $\{*\}$ is the only $0$-dimensional
contractible $G$-CW-complex, we have the following result analogous to Proposition~\ref{G_in_FF_cd_0}:

\begin{Proposition}\label{G_in_FF_gd_0}\cite[Proposition 3.19]{fluchthesis}
    Let $G$ be a group and $\FF$ a semi-full family of subgroups of $G$. Then, $\gd_{\FF} G = 0$ if and only if $G\in\FF$.
\end{Proposition}

\section{Bredon cohomology and Classifying spaces}

Given a $G$-CW-complex $X$, we will construct a chain of projective contravariant $\OFG{\FF}{G}$-modules for any family
$\FF$ containing the family of isotropy groups of $X$ as done in \cite{mislin_equivariant}.

\begin{Defn}
    Given $X$ a $G$-CW-complex, we denote by $\FF(X)$ the family of isotropy subgroups of $X$.
\end{Defn}

Let $\Delta_n$ be the $G$-set formed by taking all cosets involved in the cell attachment to construct the $n$-skeleton of $X$
from its $(n-1)$-skeleton. That is, $\Delta_n = \{ G/H_\alpha \st \alpha \in \mathcal{A}^n\}$.

Then we have the cellular chain complex $C_* (X)$ given by $$C_n (X) = \Z [\Delta_n].$$
Given $K \leq G$, if we consider the cellular chain complex defined from the $G$-CW-complex $X^K$, we have
$C_n(X^K) = \Z[\Delta_n^K]$. By definition of $\Delta_n$, and since $(G/H)^K \cong [G/K, G/H]_G$ for any $H \leq G$, we have
$C_n(X^K) \cong \bigoplus_{\alpha \in \mathcal{A}^n} \Zfull{G}{K}{H_\alpha}$.
Given a family $\FF$ such that $\FF(X) \subseteq \FF$, we can define the contravariant $\OFG{\FF}{G}$-module
\begin{align*}
    \underline{C_n (X)} : \OFG{\FF}{G} & \longrightarrow  \Ab \\
    G/K & \longmapsto C_n(X^K),
\end{align*}

We can summarize some of the properties of $\underline{C_* (X)}$ found for example
in \cite{mislin_equivariant}:

\begin{Remark}\label{properties_underline_Cstar}
\begin{enumerate}[label={\emph{(\roman*)}}, noitemsep]
    \item since $\underline{C_n (X)} = \bigoplus_{\alpha\in\mathcal{A}^n} \Zfirst{G}{H_\alpha}$,
    then $\underline{C_n (X)}$ is projective for every $n\geq 0$ and
    \item given $M\in\RMod{\FF}{G}$, we have $$\operatorname{mor}_{\FF} \left( \underline{C_*(X)}, M\right) \cong
    \operatorname{mor}_{\FF(X)} \left( \underline{C_*(X)}, \modres{F} M\right),$$ where
    $F : \FF(X) \to \FF$ is the inclusion functor.\label{Cstar_restricted}
\end{enumerate}
\end{Remark}

\begin{Defn}\label{bredcohom_on_class_spaces}
    Let $X$ be a $G$-CW-complex, $\FF$ a family of subgroups of $G$ such that $\FF(X) \subseteq \FF$ and $M \in \RMod{\FF}{G}$.
    We define the \emph{Bredon cohomology groups of }$X$ \emph{with coefficients in} $M$ as the groups
    $$\Cohom{n}{\FF}{X}{M} = \Ho^n\left(\operatorname{mor}_{\FF} \left( \underline{C_*(X)}, M\right)\right)$$
    for every $n \geq 0$.
\end{Defn}

\begin{Corollary}\cite[Corollary 3.5]{mislin_equivariant}\label{cohom_EFG_Bred_cohom}
    Let $\FF$ be a full family of subgroups of $G$. Let $X$ be a model for $\EFG{\FF}{G}$ and $M \in \RMod{\FF}{G}$. Then,
    $\underline{C_*(X)}$ is a projective resolution of $\Ztriv{\FF}$. In particular, by the definition of
    $\Cohom{n}{\FF}{G}{M}$, we have
    $$\Cohom{n}{\FF}{X}{M} \cong \Cohom{n}{\FF}{G}{M}$$
    for all $n \geq 0$.
\end{Corollary}

The following results relating Bredon cohomology and classifying spaces can give a good
overview of the basic relation between cohomological and geometric Bredon dimensions:

\begin{Proposition}\cite[Theorem 0.1 (a)]{lueckmeintrup}
    Let $G$ be a discrete group, let $\FF$ be a semi-full family of subgroups of
    $G$ and let $n \geq 3$. Then, there is an $n$-dimensional model for $\EFG{\FF}{G}$
    if and only if there exists a projective resolution of the trivial $\OFG{\FF}{G}$-module
    $\Ztriv{\FF}$ of length $n$ in $\RMod{\FF}{G}$.
\end{Proposition}

In \cite[pp. 151ff]{luecktransformation}, the author constructs a projective resolution of $\Ztriv{\FF}$
in $\RMod{\FF}{G}$ of length $n$, given an $n$-dimensional model for $\EFG{\FF}{G}$, which proves the
following:

\begin{Theorem}\label{cd_leq_gd}
    For any semi-full family $\FF$ of subgroups of a discrete group $G$ we have
    $$\cd_{\FF}(G) \leq \gd_{\FF}(G).$$
\end{Theorem}

And as a consequence of the two previous results we have:

\begin{Proposition}
    Let $\FF$ be a semi-full family of subgroups of $G$ such that $\cd_{\FF}(G) \geq 3$
    or $\gd_{\FF}(G)\geq 4$. Then, $\cd_{\FF}(G) = \gd_{\FF}(G)$.
\end{Proposition}

\begin{Defn}\label{BFG}
    Let $\FF$ be a semi-full family of subgroups of $G$. We say that $Y$ is a model
    for $\BFG{\FF}{G}$ if there is a model $X$ for $\EFG{\FF}{G}$ such that $Y = X/G$, that is,
    $Y$ is the orbit space of some classifying space for $\FF$.
\end{Defn}

\begin{Theorem}\cite[Theorem 4.2]{fluchthesis}\label{HomBFG}
    Let $\FF$ be a semi-full family of subgroups of $G$. Then, for every $n\in \N$ we have
    $$\Cohom{n}{\FF}{G}{\Ztriv{\FF}} \cong \Ho^{n}\left(\BFG{\FF}{G}\right),$$
    where $\Ho^*$ denotes the singular cohomology functor.
\end{Theorem}

\subsection{Mayer-Vietoris sequence for push-outs}
Push-outs of CW-complexes play a central role in the construction
of classifying spaces. In this section we give a Mayer-Vietoris type sequence for push-outs
relating the cohomology groups of the spaces involved.

\begin{Theorem}[Mayer-Vietoris Sequence for cellular Push-outs]\cite[Satz 3.12]{lueck_alg_top}\label{MVpushout}
    Consider the following push-out
    $$\begin{tikzcd}
    X
            \arrow[r, "\iota"]
            \arrow[d, "f"]
    & Y
            \arrow[d, "\bar{f}"] \\
    Z
            \arrow[r, "\bar{\iota}"]
    & P
    \end{tikzcd}
    $$
    in which $Y$ and $Z$ are CW-complexes, $X \subseteq Y$ is a subcomplex, $\iota : X \to Y$
    the inclusion and $f : X \to Z$ is a cellular map. Then, for every Homology theory $\mathcal{H}_*$
    we get the long exact Mayer-Vietoris sequence
    \begin{multline*}
        \cdots \xrightarrow{\partial_{n+1}} \GenHom{n}{X} \xrightarrow{\GenHom{n}{f}\oplus\GenHom{n}{\iota}}
        \GenHom{n}{Z} \oplus  \GenHom{n}{Y} \\
        \xrightarrow{\GenHom{n}{\bar{\iota}} - \GenHom{n}{\bar{f}}} \GenHom{n}{P}
        \xrightarrow{\partial_{n}} \GenHom{n-1}{X}\xrightarrow{\GenHom{n-1}{f}\oplus\GenHom{n-1}{\iota}}\\
        \GenHom{n-1}{Z} \oplus  \GenHom{n-1}{Y} \xrightarrow{\GenHom{n-1}{\bar{\iota}} - \GenHom{n-1}{\bar{f}}}\GenHom{n-1}{P}
        \xrightarrow{\partial_{n-1}} \cdots
    \end{multline*}
\end{Theorem}

\begin{Corollary}\label{MVCWpushout}
    Let $X$, $Y$ and $Z$ be CW-complexes, $f : X \to Z$ and $g : X \to Y$ cellular maps
    and $\mathcal{H}_*$ a homology theory. Then, if $P$ is the push-out of the diagram
    $Z \xleftarrow{f} X \xrightarrow{g} Y$, we get the long exact Mayer-Vietoris sequence
    \begin{multline*}
        \cdots \xrightarrow{\partial_{n+1}} \GenHom{n}{X} \xrightarrow{\GenHom{n}{f}\oplus\GenHom{n}{g}}
        \GenHom{n}{Z} \oplus  \GenHom{n}{Y} \\
        \xrightarrow{\GenHom{n}{\bar{g}} - \GenHom{n}{\bar{f}}} \GenHom{n}{P}
        \xrightarrow{\partial_{n}} \GenHom{n-1}{X}\xrightarrow{\GenHom{n-1}{f}\oplus\GenHom{n-1}{g}}\\
        \GenHom{n-1}{Z} \oplus  \GenHom{n-1}{Y} \xrightarrow{\GenHom{n-1}{\bar{g}} - \GenHom{n-1}{\bar{f}}}\GenHom{n-1}{P}
        \xrightarrow{\partial_{n-1}} \cdots
    \end{multline*}
\end{Corollary}
\begin{proof}
    Given that $g : X \to Y$ is a cellular map, by Corollary~\ref{cone_mcyl_dcyl_CW},
    the inclusion $\iota : X \to \mcyl{g}$ and the projection $\pi : \mcyl{g} \to Y$
    are cellular maps (with $X$ and $Y$ being subcomplexes of $\mcyl{g}$). By definition
    of $\mcyl{g}$, $\pi$ is a homotopy equivalence and we have $\pi \circ \iota = g$.
    Then, we can substitute $g : X \to Y$ by $\iota : X \to \mcyl{g}$ in the push-out
    and apply Theorem~\ref{MVpushout}.
\end{proof}

%% file: RelatedFamilies.tex
\chapter{Bredon dimensions for related families}\label{ch:related-families}
\minitoc

Given a group $G$ and a family of its subgroups $\FF$, building a classifying space
for the family $\FF$ can be accomplished by using known (or easier to build) classifying spaces
for other families of subgroups that are related to $\FF$ as raw materials. We can observe this
in many of the constructions for the family $\Vcy$ of virtually cyclic subgroups where
the known models for $\EFG{\Fin}{G}$ are heavily used. Similarly, there are results that provide
bounds to the Bredon cohomological dimension of a group $G$ over a family in terms of that over
a related family.

The purpose of this chapter is to present the results that provide bounds on Bredon dimensions
with respect to a family given the Bredon dimensions with respect to related families. The appropriate
topological constructs (for example, quotients, joins and mapping cones of CW-complexes)
have been already presented in previous chapters, so we proceed now to introduce
the cohomological ones.

\section{Restriction, induction and coinduction of Bredon modules}

\begin{Defn}\cite[9.12]{luecktransformation}
    Given a group $G$ and a family $\FF$ of subgroups of $G$, the \emph{tensor
    product over} $\FF$ is the bifunctor
    $$?\otimes_\FF ?? : \RMod{\FF}{G}\times\LMod{\FF}{G} \to \Ab$$
    defined by $$M \otimes_\FF N = \left(\coprod_{H\in\FF} M(G/H) \otimes N(G/H)\right)/\sim,$$
    where $\sim$ is the equivalence relation generated by $m\otimes N(f) \sim M(f)\otimes n$
    with $m\in M(G/H)$, $n\in N(G/K)$, $f \in [G/H, G/K]_\FF$ and $H,K\in\FF$.
\end{Defn}

Here, $\otimes$ denotes the tensor product of abelian groups over $\Z$.

The tensor product $\otimes_\FF$ can be made into a Bredon module evaluating it
in Bredon bimodules:

\begin{Defn}\cite[9.14]{luecktransformation}\label{defbimodule}
    Let $G_1$ and $G_2$ be two groups and $\FF_1$ and $\FF_2$ families of subgroups
    of $G_1$ and $G_2$ respectively. An $\OFG{\FF_1}{G_1}\text{-}\OFG{\FF_2}{G_2}$\emph{-bimodule}
    $M$ is a bifunctor $$M : \OFG{\FF_1}{G_1}\times\OFG{\FF_2}{G_2} \to \Ab$$
    that is covariant in the first variable and contravariant in the second.
\end{Defn}

\begin{Example}\cite[Example 1.8]{fluchthesis}
    As we saw in Example~\ref{ZfirstEx}, $\Zsecond{G}{H} \in \LMod{\FF}{G}$ and
    $\Zfirst{G}{K} \in \RMod{\FF}{G}$ for all $H,K \in \FF$. Then,
    $$\Znone{G} : \OFG{\FF}{G}\times\OFG{\FF}{G} \to \Ab$$
    is a $\OFG{\FF}{G}\text{-}\OFG{\FF}{G}$-bimodule (see \cite{fluchthesis} for
    a detailed definition).
\end{Example}

\begin{Defn}
    Let $G_1$, $G_2$, $\FF_1$ and $\FF_2$ be as in Definition~\ref{defbimodule}.
    Let $M$ be an $\OFG{\FF_2}{G_2}\text{-}\OFG{\FF_1}{G_1}$-bimodule and $N\in\LMod{\FF_1}{G_1}$.
    Then $$M(?,??)\otimes_{\FF_1} N(??)$$ is a left (covariant) $\OFG{\FF_2}{G_2}$-module.

    Symmetrically, let $M\in\RMod{\FF_1}{G_1}$ and $N$ be an $\OFG{\FF_1}{G_1}\text{-}\OFG{\FF_2}{G_2}$-bimodule.
    Then $$M(?)\otimes_{\FF_1} N(?,??)$$ is a right (contravariant) $\OFG{\FF_2}{G_2}$-module.
\end{Defn}

Note that in the tensor products in the above definition, the coproduct would
involve the abelian groups resulting from evaluating the Bredon bimodule and module
in the component marked by $??$ in the first case and by $?$ in the second, that in both
cases refer to elements of $\OFG{\FF_1}{G_1}$.

\begin{Defn}\cite[9.15 and pg. 350]{luecktransformation}
    Let $G_1$, $G_2$, $\FF_1$ and $\FF_2$ be as in Definition~\ref{defbimodule}.
    Let $F : \OFG{\FF_1}{G_1} \to \OFG{\FF_2}{G_2}$ be a functor between orbit categories.

    The \emph{restriction with} $F$ is the functor $\modres{F}: \RMod{\FF_2}{G_2} \to \RMod{\FF_1}{G_1}$
    defined by $\modres{F}M = M(??)\otimes_{\FF_2}\Z[F(?)\>\!,??]_{G_2}$.

    The \emph{induction with} $F$ is the functor $\modind{F}: \RMod{\FF_1}{G_1} \to \RMod{\FF_2}{G_2}$
    defined by $\modind{F}M = M(?)\otimes_{\FF_1}\Z[??\>\!,F(?)]_{G_2}$.

    The \emph{coinduction with} $F$ is the functor $\modcoind{F}: \RMod{\FF_1}{G_1} \to \RMod{\FF_2}{G_2}$
    defined by $\modcoind{F}M = \mor{\FF_1}{\Z[F(?)\>\!,??]_{G_2}}{M(?)}$.
\end{Defn}

The following properties of the functors described above are summarized in \cite[Proposition 1.31-1.35]{fluchthesis}:

\begin{Prop}\label{resindcoind_prop}
    Let $G_1$ and $G_2$ be groups, $\FF_1$ and $\FF_2$ families of subgroups of $G_1$ and $G_2$ respectively
    and $F : \OFG{\FF_1}{G_1} \to \OFG{\FF_2}{G_2}$ be a functor between orbit categories. Then, the following
    statements are true:
    \begin{enumerate}[label={\emph{(\arabic*)}},noitemsep]
        \item $\modind{F}$ is a left adjoint to $\modres{F}$;
        \item $\modcoind{F}$ is a right adjoint to $\modres{F}$;
        \item $\modres{F}$ is exact;
        \item $\modind{F}$ is right exact;
        \item $\modcoind{F}$ is left exact;
        \item $\modres{F}$ and $\modind{F}$ preserve arbitrary colimits;
        \item $\modres{F}$ and $\modcoind{F}$ preserve arbitrary limits;
        \item $\modind{F}$ preserves free and projective Bredon modules.
    \end{enumerate}
\end{Prop}

\section{Restriction to subgroups}

Let $\FF$ be a family of subgroups of $G$. Given a subgroup $K\leq G$ such that
$\FF\cap K \subset \FF$, let $I_K : \OFG{\FF \cap K}{K} \to \OFG{\FF}{G}$ be the
inclusion functor (defined by $I_K(H) = H$ for $H\in\FF \cap K$). Then, we have the
following results regarding restriction and induction with $I_K$:

\begin{Prop}\label{Bredon:ind_IK_exact}\cite[Proposition 3.26]{fluchthesis}
    Induction with $I_K$ is an exact functor.
\end{Prop}

\begin{Prop}\label{Bredon:res_IK_free}\cite[Proposition 3.28]{fluchthesis}
    Restriction with $I_K$ preserves free Bredon modules. In particular,
    it preserves projective Bredon modules.
\end{Prop}

These propositions, together with properties $(2)$ and $(3)$ in \ref{resindcoind_prop}, give rise
to the following isomorphisms in $\Ho^{*}$ and $\EXT^{*}$:

\begin{Prop}\label{Bredon:ext_IK}\cite[Proposition 3.29]{fluchthesis}
    For any $M\in\RMod{\FF}{G}$ and $N\in\RMod{\FF\cap K}{K}$ there is an isomorphism
    $$\Ext{*}{\FF\cap K}{\modres{I_K}M}{N} \cong \Ext{*}{\FF}{M}{\modcoind{I_K}N}$$
    that is natural in both $M$ and $N$.
\end{Prop}

\begin{Prop}\label{Bredon:cohom_IK}\cite[Proposition 3.31]{fluchthesis}
    For any $M\in\RMod{\FF\cap K}{K}$ there is an isomorphism
    $$\Cohom{*}{\FF\cap K}{K}{M} \cong \Cohom{*}{\FF}{G}{\modcoind{I_K}M}$$
    that is natural in $M$.
\end{Prop}

The following result can be obtained by using either the isomorphisms in
Proposition~\ref{Bredon:cohom_IK} and Corollary~\ref{cd_max_min} or the facts
that $\modres{I_K}$ is exact and preserves projectives and $\modres{I_K}\Ztriv{\FF} = \Ztriv{\FF\cap K}$.

\begin{Theorem}\label{Bredon:cd_subgroup}\cite[Proposition 3.32]{fluchthesis}
    Let $G$ be a group and $\FF$ a family of subgroups of $G$. Then for any $K \leq G$
    we have $\cd_{\FF \cap K} K \leq \cd_\FF G$.
\end{Theorem}

It is important to realise that, in general, not all right $\OFG{\FF\cap K}{K}$-modules
are of the form $\modres{I_K}M^\prime$ for $M^\prime\in\RMod{\FF}{G}$ and not all right
$\OFG{\FF}{G}$-modules
are of the form $\modcoind{I_K}N^\prime$ for $N^\prime\in\RMod{\FF\cap K}{K}$.
That is why we can not ensure the equality between the corresponding Bredon cohomological dimensions
by applying any of the aforementioned arguments that prove Theorem~\ref{Bredon:cd_subgroup}.

To relate the Bredon geometric dimensions, restricting ourselves to the case where $\FF$
is a full family (and hence for any subgroup $K\leq G$, $\FF\cap K \subset \FF$ holds),
note that a model for $\EFG{\FF}{G}$ is also a model for $\EFG{\FF\cap K}{K}$, and therefore:

\begin{Prop}\label{Bredon:gd_subgroup}\cite[Proposition 3.33]{fluchthesis}
    Let $G$ be a group and $\FF$ a full family of subgroups of $G$. Then for any $K \leq G$
    we have $\gd_{\FF \cap K} K \leq \gd_\FF G$.
\end{Prop}

\section{Families related by a functor}

In this section we will describe the results in the paper \cite{mpconchita} that provide upper bounds
for the Bredon cohomological dimension and similar results included (or derived from those) in \cite{lueckweiermann} for the Bredon geometric dimension of two families related by a functor.

Let $G$ be a discrete group and $\FF$ and $\HH$ two families of subgroups of $G$. Let $\pi : \FF \to \HH$ be such that
$\bar{\pi} : \OFG{\FF}{G} \to \OFG{\HH}{G}$ defined by $\bar{\pi}(G/H) = G/\pi(H)$ is a covariant functor.

In \cite[Theorem 3.9]{mpconchita}, the author proves the existence and convergence of a spectral sequence relating the Bredon cohomology
groups for both families, given the following conditions for every $S\in\HH$:

\begin{enumerate}[label={(\text{MP}\arabic*)},noitemsep]
    \item For $g \in G$ and $L \in \FF$, $L^g \leq S$ if and only if $\pi(L)^g \leq S$, and
    \item $\FF \cap S \subseteq \FF$.
\end{enumerate}

As a consequence of this result, we have:

\begin{Corollary}\cite[Corollary 4.1]{mpconchita}\label{inclusion_cd_goes_up_boundedly}
    Let $\FF, \HH, \pi : \FF \to \HH$ as above satisfying conditions (1) and (2) and assume that we have an integer n such
    that for any $S\in\HH$, $\cd_{\FF \cap S} S \leq n$. Then,
    $$ \cd_\FF G \leq n + \cd_\HH G.$$
\end{Corollary}

In the case of the Bredon geometric dimension, we restrict ourselves to $\FF \subseteq \GG$
being full families of subgroups of $G$. In that setting, we have the following result:

\begin{Proposition}\cite[Proposition 5.1 (i)]{lueckweiermann}\label{inclusion_gd_goes_up_boundedly}
    Let $\FF\subseteq\GG$ be full families of subgroups of a group $G$. Then, if there is $n\in\N$
    such that $\gd_{\FF\cap H} H \leq n$ for every $H\in\GG$,
    $$\gd_{\FF} G \leq \gd_{\GG} G + n.$$
\end{Proposition}

\subsection{Passing to quotients}

Let now $G$ be a group, $N \lhd G$ and $\FF$ a family of subgroups of $G$. Let $\HHb$ be a family of subgroups of $\bar{G} = G/N$
satisfying the following conditions:
\begin{enumerate}[label={(\roman*)},noitemsep]
    \item For any $L\in\FF$, $LN/N \in \HHb$ and
    \item for any $S/N \in \HHb$, $\FF \cap S \subseteq \FF$.
\end{enumerate}

Let $\HH = \{ S\leq G \st N \leq S \text{ and } S/N \in \HHb\}$ and $\pi : \FF \to \HH$ defined by $\pi(L) = LN$,
$\FF$,$\HH$ and $\pi$ satisfy conditions (MP1) and (MP2).

Moreover, if we take $\Phi : \RMod{\HH}{G} \to \RMod{\HHb}{\bar{G}}$ defined by
$\Phi(M)\left(\bar{G}/\bar{H}\right) = M(G/H)$, since $\{ S \mapsto S/N \st S\in\HH\}$
is a bijection between $\HH$ and $\HHb$, we have that for all $M\in\RMod{\HH}{G}$
and all $n\geq 0$
$$\Cohom{n}{\HH}{G}{M} \cong \Cohom{n}{\HHb}{\bar{G}}{\Phi(M)},$$
which means $\cd_\HH G \leq \cd_{\HHb} \bar{G}$.

Now, if we apply Corollary~\ref{inclusion_cd_goes_up_boundedly} to $\FF$, $\HH$ and $\pi : \FF \to \HH$,
we can conclude:

\begin{Corollary}\cite[Corollary 5.2.]{mpconchita}\label{cd_for_quotients}
    Under the previous assumptions over $\FF$, $\HH$, $\pi$ and $\HHb$, if there
    is $n\geq 0$ such that for any $S\in\HH$, $\cd_{\FF\cap S} S \leq n$, then
    $$\cd_\FF G \leq n + \cd_{\HHb} \bar{G}.$$
\end{Corollary}

As previously, in the case of the Bredon geometric dimension, we restrict outselves to
$\FF$ and $\HHb$ being full families of subgroups.

\begin{Lemma}\label{class_space_for_quotients}
    Let $G$ be a discrete group, $N\lhd G$ and $\HHb$ a full family of subgroups of $\bar{G} = G/N$.
    Let $\HH = \{L \leq G\st LN/N \in\HHb\}$ and $X$ be a model for $\EFG{\HHb}{\bar{G}}$. Then,
    $\HH$ is a full family of subgroups of $G$ and $X$ is a model for $\EFG{\HH}{G}$.
    In particular,
    $$\gd_{\HH} G \leq \gd_{\HHb} \bar{G}.$$
\end{Lemma}

\begin{proof}
    To see that $\HH$ is a full family of subgroups of $G$, we need to see that it is
    closed under conjugation and under taking subgroups. Let $S\in\HH$ and $g\in G$. Since $N \lhd G$,
    we have $S^gN/N = \left(SN/N\right)^{gN}$. Moreover, $\HHb$ is closed under conjugation, so
    $\left(SN/N\right)^{gN}\in\HHb$. Hence, $S^g\in\HH$, as we needed to see. Now let $S\in\HH$ and
    $K \leq S$. Since $KN/N \leq SN/N$ and $\HHb$ is closed under taking subgroups, $KN/N\in\HHb$, which
    completes the proof for $\HH$ being a full family of subgroups of $G$.

    We define the $G$-action on $X$ as $gx = (gN)x$, where $g\in G$, $x\in X$ and $(gN)x$ denotes the
    $G/N$-action on $X$. Then, if $K\leq G$, the set $X^K$ of fixed points of the $G$-action on $X$
    by $K$ is equal to the set $X^{KN/N}$ of fixed points of the $G/N$-action on $X$ by $KN/N$.

    Let $L\in \HH$. Since $LN/N\in\HHb$, $X^{LN/N}$ is contractible and so is $X^L$. Let now
    $L\notin \HH$. By definition of $\HH$, that means $LN/N\notin \HHb$, and for that reason
    $X^L = X^{LN/N} = \emptyset$. Hence, by Corollary~\ref{other_def_of_class_space}, $X$ is a
    model for $\EFG{\HH}{G}$.
\end{proof}

\begin{Theorem}\label{gd_for_quotients}
    Let $G$ be a discrete group, $N \lhd G$ and $\FF$ a full family of subgroups of $G$.
    Let $\HHb$ be a full family of subgroups of $\bar{G} = G/N$ such that for every $L\in\FF$
    the subgroup $LN/N$ of $\bar{G}$ belongs to $\HHb$. Let $\HH = \{L \leq G\st LN/N \in\HHb\}$.
    Then, if there is $n\in\N$ such that $\gd_{\FF\cap S} S \leq n$ for any $S\in\HH$,
    $$\gd_\FF G \leq n + \gd_{\HHb} \bar{G}.$$
\end{Theorem}

\begin{proof}
    By Lemma~\ref{class_space_for_quotients} and Proposition~\ref{inclusion_gd_goes_up_boundedly},
    showing that $\FF \subseteq \HH$ will conclude the proof. Let $K\in\FF$. Then, by hypothesis,
    $LN/N\in\HHb$. That means, by definition of $\HH$, that $L\in\HH$, as we wanted to see.
\end{proof}

\section{Union of families}
Given two families of subgroups of a group $G$, we can build a classifying space for the union of those families using
the classifying spaces for each of the families. The first direct approuch gives us the following result:

\begin{Lemma}\label{E_FUGjoin}
Let $\FF$ and $\GG$ be two families of subgroups of $G$. Then
$$\gd_{\FF\cup\GG}G \leq \gd_\FF G + \gd_\GG G + 1.$$
\end{Lemma}

\begin{proof}
Let $X,Y$ be models for $\EFG{\FF}{G}$ and $\EFG{\GG}{G}$ respectively. Let $Z = X * Y$ be the join of these two spaces, which
is a $G$-CW-complex by Corollary~\ref{joinGCW}.
Then $G$ acts on $Z$ as it acts on $X$ and $Y$ on each extreme of the interval and diagonally in the rest of $Z$. With this
action defined, $Z$ is a model for $\EFG{\FF \cup \GG}{G}$. And we finish the proof by noting $\operatorname{dim}(Z) = \operatorname{dim}(X) + \operatorname{dim}(Y) + 1$.
\end{proof}

If additionally we also take into account the classifying space over the intersection of those families, we have:

\begin{Lemma}\label{E_FUG}\cite[Lemma 2.4]{nucinkisetal}
Let $\FF$ and $\GG$ be two full families of subgroups of $G$. Then
$$\gd_{\FF\cup\GG}G \leq \max\{\gd_\FF G, \gd_\GG G, \gd_{\FF\cap\GG}G +1\}.$$
\end{Lemma}

\begin{proof}
    Let $X$, $Y$ and $Z$ be models for $\EFG{\FF}{F}$, $\EFG{\GG}{G}$ and
    $\EFG{\FF\cap\GG}{G}$ respectively. By the universal property of classifying
    spaces for families, there are $G$-maps, unique up to $G$-homotopy, $h :Z \to Y$
    and $f : Z \to X$. By the Cellular Approximation Theorem (\ref{CellApprox}),
    $f$ and $h$ can be assumed to be cellular. By Corollary~\ref{dmcylGCW}, $\dmcyl{f}{h}$
    is a $G$-CW-complex.

    Let $B = \left(\left(Z\times I_X\right) \sqcup X\right) \sqcup \left(\left(Z\times I_Y\right) \sqcup Y\right)$,
    where $I_X$ and $I_Y$ are two copies of $[0,1]$.
    Let $\pi: B \to \dmcyl{f}{h}$ be the quotient map. Then, if $g\in G$, $g$ acts
    on $\dmcyl{f}{h}$ by $g \pi(p,t) = \pi(gp,t)$ for $p \in Z$ and $t\in[0,1]$ and $g \pi(q) = \pi(g q)$ for $q\in X$
    or $q \in Y$. Then, given $H\leq G$,
    $\dmcyl{f}{h}^H = \pi(\left(\left(Z^H\times I_X\right) \sqcup X^H\right) \sqcup \left(\left(Z^H\times I_Y\right) \sqcup Y^H\right))$.
    Since
    And since $f$, $h$ and $\pi$ are $G$-maps, we can conclude
    $$\dmcyl{f}{h}^H = \pi(\left(\left(Z^H\times I_X\right) \sqcup X^H\right) \sqcup \left(\left(Z^H\times I_Y\right) \sqcup Y^H\right))$$

    Let $H \in \FF \cup \GG$. If $H \in \FF\bs\GG$, $Y^H$ and $Z^H$ are both
    empty, since $H\notin \GG$ and $H\notin \FF\cap\GG$. In that case, $X^H$ is
    non-empty and contractible, so $\dmcyl{f}{h}^H = \pi(X^H)$. The
    restriction of $\pi$ to $X\subseteq \dmcyl{f}{h}$ is the identity, so $\pi(X^H)$ is
    also contractible. Analogously in the case $H\in\GG\bs\FF$. In the case $H\in\FF\cap\GG$,
    $X^H$, $Y^H$, $Z^H$ and hence $Z^H\times I_X$ and $Z^H\times I_Y$ are all non-empty and contractible.
    Note that $f(Z^H) \subseteq {(f(Z))}^H \subseteq X^H$ and $h(Z^H)\subseteq{(h(Z))}^H \subseteq Y^H$.
    Therefore, $\dmcyl{f}{h}^H$ is non-empty and contractible. Finally, if $H \notin \FF\cup\GG$,
    $\dmcyl{f}{h}^H = \emptyset$ as $X^H = Z^H = Y^H = \emptyset$.

    Hence, $\dmcyl{f}{h}$ yields a model for $\EFG{\FF\cup \GG}{G}$ of the desired dimension.
\end{proof}

And for this last case, we can apply Corollary~\ref{MVCWpushout} and use the following corollary to
obtain a Bredon cohomological equivalent of Lemma~\ref{E_FUG}.

\begin{Lemma}\label{cdE_FUG}
    Let $\FF$ and $\GG$ be two full families of subgroups of $G$,
    $M$ be a right $\OFG{\FF\cup\GG}{G}$-module and $F_\FF : \FF \to \FF\cup\GG$,
    $F_\GG : \GG \to \FF\cup\GG$
    and $I : \FF\cap\GG \to \FF\cup\GG$ the inclusion functors. Then, the following
    sequence in Bredon Cohomology is exact
    \begin{multline*}
        \cdots \longrightarrow \Cohom{n}{\FF\cup\GG}{G}{M} \longrightarrow \\
        \Cohom{n}{\FF}{G}{\modres{F_\FF} M} \oplus  \Cohom{n}{\GG}{G}{\modres{F_\GG}M}
        \longrightarrow \\ \Cohom{n}{\FF\cap\GG}{G}{\modres{I} M}
        \longrightarrow \Cohom{n+1}{\FF\cup\GG}{G}{M}\longrightarrow \cdots
    \end{multline*}

    and hence
    $$\cd_{\FF\cup\GG}G \leq \max\{\cd_\FF G, \cd_\GG G, \cd_{\FF\cap\GG}G +1\}.$$
\end{Lemma}
\begin{proof}
    Let $X$, $Y$ and $Z$ be models for $\EFG{\FF}{G}$, $\EFG{\GG}{G}$ and
    $\EFG{\FF\cap\GG}{G}$ respectively. Consider $P = \dmcyl{f}{h}$ the model for $\EFG{\FF\cup\GG}{G}$
    described in the proof of Lemma~\ref{E_FUG}.

    Given $M\in\RMod{\FF\cup\GG}{G}$ and taking as (co)homology theory the one defined in \ref{bredcohom_on_class_spaces}, by
    Corollary~\ref{MVCWpushout} we have the following long exact sequence:
    \begin{multline*}
        \cdots \longrightarrow \Cohom{n-1}{\FF\cup\GG}{P}{M} \longrightarrow \\
        \Cohom{n-1}{\FF}{X}{\modres{F_\FF} M} \oplus  \Cohom{n-1}{\GG}{Y}{\modres{F_\GG}M}
        \longrightarrow \\ \Cohom{n-1}{\FF\cap\GG}{Z}{\modres{I} M}
        \longrightarrow \Cohom{n}{\FF\cup\GG}{P}{M} \longrightarrow \cdots
    \end{multline*}

    And using Corollary~\ref{cohom_EFG_Bred_cohom}, the previous Mayer-Vietoris
    long exact sequence is equivalent to the following long exact sequence:
    \begin{multline*}
        \cdots \longrightarrow \Cohom{n}{\FF\cup\GG}{G}{M} \longrightarrow \\
        \Cohom{n}{\FF}{G}{\modres{F_\FF} M} \oplus  \Cohom{n}{\GG}{G}{\modres{F_\GG}M}
        \longrightarrow \\ \Cohom{n}{\FF\cap\GG}{G}{\modres{I} M}
        \longrightarrow \Cohom{n+1}{\FF\cup\GG}{G}{M} \longrightarrow \cdots
    \end{multline*}

    Now, if we take $d = \max\{\cd_\FF G, \cd_\GG G, \cd_{\FF\cap\GG}G +1\}$, we know
    $$\Cohom{d+1}{\FF}{G}{\modres{F_\FF} M} = \Cohom{d+1}{\GG}{G}{\modres{F_\GG}M} = \Cohom{d}{\FF\cap\GG}{G}{\modres{I} M} = 0$$
    by Proposition~\ref{pd_zeros}. Hence, the long exact sequence above is
    $$\cdots \to 0 \to \Cohom{d+1}{\FF\cup\GG}{G}{M} \to 0 \to \cdots$$
    As the sequence is exact, in particular it is exact at $\Cohom{d+1}{\FF\cup\GG}{G}{M}$,
    which means $\Cohom{d+1}{\FF\cup\GG}{G}{M} = 0$. Since this is true for all
    $M\in\RMod{\FF\cup\GG}{G}$, by Proposition~\ref{pd_zeros},
    we can conclude ${\cd_{\FF\cup\GG} G \leq d}$, as we wanted to see.
\end{proof}

\section{Strongly structured inclusions}\label{sec:LW}

\subsection{A pre-example}

We will firstly describe Farrell's construction, which produces a classifying space
for the family of (virtually) cyclic subgroups of $\Z^2$. This construction can be also
found in \cite[p. 108]{juanpinedaleary} and in \cite[pp. 87-89]{fluchthesis}.

Studying these classical constructions of classifying spaces, we can perceive a hint of
what nowadays represents the most broadly used procedure to build such spaces for a family $\GG$
of subgroups of a group $G$ given known models for $\EFG{\FF}{G}$ and other classifying spaces for families
related with $\FF$, where $\FF \subseteq \GG$:
the L{\"u}ck-Weiermann method.

Let then $G = \Z^2$ and $\FF = \Vcy$ be the family of (virtually) cyclic subgroups of $G$, which
in this case coincides with the family of the subgroups of $\Z^2$ that are isomorphic to $\Z$.

Note that for every $H\in\FF$ there is a unique $\bar{H}\in\FF$ such that $H\leq\bar{H}$
and $\bar{H}$ is maximal among those subgroups with the same properties, i.e., if
$L\in\FF$ is such that $H\leq L$ then $L\leq\bar{H}$.

Let $\mathcal{H}$ denote the set of the maximals subgroups of $\FF$ we just described. If
$H_1, H_2\in\mathcal{H}$ are different, then their intersection is necessarily trivial, as
a proper non-trivial intersection between two maximals would contradict the fact that they are maximals.
Moreover, $\Z^2/H \cong \Z$ for any $H\in\mathcal{H}$.

The set $\mathcal{H}$ can be indexed by $\Z$ and we write $H_i$ to denote the $i$-th maximal subgroup.
As we saw in Example~\ref{RasEGforZ}, and since $\Z^2/H \cong \Z$, if we take $X_i$ to be a copy
of $\R$, $X_i$ is a model for $\EFG{\{\{1\}\}}{\Z^2/H_i}$, which in this case is equivalent
to being a model for $\EFG{\Fin}{\Z^2/H_i}$.

For any $i\in\Z$, we can define a $\Z^2$-action on $X_i$ as $gx = gH_i x$ for $g\in\Z^2$ and
for all $x\in X_i$, where $gH_i$ represents the class of $g$ in $\Z^2/H_i$. Note that
if $H\leq \Z^2$, the set of fixed points by $H$ with respect to the described $\Z^2$-action on $X_i$
is contractible if $H\leq H_i$ or empty otherwise, that is, $X_i$ is a model for $\EFG{\All(H_i)}{\Z^2}$,
where $\All(\Z^2)$ is the family of all subgroups of $H_i$.

Consider now the space $Y_i = X_i*X_{i+1}$ for $i\in\Z$, on which we have the diagonal $\Z^2$-action
defined from the $\Z^2$-actions in each of the spaces in the join in the proof of Corollary~\ref{joinGCW}.
Let $K\leq \Z^2$. If $K \leq H_i$ is a non-trivial subgroup, since $H_i \cap H_{i+1} = \emptyset$
and $K\in\All(H_i)$, we have $Y_i^K = X_i^K \simeq \{*\}$. The same reasoning holds
for any non-trivial $K\leq H_{i+1}$. In the case of $K = \{1\}$, $Y_i^K = Y_i \simeq \{*\}$. And finally,
in the case where $K\notin\All(H_i)\cup\All(H_{i+1})$, the set of fixed points by $K$ will be empty,
as $X_i^K = X_{i+1}^K$.

Hence, $Y_i$ is a model for $\EFG{\All(H_i)\cup\All(H_{i+1})}{\Z^2}$. This can now be proved using
Lemma~\ref{E_FUGjoin}. Moreover, taking $A_i = \{[x,y,t]\in Y_i \st t = 1/2\}$, $A_i$ is a classifying
space for the family $\All(H_i)\cap\All(H_{i+1}) = \{\{1\}\}$ of subgroups of $\Z^2$ so Lemma~\ref{E_FUG}
can also be used.

If we take
$$X = \left(\bigsqcup_{i\in\Z} X_i*X_{i+1}\right)/\sim,$$
where $\sim$ is the equivalence relation consisting on identifying, for each $i\in\Z$,
the pair of copies of $X_i$ in $X$, it follows that $X$ is a model for $\EFG{\Vcy}{\Z^2}$.

We can observe a similar construction in \cite{connollyetal}, and it is a very interesting exercise
to build the same spaces obtained in that publication using the method we will describe in this section
and find the resemblances of both approaches.

\subsection{Construction}

\begin{Defn}\label{def_strong_equiv_LW}\cite[(2.1)]{lueckweiermann}
    Let $\FF$ and $\GG$ families of subgroups of a given group $G$ such that $\FF \subseteq \GG$. Let $\sim$ be an equivalence relation on $\GG\bs\FF$ satisfying:
    \begin{enumerate}[label={(\roman*)}, noitemsep]
        \item For $H,K \in \GG\bs\FF$ with $H \leq K$ we have $H \sim K.$
        \item Let $H,K \in \GG\bs \FF$ and $g \in G$, then $H \sim K \iff gHg^{-1} \sim gKg^{-1}.$
    \end{enumerate}
    We call $\sim$ a \textit{strong equivalence relation}.
    Denote by $[\GG\bs\FF] $ the equivalence classes of $\sim$ and define for all $[H]\in [\GG\bs\FF]$ the following subgroup of $G$:
    $$\Nzer{G}{H} =\{g \in G\,|\, [gHg^{-1}]=[H]\}.$$
    Now define a family of subgroups of $\Nzer{G}{H}$ by
    $$\GG[H] =\{K \leq \Nzer{G}{H} \,|\, K \in \GG\bs\FF \, ,\, [K]=[H]\}\cup (\FF \cap \Nzer{G}{H}).$$
    Here $\FF \cap \Nzer{G}{H}$ is the family of subgroups of $\Nzer{G}{H}$ belonging to $\FF$.
\end{Defn}

\begin{Theorem}\label{lw-main}\cite[Theorem 2.3]{lueckweiermann} Let $\FF \subseteq \GG$ and $\sim$ be as in Definition \ref{def_strong_equiv_LW}. Denote by $\mathcal{H}$ a complete set of representatives of the conjugacy classes in $[\GG\bs\FF].$ Then the $G$-CW-complex given by the cellular $G$-pushout
$$\begin{tikzcd}
\underset{[H]\in \mathcal{H}}{\bigsqcup} G \times_{\Nzer{G}{H}}\EFG{\FF \cap \Nzer{G}{H}}{\Nzer{G}{H}}
        \arrow[r, "\iota"]
        \arrow[d, "\underset{[H]\in \mathcal{H}}{\sqcup} id_G \times_{\Nzer{G}{H}}f_{[H]}"]
& \EFG{\FF}{G}
        \arrow[d] \\
\underset{[H]\in \mathcal{H}}{\bigsqcup} G \times_{\Nzer{G}{H}}\EFG{\GG[H]}{\Nzer{G}{H}}
        \arrow[r]
& X
\end{tikzcd}
$$
where either $\iota$ or the $f_{[H]}$ are inclusions, is a model for $\EFG{\GG}{G}.$
\end{Theorem}

\subsection{Mayer-Vietoris sequence}

In \cite{lueckweiermann}, the authors use a Mayer-Vietoris type long exact sequence in Bredon
Cohomology that can be derived from Theorems \ref{lw-main} and \ref{MVpushout}. We include such
derivation for completeness and comprehension, given that it is not explicitly detailed
in the original source.

\begin{Proposition}\label{MV_LW}
    Let $\FF\subseteq\GG$ be two full families of subgroups of $G$ such that
    there is a strong equivalence relation $\sim$ in $\GG\bs\FF$, as in
    Definition~\ref{def_strong_equiv_LW}. Let $\mathcal{H}$ be a set of
    representatives of the classes in $[\GG\bs\FF]$. Let $M\in\RMod{\GG}{G}$. Let
    $F_\FF : \FF \to \GG$ and $F_{[H]} : \GG[H] \to \GG$
    and $I_{[H]} : \FF\cap\Nzer{G}{H} \to \GG$ the inclusion functors for each $H\in\mathcal{H}$ .
    Then, the following sequence in Bredon Cohomology is exact
    \begin{multline*}
        \cdots \longrightarrow \Cohom{n-1}{\GG}{G}{M} \longrightarrow \\
        \left( \underset{[H]\in \mathcal{H}}{\prod}
        \Cohom{n-1}{\GG[H]}{\Nzer{G}{H}}{\modres{F_{[H]}} M}\right)
        \oplus  \Cohom{n-1}{\FF}{G}{\modres{F_\FF} M}
        \longrightarrow \\
        \underset{[H]\in\mathcal{H}}{\prod} \Cohom{n-1}{\FF\cap\Nzer{G}{H}}{\Nzer{G}{H}}{\modres{I_{[H]}} M}
        \longrightarrow \Cohom{n}{\GG}{G}{M}\longrightarrow \cdots
    \end{multline*}
\end{Proposition}
\begin{proof}
    Let $X_{[H]}$, $Y$ and $Z_{[H]}$ be models for $\EFG{\GG[H]}{\Nzer{G}{H}}$, $\EFG{\FF}{G}$ and
    $\EFG{\FF\cap\Nzer{G}{H}}{\Nzer{G}{H}}$ for each $[H]\in\mathcal{H}$, respectively.
    Consider $P$ the model for $\EFG{\GG}{G}$ obtained as the
    $G$-pushout of the diagram in Theorem~\ref{lw-main}.

    Given $M\in\RMod{\GG}{G}$ and taking as (co)homology theory the one
    defined in \ref{bredcohom_on_class_spaces}, by
    Corollary~\ref{MVCWpushout} we have the following long exact sequence:
    \begin{multline*}
        \cdots \longrightarrow \Cohom{n-1}{\GG}{P}{M} \longrightarrow \\
        \left( \underset{[H]\in \mathcal{H}}{\prod}
        \Cohom{n-1}{\GG[H]}{X_{[H]}}{\modres{F_{[H]}} M}\right)
        \oplus  \Cohom{n-1}{\FF}{Y}{\modres{F_\FF} M}
        \longrightarrow \\
        \underset{[H]\in\mathcal{H}}{\prod} \Cohom{n-1}{\FF\cap\Nzer{G}{H}}{Z_{[H]}}{\modres{I_{[H]}} M}
        \longrightarrow \Cohom{n}{\GG}{P}{M}\longrightarrow \cdots
    \end{multline*}

    And using Corollary~\ref{cohom_EFG_Bred_cohom}, the previous Mayer-Vietoris
    long exact sequence is equivalent to the following long exact sequence:
    \begin{multline*}
        \cdots \longrightarrow \Cohom{n-1}{\GG}{G}{M} \longrightarrow \\
        \left( \underset{[H]\in \mathcal{H}}{\prod}
        \Cohom{n-1}{\GG[H]}{\Nzer{G}{H}}{\modres{F_{[H]}} M}\right)
        \oplus  \Cohom{n-1}{\FF}{G}{\modres{F_\FF} M}
        \longrightarrow \\
        \underset{[H]\in\mathcal{H}}{\prod} \Cohom{n-1}{\FF\cap\Nzer{G}{H}}{\Nzer{G}{H}}{\modres{I_{[H]}} M}
        \longrightarrow \Cohom{n}{\GG}{G}{M}\longrightarrow \cdots
    \end{multline*}
\end{proof}

This result was generalised for arbitrary $\operatorname{Ext}$ functors. The approach in that
case is strictly algebraic, meaning that it is independent from the geometric construction in
Theorem~\ref{lw-main}:

\begin{Theorem}\cite[Theorem 7.7]{degrijsepetrosyan}
    Let $\FF\subseteq\GG$ be two families of subgroups of a group $G$ such that the set
    $\GG\bs\FF$ is equipped with a strong equivalence relation. Let $\mathcal{H}$ be a set
    of representatives of the classes in $[\GG\bs\FF]$. Let $M\in\RMod{\GG}{G}$. Let
    $F_\FF : \FF \to \GG$ and $F_{[H]} : \GG[H] \to \GG$
    and $I_{[H]} : \FF\cap\Nzer{G}{H} \to \GG$ the inclusion functors for each $H\in\mathcal{H}$ .
    Then, the following sequence is exact:
    \begin{multline*}
        \cdots \longrightarrow \Ext{n-1}{\GG}{M}{N} \longrightarrow \\
        \left( \underset{[H]\in \mathcal{H}}{\prod}
        \Ext{n-1}{\GG[H]}{\modres{F_{[H]}} M}{\modres{F_{[H]}} N}\right)
        \oplus  \Ext{n-1}{\FF}{\modres{F_\FF} M}{\modres{F_\FF} N}
        \longrightarrow \\
        \underset{[H]\in\mathcal{H}}{\prod}
        \Ext{n-1}{\FF\cap\Nzer{G}{H}}{\modres{I_{[H]}} M}{\modres{I_{[H]}} N}
        \longrightarrow \Ext{n}{\GG}{M}{N}\longrightarrow \cdots
    \end{multline*}
\end{Theorem}

\subsection{Dimensions}

The condition in Theorem \ref{lw-main} on the two maps being inclusions is not
that strong a restriction as one can replace the spaces by the mapping cylinders, as we saw
in Corollary~\ref{MVCWpushout}.

\begin{Corollary}\label{lw-gdim}\cite[Remark 2.5]{lueckweiermann}
    Suppose there exists an $n$-dimensional model for $\EFG{\FF}{G}$ and for
    each $H \in \mathcal{H}$ there exist a $(n-1)$-dimensional model for
    $\EFG{\FF \cap \Nzer{G}{H}}{\Nzer{G}{H}}$ and a $n$-dimensional model for
    $\EFG{\GG[H]}{\Nzer{G}{H}}$. Then there is an $n$-dimensional model for $\EFG{\GG}{G}.$
\end{Corollary}

Analogously, as a consequence of the long exact sequence in Proposition~\ref{MV_LW}:

\begin{Theorem}\label{lw-cdim}\cite[Theorem 7.2]{degrijsepetrosyan}
    Suppose there is a natural number $n$ such that $\cd_\FF G \leq n$
    and for each $[H] \in \mathcal{H}$ $\cd_{\FF \cap \Nzer{G}{H}} \Nzer{G}{H} \leq n -1$
    and $\cd_{\FF[H]} \Nzer{G}{H} \leq n$. Then $\cd_\GG G \leq n$.
\end{Theorem}

If we express Corollary \ref{lw-gdim} in terms of Bredon geometrical dimensions and rephrase Theorem \ref{lw-cdim}, we obtain upper bounds for $\gd_\GG G$ and $\cd_\GG G$.

\begin{Corollary}\label{lw-dimensions} The following inequalities hold:
    \begin{enumerate}[label={\emph{(\roman*)}}]
        \item $\gd_\GG G \leq
        \operatorname{max}\{\underset{[H] \in \mathcal{H}}{\operatorname{max}}\{\gd_{\GG[H]}(\Nzer{G}{H}),
        \gd_{\FF \cap \Nzer{G}{H}}(\Nzer{G}{H}) + 1\},
        \gd_\FF G \}$

        \item $\cd_\GG G \leq
        \operatorname{max}\{\underset{[H] \in \mathcal{H}}{\operatorname{max}}\{\cd_{\GG[H]}(\Nzer{G}{H}),
        \cd_{\FF \cap \Nzer{G}{H}}(\Nzer{G}{H}) + 1\},
        \cd_\FF G \}$
    \end{enumerate}
\end{Corollary}

\section{Families of subgroups of a direct union of groups}

In this section we will present a series of results that will later be useful to
extend applications of the theorems in Chapter~\ref{ch:chains}. These results can be found
in \cite{nucinkis2004dimensions}, \cite{lueckweiermann} and \cite{fluchthesis}.

\begin{Defn}\label{def_direct_union}
    Let $\{G_\lambda \st \lambda \in \Lambda\}$ be a set of subgroups of $G$, where $\Lambda$
    is an indexing set. We say that $G$ is the \emph{direct union of the groups } $G_\lambda$
    if the the following conditions hold:
    \begin{enumerate}[label = {(\roman*)}, noitemsep]
        \item for every $\lambda, \mu \in \Lambda$ there is $\nu\in\Lambda$ such that
        $G_\lambda \leq G_\nu$ and $G_\mu \leq G_\nu$ and
        \item $G \subseteq \bigcup_{\lambda \in \Lambda} G_\lambda$ as sets.
    \end{enumerate}
\end{Defn}

Direct unions are a particular case of direct limits if we take the partial order in $\Lambda$
$\lambda \leq \mu$ if and only if $G_\lambda \leq G_\mu$ and taking inclusions as homomorphisms.

\begin{Defn}\label{Flambda_compatible}
    Let $G$ be the direct union of $\{G_\lambda \st \lambda\in\Lambda\}$. Let $\FF$
    be a family of subgroups of $G$ and for each $\lambda\in\Lambda$ let $\FF_\lambda$
    be a family of subgroups of $G_\lambda$. We say $\FF$ and $\FF_\lambda$ for $\lambda\in\Lambda$
    are \emph{compatible with the direct union} if the following holds for every
    $\lambda,\mu\in\Lambda$:
    \begin{enumerate}[label={(\arabic*)}, noitemsep]
        \item $\FF_\lambda \subseteq \FF_\mu$ if $\lambda\leq\mu$;
        \item $\FF_\lambda \subseteq \FF$;
        \item $\FF \subseteq \bigcup_{\lambda\in\Lambda} \FF_\lambda$ and
        \item $\FF_\lambda = \FF \cap G_\lambda$.
    \end{enumerate}
\end{Defn}

\begin{Proposition}\label{fg_implies_compatible}\cite[Proposition 3.43]{fluchthesis}
    Let $G$ be the direct union of $\{G_\lambda \st \lambda\in\Lambda\}$ and let $\FF$
    be a full family of finitely generated subgroups of $G$. Then, $\FF$ and
    $\{\FF\cap G_\lambda\st\lambda\in\Lambda\}$ are compatible with the direct union.
\end{Proposition}

\begin{Proposition}\label{every_sg_in_one_summand_implies_compatible}
    Let $G$ be the direct union of $\{G_\lambda \st \lambda\in\Lambda\}$ and let
    $\FF$ be a full family of subgroups of $G$ such that for every $K\in\FF$ there is
    $\lambda\in\Lambda$ such that $K\leq G_\lambda$. Then, $\FF_\lambda = \FF \cap G_\lambda$
    for $\lambda\in\Lambda$ and $\FF$ are compatible with the direct union.
\end{Proposition}
\begin{proof}
    Conditions (1), (2) and (4) in Definition~\ref{Flambda_compatible} are true given
    that $\FF$ is a full family of subgroups and the fact that we are taking
    $\FF_\lambda = \FF \cap G_\lambda$. To prove condition (iii), let $K\in\FF$. By hypothesis,
    we know there is $\lambda\in\Lambda$ such that $K \leq G_\lambda$. Hence, $K\cap G_\lambda = K$,
    so $K \in \FF_\lambda$, as we needed to see.
\end{proof}

The following theorem can be deduced from \cite[Theorem 4.3]{lueckweiermann} and
\cite[Theorem 4.1]{nucinkis2004dimensions}, but since we didn't introduce flat Bredon
modules and hence Bredon homological dimensions, we will only give the cohomological version:

\begin{Theorem}\label{cd_direct_unions}\cite[Theorem 3.42]{fluchthesis}
    Let $G$ be a group that is the direct union of $\{G_\lambda\st\lambda\in\Lambda\}$, where
    $\Lambda$ is a countable set of indexes.
    Let $\FF$ and $\FF_\lambda$ be full families of subgroups of $G$ and $G_\lambda$
    for all $\lambda\in\Lambda$, respectively, that are compatible with the direct union.
    Then,
    $$\underset{\lambda\in\Lambda}{\operatorname{sup}}\{\cd_{\FF_\lambda} G_\lambda\}
    \leq \cd_{\FF} G \leq
    \underset{\lambda\in\Lambda}{\operatorname{sup}}\{\cd_{\FF_\lambda} G_\lambda\} + 1.$$
\end{Theorem}

Note that the first inequality is given by Theorem~\ref{Bredon:cd_subgroup}, since
$\FF_\lambda = \FF\cap G_\lambda$ for all $\lambda\in\Lambda$.

In the case of the Bredon geometric dimensions, the following result is proved whithin the
proof of \cite[Theorem 4.3]{lueckweiermann}:

\begin{Theorem}\label{gd_direct_unions}
    Let $G$ be a group that is the direct union of $\{G_\lambda\st\lambda\in\Lambda\}$, where
    $\Lambda$ is a countable set of indexes.
    Let $\FF$ and $\FF_\lambda$ be full families of subgroups of $G$ and $G_\lambda$
    for all $\lambda\in\Lambda$, respectively, that are compatible with the direct union.
    Then,
    $$\underset{\lambda\in\Lambda}{\operatorname{sup}}\{\gd_{\FF_\lambda} G_\lambda\}
    \leq \gd_{\FF} G \leq
    \underset{\lambda\in\Lambda}{\operatorname{sup}}\{\gd_{\FF_\lambda} G_\lambda\} + 1.$$
\end{Theorem}
In this case, the first inequality is given by Theorem~\ref{Bredon:gd_subgroup}.

\begin{Defn}
    Let $\mathfrak{X}$ be a class of groups. We say a group $G$ is \emph{locally} $\mathfrak{X}$ if
    for all finitely generated subgroups $H\leq G$, $H\in\mathfrak{X}$.
\end{Defn}

Using the fact that every group is the direct union of its finitely generated
subgroups and Propositions \ref{G_in_FF_cd_0} and \ref{G_in_FF_gd_0}, we get a first application
of Theorems \ref{cd_direct_unions} and \ref{gd_direct_unions}:

\begin{Proposition}\label{cd_locallyF}\cite[Proposition 3.47]{fluchthesis}
    Let $G$ be a group and $\FF$ a full family of finitely generated subgroups
    of $G$. If $G$ is locally $\FF$ and $G$ is countable, then $$\cd_\FF G \leq 1.$$
\end{Proposition}

\begin{Proposition}\label{gd_locallyF}
    Let $G$ be a group and $\FF$ a full family of finitely generated subgroups
    of $G$. If $G$ is locally $\FF$ and $G$ is countable, then $$\gd_\FF G \leq 1.$$
\end{Proposition}

%% file: ClassForChains.tex
\chapter{Classifying spaces for chains of families of subgroups}\label{ch:chains}
\minitoc

The results presented in Chapter~\ref{ch:related-families}, most importantly those from \cite{lueckweiermann},
have been used fruitfully to build the classifying spaces for the family
of virtually cyclic subgroups of a wide variety of groups from those for the family of finite subgroups.

In \cite{nucinkisetal}, we used these methods to recursively build classifying spaces for the families
$\AA_0 \subseteq \AA_1 \subseteq \ldots \subseteq \AA_n$ of subgroups of bounded torsion-free rank of any
finitely generated abelian group.

The objective of this chapter is to widen the results presented in the aforementioned sources, to
be able to study the Bredon dimensions with respect to families forming an ascending chain with
certain properties.

\section{Strongly structured ascending chains of families of subgroups}

\begin{Defn}
    A chain of families of subgroups $\FF_0 \subseteq \FF_1 \subseteq \ldots \subseteq \FF_r \subseteq \ldots$
    of a group $G$ is said to be a \emph{strongly structured ascending chain of subgroups
    of} $G$ if for all $r,i\in \mathbb{N}$ with $i \leq r$, if $H, K \in \dF{F}{r}$
    are such that $H\cap K \in \dF{F}{r}$ and $L \in \dF{F}{i}$, then
    $L\cap H \in \dF{F}{i}$ if and only if $L\cap K \in \dF{F}{i}.$

    We will use the abbreviation ``SSACFS" for ``strongly structured ascending chains of families of subgroups".
\end{Defn}

\begin{Example}
    Let $G$ be a group and $\FF_n$ be the family of finitely generated abelian
    subgroups of $G$ of torsion-free rank smaller than or equal to $n$. Then,
    $\left(\FF_r\right)_{r\in\mathbb{N}}$ is a strongly structured ascending
    chain of families of subgroups of $G$.

    If $G$ is a finitely generated abelian group, then we can find the construction
    we generalise in this chapter for this particular choice of group and chain of
    families of subgroups in \cite{nucinkisetal}.
\end{Example}

Now let us present a more general example related to the previous one and also to
Example~\ref{bounded_rank_subgroups_of_a_class_is_family}, for which we will need
the concept of commensurability of groups:

\begin{Defn}\label{def_commensurable}
    We say that two groups $G_1$ and $G_2$ are \emph{commensurable} if and only if
    there are finite index subgroups $H_1 \leq G_1$ and $H_2 \leq G_2$ such that
    $H_1 \cong H_2$.

    When we restrict ourselves to subgroups $H, K$ of a given group $G$,
    we say $H$ and $K$ are commensurable if and only if $|H:H\cap K| < \infty$
    and $|K:K\cap H| < \infty$. In that case, we define the
    \emph{commensurator of} $H$ \emph{in} $G$ as the set of all elements $g\in G$
    such that $H$ and $H^g$ are commensurable and we denote it by $\comm_G (H)$.
\end{Defn}

\begin{Lemma}\label{bounded_rank_families_are_SSCFS}
    Let $\mathfrak{X}$ be a non-empty class of groups closed under taking subgroups
    and let $r : \mathfrak{X} \to \N \cup \{\infty\}$ be a rank such that if
    $H,K \in \mathfrak{X}$ are such that there is an injective homomorphism
    $f : H \to K$, then:
    \begin{enumerate}[label={\emph{(\roman*)}}, noitemsep]
        \item $r(\{1\}) = 0$,
        \item $r(H) \leq r(K)$ and
        \item $r(H) = r(K)$ if and only if $H$ and $K$ are commensurable.
    \end{enumerate}
    Let $G$ be a group and for all $n\in \N\cup\{\infty\}$ take
    $$\mathfrak{X}_n (G) = \{ H \leq G \st H \in \mathfrak{X} \text{ and } r(H) \leq n\}.$$
    Then,
    $\left(\mathfrak{X}_n (G) \right)_{n\in\N}$ is a strongly structured ascending chain
    of families of subgroups of $G$.
\end{Lemma}
\begin{proof}
    First, since $\All(G) \cap \mathfrak{X} \neq \emptyset$, and since $\mathfrak{X}$
    is closed under taking subgroups, $\mathfrak{X}_n(G)\neq\emptyset$ for all $n\in\N$ and
    closed under taking subgroups.
    Also, if $H\in\mathfrak{X}_n(G)$ and $g\in G$, since conjugation is an isomorphism, applying
    $(i)$ to conjugation by $g$ and its inverse, we have $r(H^g) = r(H)$.

    That means $\mathfrak{X}_n(G)$ is a full family of subgroups of $G$ for all $n\in\N$.

    And by construction we have $\mathfrak{X}_k(G) \subseteq \mathfrak{X}_n(G)$ for all
    $k\leq n$. So $\left(\mathfrak{X}_n\right)_{n\in\N}$ is an ascending chain of full families
    of subgroups of $G$.

    It only remains to prove that for all $r,i\in \mathbb{N}$ with $i \leq r$, if
    $H, K \in \mathfrak{X}_r(G) \bs \mathfrak{X}_{r-1}(G)$
    are such that $H\cap K \in \mathfrak{X}_r(G) \bs \mathfrak{X}_{r-1}(G)$ and
    $L \in \mathfrak{X}_i(G) \bs \mathfrak{X}_{i-1}(G)$, then
    $L\cap H \in \mathfrak{X}_i(G) \bs \mathfrak{X}_{i-1}(G)$ if and only if
    $L\cap K \in \mathfrak{X}_i(G) \bs \mathfrak{X}_{i-1}(G).$

    Observe that for $H \leq G$, $H\in \mathfrak{X}_r(G) \bs \mathfrak{X}_{r-1}(G)$
    if and only if $r(H) = r$. Therefore, given
    $0 \leq i \leq r$ and given $H,K, L \leq G$ such that $h(H) = h(K) = h(H\cap K) = r$
    and $h(L) = i$, we need to prove that $h(H\cap L) = i$ if and only if $h(K \cap L) = i$.
    By hypothesis $(ii)$, since $L\cap K \leq L$ and $L\cap H \leq L$, it is equivalent to
    prove that $| L : L\cap H | < \infty$ if and only if $| L : L \cap K | < \infty$
    given that $| H : H\cap K| < \infty$ and $| K : H\cap K | < \infty$.

    Assume then that $| H : H\cap K| < \infty$, $| K : H\cap K | < \infty$ and $| L : L \cap H | < \infty$.
    Since $| H :  H\cap K | < \infty$, intersecting with $L$, we obtain
    $| L \cap H : L \cap H \cap K| < \infty$. Hence, since $| L : L\cap H | < \infty$
    and index is multiplicative, $| L : L\cap H \cap K| < \infty$. Finally, since
    $L\cap H \cap K \leq L \cap K\leq L$, we can conclude $| L : L \cap K | < \infty$,
    as we wanted to show.

    The converse implication is symmetrical.
\end{proof}

We will use the results in \cite{lueckweiermann} summarised in section~\ref{sec:LW} of
Chapter~\ref*{ch:related-families}
for each of the inclusions in the chain of families, so we need to start by showing
that we can use such results:

\begin{Defn}\label{def_eqrel_r}
    Given a chain $\left(\FF_r\right)_{r\in\N}$ of full families of subgroups of a group $G$,
    for each $r \in \N$, let $\sim_r$  denote the following relation on $\dF{F}{r}:$
    $$H \sim_r K \iff H\cap K \in \dF{F}{r}$$
\end{Defn}

\begin{Lemma}\label{SSACFS_LW_eqrel}
    If $\left(\FF_r\right)_{r\in\N}$ is a strongly structured ascending chain
    of families of subgroups of a group $G$, then $\FF_{r-1} \subseteq \FF_r$ is a strongly
    structured inclusion of families of subgroups of $G$ with respect to $\sim_r$ for
    every $r>0$, i.e., $\sim_r$ is a strong equivalence relation
    (in the sense of \cite{lueckweiermann}) in $\dF{F}{r}$ for every $r>0$.
\end{Lemma}
\begin{proof}
    $\sim_r$ is clearly reflexive and symmetric. As for transitivity,
    given $H,K,L \in \dF{F}{r}$ such that $H\sim_r K$ and $K \sim_r L$,
    we need to see that $H\sim_r L$. By definition of SSACFS, in the particular
    case that $i = r$, since $H\cap K \in \dF{F}{r}$ and $K\cap L \in \dF{F}{r}$,
    then $H \cap L \in \dF{F}{r}$. Hence, $\sim_r$ is an equivalence relation in $\dF{F}{r}$.

    To prove that it is strong in the sense of L\"uck-Weiermann, first let $H,K \in \dF{F}{r}$
    with $H \leq K$. Then, $H\cap K = H \in \dF{F}{r}$, so $H \sim_r K$.
    Finally, let $g \in G$ and $H,K \in \dF{F}{r}$. We need to see $H \sim_r K$
    if and only if $H^g \sim_r K^g$.
    And for that it suffices to prove that $(H\cap K)^g = H^g \cap K^g$,
    since $\FF_r$ and $\FF_{r-1}$ are families of subgroups of $G$ and, therefore,
    closed under conjugation. An element $l \in G$ belongs to $H^g \cap K^g$ if and only
    if there are $h \in H$ and $k \in K$ such that $l = ghg^{-1} = gkg^{-1}$.
    But in that case $h = k$ and so $l \in (H \cap K)^g$. The other implication follows trivially.
\end{proof}

And not only the equivalence relations $\sim_r$ are strong in the sense of \cite{lueckweiermann},
but they are the finest of all possible strong equivalence relations at each inclusion.

\begin{Lemma}\label{sim_r_is_finest}
    Let $\left(\FF_r\right)_{r\in\N}$ be a strongly structured ascending chain of families
    of subgroups of a group $G$, let $r>0$ and let $\sim$ be any equivalence relation in $\dF{F}{r}$
    that is strong in the sense of \cite{lueckweiermann}. Then $\sim_r$ is finer than $\sim$.
\end{Lemma}
\begin{proof}
    Let $H, K \in \dF{F}{r}$ such that $H \sim_r K$. Then, $H\cap K \in \dF{F}{r}$.
    By $(i)$ in Definition~\ref{def_strong_equiv_LW}, $H\cap K \sim H$ and $H\cap K \sim K$
    given $H\cap K \leq H$ and $H\cap K \leq K$ respectively. Finally, by transitivity
    of $\sim$, that implies $H\sim K$, as we wanted to see.
\end{proof}

Following the example in Lemma~\ref{bounded_rank_families_are_SSCFS}, we have:

\begin{Lemma}\label{eq_rel_for_BRF_is_comm}
    Given a group $G$ and $\left(\mathfrak{X}_n(G)\right)_{n\in\N}$ as in
    Lemma~\ref{bounded_rank_families_are_SSCFS}, for every $r>0$ and
    $H,K \in \mathfrak{X}_r(G) \bs \mathfrak{X}_{r-1}(G)$,
    $H\sim_r K$ if and only if $H$ and $K$ are commensurable.
\end{Lemma}
\begin{proof}
    Given $H,K \in \mathfrak{X}_r(G) \bs \mathfrak{X}_{r-1}(G)$, by definition of
    $\sim_r$, $H\sim_r K$ if and only if $r(H\cap K) = r(H) = r(K) = r$. Since
    commensurability is transitive, $H\cap K \leq H$ and $H\cap K \leq K$, by hypothesis
    $(ii)$, $r(H\cap K) = r(H) = r(K) = r$ if and only if $H$ and $K$ are commensurable,
    as we needed to see.
\end{proof}

\section{Construction}\label{sec:construction_dFFr}

Our objective is to be able to give a bound on the Bredon cohomological and geometric
dimensions of a group over the families in a strongly structured ascending chain
of families of subgroups $\left(\FF_r\right)_{r\in\N}$. For that, in this
section we are going to build classifying spaces for the mentioned families using recursion
over $r$.

By Theorem~\ref{EFG_existence}, models for $\EFG{\FF_r}{G}$ exist for all $r\in\N$, so
for all results related to constructions and dimensions we don't need to construct any
initial spaces. However, as we will be interested in finite-dimensionality,
finite-dimensional models for those basic spaces in the process we will
describe in this section will be a must-have to draw any conclusions that can not be
already drawn from Theorem~\ref{EFG_existence}.

As the first example of the aforementioned basic spaces (i.e. spaces used but not being built
in the recursive process), a finite-dimensional model for $\EFG{\FF_0}{G}$
will be necessary. But let us start the construction process and discover
all elements needed as they appear.

From now on, let $G$ be a discrete group and $(\FF_r)_{r\in\N}$ a strongly structured
ascending chain of families of subgroups of $G$.

\subsection{Set-up and involved spaces}\label{sec:setup_chains}

Let $r > 0$. The first step will consist of applying L\"uck-Weiermann to the families
$\FF_{r-1} \subseteq \FF_r$ of subgroups of $G$.

We start, then, by choosing the strong equivalence relations at each level: by Lemma~\ref{SSACFS_LW_eqrel}, we can take
$\sim_r$ as strong equivalence relation in the sense of LW in $\dF{F}{r}$.

\begin{Defn}
    Given $H\in \dF{F}{r}$, we denote by $[H]_r$ the equivalence class of $H$
    and we denote by $[\dF{F}{r}]$ the set of all equivalence classes. Define
    for all $[H]_r \in [\dF{F}{r}]$ the subgroup of $G$
    $$\Nzerr{r}{G}{H} =\{g \in G\,|\, H^{g}\sim_r H \}$$
    and the family of subgroups of $\Nzerr{r}{G}{H}$
    $$\FF_r[H] =\{K \leq \Nzerr{r}{G}{H} \st K \in \dF{F}{r}
    \, ,\, K\sim_r H\}\cup (\FF_{r-1} \cap \Nzerr{r}{G}{H}).$$
\end{Defn}

We can now apply Theorem~\ref{lw-main}. The spaces involved in the push-out would be a model for
$\EFG{\FF_{r-1}}{G}$ and, for each representative of the equivalence classes under $\sim_r$, models for
$\EFG{\FF_{r-1} \cap \Nzerr{r}{G}{H}}{\Nzerr{r}{G}{H}}$ and $\EFG{\FF_r[H]}{\Nzerr{r}{G}{H}}$. Of the last
two sets of models, we can give an upper bound of the Bredon dimensions of $\Nzerr{r}{G}{H}$ with respect
to $\FF_{r-1} \cap \Nzerr{r}{G}{H}$ in terms of the Bredon dimensions of $G$ with respect to
$\FF_{r-1}$ by Theorem~\ref{Bredon:cd_subgroup} and Proposition~\ref{Bredon:gd_subgroup}. However,
in the case of the family $\FF_r[H]$ of subgroups of $\Nzerr{r}{G}{H}$, we need to find a suitable
classifying space that is related to those for the families $\FF_i$ with $i\leq r-1$.

The family $\FF_r[H]$ is given as a union of the sets of subgroups of $\Nzerr{r}{G}{H}$ $[H]_r$ and
$\FF_{r-1} \cap \Nzerr{r}{G}{H}$. Note that $\FF_{r-1}\cap \Nzerr{r}{G}{H}$ is a family of subgroups
of $\Nzerr{r}{G}{H}$ but $[H]_r = \{K \leq \Nzerr{r}{G}{H} \st H \sim_r K \}$ is not closed under
taking subgroups. To complete $[H]_r$, we can add those subgroups in $\FF_{r-1}\cap \Nzerr{r}{G}{H}$
that are related to a subgroup of $H$. This way, and using the structure on the chain of families
$\left(\FF_r\right)_{r\in \N}$, we will be able to build the spaces needed to use Lemma~\ref{E_FUG}.

\begin{Defn}\label{defRS} Let $r>0$,  $H \in \dF{F}{r}$ and $0 < i \leq r$. Then, we define $\RS{i}{G}{H}$
to be union between $\FF_0 \cap \Nzerr{r}{G}{H}$ and the set of subgroups $K$ of $\Nzerr{r}{G}{H}$ in
$\FF_i$ such that there is $L\leq H$ with $L\in\dF{F}{j}$ and $L\sim_j K$ for some $0 < j \leq i$.

In the case $i=0$, we define $$\RS{0}{G}{H} = \FF_0 \cap \Nzerr{r}{G}{H}.$$
\end{Defn}

\begin{Lemma}
    Let $r>0$,  $H \in \dF{F}{r}$ and $0 < i \leq r$. Then,
    $$\RS{i}{G}{H} = \{ K \in \FF_i \cap \Nzerr{r}{G}{H} \st
    K \sim_j K\cap H \text{ with } 0 < j \leq i \}\cup (\FF_0 \cap \Nzerr{r}{G}{H}).$$
\end{Lemma}

\begin{proof}
    Let $K \leq \Nzerr{r}{G}{H}$ such that $K \in \dF{F}{j}$ for $0 < j \leq i$.
    For the first inclusion, assume there is $L\leq H$ such that $L \sim_j K$. We need
    to see that then $K\sim_j K\cap H$. So we need to see $K\cap H \in \dF{F}{j}$. Since
    $K \in \FF_j$, so is $K\cap H$. Then $K\cap H \notin \dF{F}{j}$ would be equivalent
    to $K\cap H \in \FF_{j-1}$. Since $\FF_{j-1}$ is closed under taking subgroups and
    $L\cap K \leq K \cap H$, we would have $L\cap K \in \FF_{j-1}$, which is a
    contradiction with the fact that $L\sim_j K$. Therefore, $K\cap H \in \dF{F}{j}$
    as we wanted to see.

    The other inclusion is direct taking $L = K\cap H$ as the subgroup of $H$ involved in
    the definition of $\RS{i}{G}{H}$.
\end{proof}

And now we need to see that $\RS{i}{G}{H}$ is well-defined in terms of the classes in $[\dF{F}{r}]$ and
that $\RS{i}{G}{H}$ is a family of subgroups of $\Nzerr{r}{G}{H}$ for all $0 \leq i \leq r$:

\begin{Lemma}
    For $H,H' \in \dF{F}{r}$ with $H\sim_r H'$, $\RS{i}{G}{H} =\RS{i}{G}{H'}$.
\end{Lemma}
\begin{proof}
    Since $H\sim_r H'$, we know $\Nzerr{r}{G}{H} = \Nzerr{r}{G}{H'}$. Hence,
    $\RS{0}{G}{H} =\RS{0}{G}{H'}$. Let now $0 < i \leq r$.

    Given the symmetry of $\sim_r$, we only need to prove one inclusion. Let
    $K\in \RS{i}{G}{H}$ and let $0 < j \leq i$ such that $K \sim_j K \cap H$.
    We know $H\sim_r H'$, and hence, since $\left(\FF_r \right)_{r\in\N}$ is a
    SSACFS of $G$, we have $K \sim_j K\cap H'$, as we needed to prove.

    Note that if there is no such $j$ it means $K \in \FF_0 \cap \Nzerr{r}{G}{H}$,
    and therefore $K$ will belong to $\RS{i}{G}{H'}$ since $\Nzerr{r}{G}{H} = \Nzerr{r}{G}{H'}$
    and $\RS{0}{G}{H} \subseteq \RS{i}{G}{H}$.
\end{proof}

\begin{Lemma}\label{RrGH-fam}
    Let $r>0$, $H \in \dF{F}{r}$ and $i\leq r$. Then, $\RS{i}{G}{H}$ is a family
    of subgroups of $\Nzerr{r}{G}{H}$.
\end{Lemma}
\begin{proof}
    Firstly, $\{1\}\in \RS{i}{G}{H} \neq \emptyset$. To prove closure under
    taking subgroups and conjugation, we will assume $K\notin \FF_0 \cap \Nzerr{r}{G}{H}$,
    since otherwise the statements reduce to $\FF_0\cap \Nzerr{r}{G}{H}$ being a family,
    which we already know is true. Let $K \leq \Nzerr{r}{G}{H}$ with $K \in \dF{F}{j}$
    for some $0 < j \leq i$ such that $K\sim_j K\cap H$.

    Let $L\leq K$ and $0 < l \leq j$ such that $L\in\dF{F}{l}$. Since $K\sim_j K \cap H$,
    since $\left(\FF_r \right)_{r\in\N}$ is a SSACFS of $G$, $L\cap K \in \dF{F}{l}$
    if and only if $L\cap K \cap H \in \dF{F}{l}$. But $L\cap K = L \in \dF{F}{l}$
    by hypothesis, so we can conclude that $\RS{i}{G}{H}$ is closed under taking
    subgroups, as $L\sim_l L\cap H$.

    Now if $g\in \Nzerr{r}{G}{H}$, we need to prove $K^g \in \RS{i}{G}{H}$. Since
    $K\in \dF{F}{j}$ and both $\FF_j$ and $\FF_{j-1}$ are families, $K^g\in \dF{F}{j}$.
    For the same reason, since $K\sim_j K\cap H$, $(K\cap H)^g = K^g \cap H^g \in \dF{F}{j}$.
    And because $g\in \Nzerr{r}{G}{H}$ we have $H\sim_r H^g$. Therefore, since
    $\left(\FF_r \right)_{r\in\N}$ is a SSACFS of $G$, as we have $K^g\sim_j K^g \cap H^g$
    and $H\sim_r H^g$, we know $K^g \sim_j K^g \cap H$. This proves $K^g \in \RS{i}{G}{H}$,
    as we needed to see.
\end{proof}

Finally, we express the family $\FF[H]$ of subgroups of $\Nzerr{r}{G}{H}$ as a union of families
and find their intersection to prove that we can use the results for unions of families described in
Chapter~\ref{ch:related-families}:

\begin{Lemma}\label{ingrforE_FUG} Let $r>0$ and $H\in\dF{F}{r}$. Then, the following hold:
    \begin{enumerate}[label={\emph{(\roman*)}}]
        \item $\FF_r[H] = \RS{r}{G}{H} \cup \left( \FF_{r-1} \cap \Nzerr{r}{G}{H}\right)$
        \item $\RS{r-1}{G}{H} = \RS{r}{G}{H} \cap \left( \FF_{r-1} \cap \Nzerr{r}{G}{H}\right)$
    \end{enumerate}
\end{Lemma}
\begin{proof}
    \begin{enumerate}[label={(\roman*)}]
        \item Since $\{K \leq \Nzerr{r}{G}{H} \st K \in \dF{F}{r} \, ,\, H\sim_r K\} \subseteq \RS{r}{G}{H}$,
        $\FF_r[H] \subseteq \RS{r}{G}{H} \cup \left( \FF_{r-1} \cap \Nzerr{r}{G}{H}\right)$.
        For the other inclusion, let $K\in \RS{r}{G}{H}$. If $K\in \FF_0$ or $K\in \dF{F}{i}$
        for some $0 < i < r$, then $K\in \FF_{r-1} \cap \Nzerr{r}{G}{H}$. If $K\in \dF{F}{r}$,
        since $K\sim_r K\cap H$ is equivalent to $K\sim_r H$, we have
        $K \leq \{K \leq \Nzerr{r}{G}{H} \st K \in \dF{F}{r} \, ,\, H\sim_r K\} \subseteq \FF_r[H]$.

        \item Let $K \in \RS{r}{G}{H} \cap \left( \FF_{r-1} \cap \Nzerr{r}{G}{H}\right)$.
        In particular, $K\in \FF_{r-1}$. If $K \in \FF_0$, we are done, since
        $\FF_0 \cap \Nzerr{r}{G}{H} \subseteq \RS{i}{G}{H}$ for all $0 \leq i \leq r$.
        Let $i$ be such that $0< i \leq r-1$ and $K \in \dF{F}{i}$. And since
        $K\in \RS{r}{G}{H}$, $K\sim_i K\cap H$, so $K\in \RS{r-1}{G}{H}$, since $i\leq r-1$.
        The other inclusion follows directly from $\RS{r-1}{G}{H}$ being a subset
        of both $\RS{r}{G}{H}$ and $\FF_{r-1} \cap \Nzerr{r}{G}{H}$.
    \end{enumerate}
\end{proof}

Therefore, let us first build models for $\EFG{\RS{i}{G}{H}}{\Nzerr{r}{G}{H}}$.

\subsection{Classifying spaces for the families $\left(\RS{i}{G}{H}\right)_{i=0}^r$}\label{sec:class_for_RS}

In this section, we will take advantage of the structure that the families
$\left(\RS{i}{G}{H}\right)_{i=0}^r$ inherit from $(\FF_r)_{r\in\N}$ to build recursively
models for $\EFG{\RS{i}{G}{H}}{\Nzerr{r}{G}{H}}$, where $H\in\dF{F}{r}$.

For the base case, note that $\RS{0}{G}{H} = \FF_0 \cap \Nzerr{r}{G}{H}$. Hence, by
Theorem~\ref{Bredon:cd_subgroup} and Proposition~\ref{Bredon:gd_subgroup},
the Bredon dimensions of $\Nzerr{r}{G}{H}$ with respect to the family $\RS{0}{G}{H}$
will be finite in the case $\cd_{\FF_0} G$ and $\gd_{\FF_0} G$ are.

Let now $0 < i \leq r$. We want to use L\"uck-Weiermann method on the families
$\RS{i-1}{G}{H} \subseteq \RS{i}{G}{H}$.

\begin{Lemma}\label{RSisSSACFS}
    $\left(\RS{i}{G}{H}\right)_{i=0}^r$ is a SSACFS of $\Nzerr{r}{G}{H}$.
\end{Lemma}
\begin{proof}
    Let $0 \leq i \leq j \leq r$. Let $K, L \in \dRS{j}{G}{H}$ such that
    $K \cap L \in \dRS{j}{G}{H}$. Assume $M \in \dRS{i}{G}{H}$. We need to
    see $M\cap K\in \dRS{i}{G}{H}$ if and only if $M \cap L \in \dRS{i}{G}{H}$.
    Given the symmetry of intersection, we only need to prove one
    of the implications. Let $M$ then be such that $M\cap K \in \dRS{i}{G}{H}$.
    By definition of $\RS{i}{G}{H}$ and $\RS{i-1}{G}{H}$, that happens if and only if
    $M\cap K \in \dF{F}{i}$ and $M\cap K \cap H \in \dF{F}{i}$.
    We need to see $M \cap L, M \cap L \cap H \in \dF{F}{i}$.

    Since $K, L, K \cap L \in \dRS{j}{G}{H}$, we know
    $K\cap H, L \cap H, K \cap L\cap H \in \dF{F}{j}$.

    The chain $(\FF_r)_{r\in\N}$ is a SSACFS, $K,L,K\cap L \in \dF{F}{j}$ and
    $M,M\cap L \in \dF{F}{i}$, therefore $M\cap L \in \dF{F}{i}$.

    Analogously, since $K\cap H, L\cap H, K\cap L \cap H \in \dF{F}{j}$ and
    $M\cap K \cap H \in \dF{F}{i}$, we have $M\cap L \cap H \in \dF{F}{i}$,
    as we needed to see.
\end{proof}

\begin{Corollary}\label{sim_i_rest_to_dRS}
    The restriction of $\sim_i$ to $\dRS{i}{G}{H} \subseteq \dF{F}{i}$ is
    a strong equivalence relation for $0 < i \leq r$, and in particular
    it is the same equivalence relation than that defined from the chain
    being a SSACFS.
\end{Corollary}

\begin{proof}
    It is a result of Lemma~\ref{RSisSSACFS} and the fact that
    $\dRS{i}{G}{H} \subseteq \dF{F}{i}$.
\end{proof}

Note that if $i=r$, then $\dRS{r}{G}{H} = [H]_r$ and hence $[\dRS{r}{G}{H}]$
consists solely on the class of $H$ under the restriction of $\sim_r$ to
$\dRS{r}{G}{H}$ (which coincides with $[H]_r)$. This means the family $\RS{r}{G}{H}[H]$
would coincide with $\RS{r}{G}{H}$ and hence the classifying space that we
want to obtain as a result of a push-out would appear in the diagram as one
of the necessary classifying spaces. Also, by Lemma~\ref{sim_r_is_finest} and
Corollary~\ref{sim_i_rest_to_dRS}, the restriction of $\sim_r$ to $\dRS{r}{G}{H}$
is the finest equivalence relation that is strong in the sense of LW. That means that
utilizing any other strong equivalence in $\dRS{r}{G}{H}$ would always result in $[H]_r$
as the only equivalence class. For that reason,
it is not possible to bound the Bredon dimensions of $\Nzerr{r}{G}{H}$ with
respect to $\RS{r}{G}{H}$ using LW method on $\RS{r-1}{G}{H} \subseteq \RS{r}{G}{H}$.

We encountered now a set of the basic spaces that we mentioned at the beginning
of Section~\ref{sec:construction_dFFr}. And these will be the last of such spaces,
so we can now summarize this information:

\begin{Observation}\label{necessary_spaces}
    \begin{enumerate}[label={(\arabic*)}, noitemsep]
        \item No upper bound for the Bredon dimensions of $G$ with respect to
        $\FF_0$ can be deduced from the construction we are describing, since we
        are using this family as the base case for our recursive process.
        \item For each $k > 0$ and each class $[K]_k$ with respect to $\sim_k$,
        we won't be able, in general, to provide a bound for the Bredon dimensions of
        $\Nzerr{k}{G}{K}$ with respect to $\RS{k}{G}{K}$ as part of the
        recursive process we are providing. We will be able, however, under
        certain additional conditions.
        \item To prove that the Bredon dimensions of $G$ with respect to
        $\FF_r$ are finite (and to give a finite upper bound for them) using
        the recursive construction we are providing, finite-dimensional
        models for $\EFG{\FF_0}{G}$ and for $\EFG{\RS{k}{G}{K}}{\Nzerr{k}{G}{K}}$
        for all $K\in\dRS{k}{G}{H}$ and $0 < k \leq r$ are required.
    \end{enumerate}
\end{Observation}

Let hence $i\in\N$ such that $0 < i < r$. Let $\mathcal{K}_i$ be a set of representatives
$[K]_i$ for the equivalence classes in $[\dRS{i}{G}{H}]$.

\begin{Lemma}\label{norm_of_the_norm}
    $$\Nzerr{i}{\Nzerr{r}{G}{H}}{K} = \Nzerr{r}{\Nzerr{i}{G}{K}}{H} =
    \Nzerr{i}{G}{K}\cap \Nzerr{r}{G}{H}.$$
\end{Lemma}
\begin{proof}
    We only need to prove that $\Nzerr{r}{\Nzerr{i}{G}{K}}{H} =
    \Nzerr{i}{G}{K}\cap \Nzerr{r}{G}{H}$, since intersection is symmetric.

    Let $g \in \Nzerr{r}{\Nzerr{i}{G}{K}}{H}$. That happens if and only if $g$ is an element
    of $\Nzerr{i}{G}{K}$ such that $H^g \sim_r H$. Since $\Nzerr{i}{G}{K} \leq G$,
    $H^g \sim_r H$ if and only if $g \in \Nzerr{r}{G}{H}$, which completes the proof.
\end{proof}

Therefore, we can express the family
$\RS{i-1}{G}{H} \cap \Nzerr{i}{\Nzerr{r}{G}{H}}{K}$ in a simpler way that will
help on the computation of dimensions just by considering $\Nzerr{i}{G}{K}$ the
ambient group. Thus, we have

\begin{Remark}\label{changeofambient}
    $\RS{j}{G}{H} \cap \Nzerr{i}{\Nzerr{r}{G}{H}}{K} = \RS{j}{\Nzerr{i}{G}{K}}{H}$
    for all $j\leq i$.
\end{Remark}

Then, as a consequence of Theorem~\ref{lw-main}, Lemma~\ref{RSisSSACFS}, Corollary~\ref{sim_i_rest_to_dRS},
Lemma~\ref{SSACFS_LW_eqrel} and Remark~\ref{changeofambient}, we have

\begin{Corollary}\label{modelforRSi} In the configuration described above,
the $\Nzerr{r}{G}{H}$-CW-complex $Y$ given by the $\Nzerr{r}{G}{H}$-pushout
$$\begin{tikzcd}
\underset{[K]_i\in \mathcal{K}_i}{\bigsqcup}
    \Nzerr{r}{G}{H}
    \times_{\Nzerr{i}{\Nzerr{r}{G}{H}}{K}}
    \EFG{\RS{i-1}{\Nzerr{i}{G}{K}}{H}}{\Nzerr{i}{\Nzerr{r}{G}{H}}{K}}
\arrow[r, "\iota"]
\arrow[dd, "\underset{[K]_i\in \mathcal{K}_i}{\sqcup} id_{\Nzerr{r}{G}{H}} \times_{\Nzerr{i}{\Nzerr{r}{G}{H}}{K}}g_{[K]_i}"]
& \EFG{\RS{i-1}{G}{H}}{\Nzerr{r}{G}{H}} \arrow[dd] \\
\\
\underset{[K]_i\in \mathcal{K}_i}{\bigsqcup}
    \Nzerr{r}{G}{H}
    \times_{\Nzerr{i}{\Nzerr{r}{G}{H}}{K}}
    \EFG{\RS{i}{G}{H}[K]}{\Nzerr{i}{\Nzerr{r}{G}{H}}{K}}
\arrow[r]
& Y
\end{tikzcd}$$
is a model for $\EFG{\RS{i}{G}{H}}{\Nzerr{r}{G}{H}}$ if either $\iota$ is an inclusion or $g_{[K]_i}$ are inclusions for all $[K]_i \in \mathcal{K}_i$.
\end{Corollary}

By Theorem~\ref{Bredon:cd_subgroup} and Proposition~\ref{Bredon:gd_subgroup}, the
Bredon dimensions of $\Nzerr{i}{\Nzerr{r}{G}{H}}{K}$ with respect to the families
$\RS{i-1}{\Nzerr{i}{G}{K}}{H}$ can be bounded by those of $\Nzerr{r}{G}{H}$ with
respect to $\RS{i-1}{G}{H}$.

We want now to be able to relate the models for $\EFG{\RS{i}{G}{H}[K]}{\Nzerr{i}{\Nzerr{r}{G}{H}}{K}}$
with those appearing naturally in the recursive process we are describing. Note that we can
write the family $\RS{i}{G}{H}[K]$ as a union of full families:

\begin{Lemma}\label{RSKiUnion}
    $\RS{i}{G}{H}[K] = \RS{i}{\Nzerr{r}{G}{H}}{K} \cup \RS{i-1}{\Nzerr{i}{G}{K}}{H}.$
\end{Lemma}
\begin{proof}
    From Definition \ref{def_strong_equiv_LW} and using Remark~\ref{changeofambient},
    we have $\RS{i}{G}{H}[K] = \{L\leq \Nzerr{i}{\Nzerr{r}{G}{H}}{K}
    \st L \in \dRS{i}{G}{H} \, , \, L \sim_i K \}\cup \RS{i-1}{\Nzerr{i}{G}{K}}{H}$.
    It suffices to prove $\{L\leq \Nzerr{i}{\Nzerr{r}{G}{H}}{K} \st L \in \dRS{i}{G}{H}
    \, , \, L \sim_i K \} \subseteq \RS{i}{\Nzerr{r}{G}{H}}{K}$ and $\RS{i}{\Nzerr{r}{G}{H}}{K}
    \subseteq \RS{i}{G}{H}[K]$.

    For the first inclusion, let $M \leq \Nzerr{i}{\Nzerr{r}{G}{H}}{K}$ be such that
    $M \in \dRS{i}{G}{H}$ and $M\sim_i K$. In particular, $M \in \dF{F}{i}$ and
    $M \sim_i K$, so $M\in \RS{i}{G}{H} \cap \Nzerr{i}{\Nzerr{r}{G}{H}}{K}$. And by
    Remark~\ref{changeofambient}, $\RS{i}{G}{H} \cap \Nzerr{i}{\Nzerr{r}{G}{H}}{K}
    = \RS{i}{\Nzerr{r}{G}{H}}{K}$.

    For the second inclusion, let $M\in\RS{i}{\Nzerr{r}{G}{H}}{K}$ (which means
    $M \leq \Nzerr{i}{\Nzerr{r}{G}{H}}{K}$) and for some $j\leq i$, $M\in \dF{F}{j}$
    and $M \sim_j M\cap K$. If $j = i$ then $M\sim_i K$, so
    $M \in \{L\leq \Nzerr{i}{\Nzerr{r}{G}{H}}{K} \st L \in \dRS{i}{G}{H}
    \, , \, L \sim_i K \}$. Assume now $j < i$. Since $\RS{i}{G}{H}$ is a family
    of subgroups and $K\in\RS{i}{G}{H}$, so will be $M\cap K$. Moreover, since
    $M\sim_j M\cap K$, we also have $M \in \RS{i}{G}{H}$. But $j < i$, so we can
    conclude that $M \in \RS{i-1}{G}{H}\cap \Nzerr{i}{\Nzerr{r}{G}{H}}{K}$, concluding our proof.
\end{proof}

\begin{Lemma}\label{RSKiInter}
    $\RS{i}{\Nzerr{r}{G}{H}}{K} \cap \RS{i-1}{\Nzerr{i}{G}{K}}{H} = \RS{i-1}{\Nzerr{r}{G}{H}}{K}.$
\end{Lemma}
\begin{proof}

    We need to prove both inclusions.

    For the first one, take $M \in \RS{i}{\Nzerr{r}{G}{H}}{K} \cap \RS{i-1}{\Nzerr{i}{G}{K}}{H}$.
    In particular, $M\in \FF_{i-1}\cap \Nzerr{i}{\Nzerr{r}{G}{H}}{K}$ and
    $M\sim_j M\cap K$ for some $j \leq i - 1$ (it is so for some $j\leq i$,
    but since $M\in \FF_{i-1}$, $j\leq i-1$). That means $M\in \RS{i-1}{\Nzerr{r}{G}{H}}{K}$.

    For the converse inclusion, let $M\in \RS{i-1}{\Nzerr{r}{G}{H}}{K}$. That means
    $M\in \FF_i\cap \RS{i-1}{\Nzerr{i}{G}{K}}{H}$ and $M\sim_j M\cap K$ for some $j\leq i-1$.
    It only remains to prove $M \sim_j M \cap H$, but that is consequence of
    Lemma~\ref{RSisSSACFS} and $K\sim_i K\cap H$.
\end{proof}

\begin{Corollary}\label{gdRSKi}
    Let $H\in \dF{F}{r}$ for $r>0$ and the chain of families $\left(\RS{i}{G}{H}\right)_{i=0}^r$
    defined in \ref{defRS}. Let $K\in \dRS{i}{G}{H}$ for $0 < i < r$. Then,
    \begin{equation*}
        \begin{aligned}
            \gd_{\RS{i}{G}{H}[K]} (\Nzerr{i}{\Nzerr{r}{G}{H}}{K}) \leq
            \operatorname{max}\{ & \gd_{\RS{i}{\Nzerr{r}{G}{H}}{K}} (\Nzerr{i}{\Nzerr{r}{G}{H}}{K})\, , \\
            & \gd_{\RS{i-1}{\Nzerr{i}{G}{K}}{H}} (\Nzerr{i}{\Nzerr{r}{G}{H}}{K})\, , \\
            & \gd_{\RS{i-1}{\Nzerr{r}{G}{H}}{K}} (\Nzerr{i}{\Nzerr{r}{G}{H}}{K}) + 1 \}.
        \end{aligned}
    \end{equation*}
\end{Corollary}
\begin{proof}
    It is a consequence of Lemmas \ref{RSKiUnion}, \ref{RSKiInter} and \ref{E_FUG}.
\end{proof}

\begin{Corollary}\label{cdRSKi}
    Let $H\in \dF{F}{r}$ for $r>0$ and the chain of families $\left(\RS{i}{G}{H}\right)_{i=0}^r$
    defined in \ref{defRS}. Let $K\in \dRS{i}{G}{H}$ for $0 < i < r$. Then,
    \begin{equation*}
        \begin{aligned}
            \cd_{\RS{i}{G}{H}[K]} (\Nzerr{i}{\Nzerr{r}{G}{H}}{K}) \leq
            \operatorname{max}\{ & \cd_{\RS{i}{\Nzerr{r}{G}{H}}{K}} (\Nzerr{i}{\Nzerr{r}{G}{H}}{K})\, , \\
            & \cd_{\RS{i-1}{\Nzerr{i}{G}{K}}{H}} (\Nzerr{i}{\Nzerr{r}{G}{H}}{K})\, , \\
            & \cd_{\RS{i-1}{\Nzerr{r}{G}{H}}{K}} (\Nzerr{i}{\Nzerr{r}{G}{H}}{K}) + 1 \}.
        \end{aligned}
    \end{equation*}
\end{Corollary}
\begin{proof}
    It is a consequence of Lemmas \ref{RSKiUnion}, \ref{RSKiInter} and \ref{cdE_FUG}.
\end{proof}

Finally, we can draw conclusions on the Bredon dimensions corresponding to all
groups and families of subgroups appearing in the push-out in Corollary~\ref{modelforRSi}.

\begin{Corollary}\label{gdRSi_first}
    Let $G$ be a group and $\left( \FF_n \right)_{n\in\N}$ a SSACFS of $G$. Let
    $H\in \dF{F}{r}$ and the chain of families $\left(\RS{i}{G}{H}\right)_{i=0}^r$
    defined in \ref{defRS}. Then, if $0 < i < r$,
    \begin{equation*}
        \begin{aligned}
            \gd_{\RS{i}{G}{H}} \Nzerr{r}{G}{H} \leq
                \underset{[K]_i \in \mathcal{K}}{\operatorname{max}}\{
                & \gd_{\RS{i-1}{\Nzerr{i}{G}{K}}{H}}(\Nzerr{i}{\Nzerr{r}{G}{H}}{K}) + 1 \, , \\
                & \gd_{\RS{i}{G}{H}[K]_i}(\Nzerr{i}{\Nzerr{r}{G}{H}}{K}) \, , \\
                & \gd_{\RS{i-1}{G}{H}}(\Nzerr{r}{G}{H}) \}.
        \end{aligned}
    \end{equation*}
\end{Corollary}
\begin{proof}
    It follows from Corollaries \ref{modelforRSi} and \ref{lw-dimensions}.
\end{proof}

\begin{Corollary}\label{cdRSi_first}
    Let $G$ be a group and $\left( \FF_n \right)_{n\in\N}$ a SSACFS of $G$. Let
    $H\in \dF{F}{r}$ and the chain of families $\left(\RS{i}{G}{H}\right)_{i=0}^r$
    defined in \ref{defRS}. Then, if $0 < i < r$,
    \begin{equation*}
        \begin{aligned}
            \cd_{\RS{i}{G}{H}} \Nzerr{r}{G}{H} \leq
                \underset{[K]_i \in \mathcal{K}}{\operatorname{max}}\{
                & \cd_{\RS{i-1}{\Nzerr{i}{G}{K}}{H}}(\Nzerr{i}{\Nzerr{r}{G}{H}}{K}) + 1 \, , \\
                & \cd_{\RS{i}{G}{H}[K]_i}(\Nzerr{i}{\Nzerr{r}{G}{H}}{K}) \, , \\
                & \cd_{\RS{i-1}{G}{H}}(\Nzerr{r}{G}{H}) \}.
        \end{aligned}
    \end{equation*}
\end{Corollary}
\begin{proof}
    It follows from Corollaries \ref{modelforRSi} and \ref{lw-dimensions}.
\end{proof}

Note that we now have upper bounds for the Bredon dimensions related to all
families of subgroups appearing in Corollary~\ref{modelforRSi} in terms of
those of classifying spaces for which the recursive process gives upper bounds
on the Bredon dimensions or that are listed in Observation~\ref{necessary_spaces}.
This means that if we have finite-dimensional classifying spaces for the families
in Observation~\ref{necessary_spaces}, we will have finite-dimensional classifiying
spaces for the families $\RS{i}{G}{H}$ of subgroups of $\Nzerr{r}{G}{H}$ for $0 \leq i \leq r$,
but a more specific result will be given in Section~\ref{sec:bredon_dim_chains}.

\subsection{Classifying spaces for the families $\left(\FF_r\right)_{r\in\N}$}

We have now all necessary ingredients to tackle the recursive construction process for
the chain $\left(\FF_r\right)_{r\in\N}$.

\begin{Proposition}\label{modelforFr}
    Let $r>0$ and take $\sim_r$ and $[H]_r$, $\Nzerr{r}{G}{H}$ and $\FF_r[H]$ for $H\in\dF{F}{r}$
    as defined in Section~\ref{sec:setup_chains}. Let $\mathcal{H}_r$ be a set of representatives
    for the equivalence classes in $[\dF{F}{r}]$. Then, the $G$-CW-complex $X$ given by the $G$-pushout
    $$\begin{tikzcd}
        \underset{[H]_r \in \mathcal{H}_r}{\bigsqcup} G
            \times_{\Nzerr{r}{G}{H}}
            \EFG{\FF_{r-1} \cap \Nzerr{r}{G}{H}}{\Nzerr{r}{G}{H}}
        \arrow[r, "\iota"]
        \arrow[dd, "\underset{[H]_r\in \mathcal{H}_r}{\sqcup} id_G \times_{\Nzerr{r}{G}{H}}f_{[H]_r}"]
        & \EFG{\FF_{r-1}}{G} \arrow[dd] \\
        \\
        \underset{[H]_r \in \mathcal{H}_r}{\bigsqcup} G
            \times_{\Nzerr{r}{G}{H}}
            \EFG{\FF_r[H]}{\Nzerr{r}{G}{H}}
        \arrow[r]
        & X
    \end{tikzcd}$$
    is a model for $\EFG{\FF_r}{G}$ if either $\iota$ is an inclusion or $f_{[H]}$ are inclusions for all $[H]_r \in \mathcal{H}_r$.
\end{Proposition}
\begin{proof}
    Consequence of Lemma~\ref{SSACFS_LW_eqrel} and Theorem~\ref{lw-main}.
\end{proof}

And in the case of the spaces in the left-bottom corner of the previous diagram, we have:

\begin{Corollary}\label{gdFrH}
    If $r>0$ and $H\in \dF{F}{r}$,
    \begin{equation*}
        \begin{aligned}
            \gd_{\FF_r[H]} \Nzerr{r}{G}{H} \leq \operatorname{max}\{
            & \gd_{\FF_{r-1}\cap \Nzerr{r}{G}{H}} \Nzerr{r}{G}{H}\, , \,
            \gd_{\RS{r}{G}{H}} \Nzerr{r}{G}{H}\, , \\
            & \gd_{\RS{r-1}{G}{H}} \Nzerr{r}{G}{H} + 1 \}.
        \end{aligned}
    \end{equation*}
\end{Corollary}
\begin{proof}
    The proof follows from Lemmas \ref{RrGH-fam}, \ref{ingrforE_FUG} and \ref{E_FUG}.
\end{proof}

\begin{Corollary}\label{cdFrH}
    If $r>0$ and $H\in \dF{F}{r}$,
    \begin{equation*}
        \begin{aligned}
            \cd_{\FF_r[H]} \Nzerr{r}{G}{H} \leq \operatorname{max}\{
            & \cd_{\FF_{r-1}\cap \Nzerr{r}{G}{H}} \Nzerr{r}{G}{H}\, , \,
            \cd_{\RS{r}{G}{H}} \Nzerr{r}{G}{H}\, , \\
            & \cd_{\RS{r-1}{G}{H}} \Nzerr{r}{G}{H} + 1 \}.
        \end{aligned}
    \end{equation*}
\end{Corollary}
\begin{proof}
    The proof follows from Lemmas \ref{RrGH-fam}, \ref{ingrforE_FUG} and \ref{cdE_FUG}.
\end{proof}

\section{Bredon dimensions and Mayer-Vietoris sequences}\label{sec:bredon_dim_chains}

In this section, we will present the results regarding upper bounds for
Bredon cohomological and geometric dimensions of the elements involved
in the construction described in the previous section.

For the following results, let $G$ be a group and $(\FF_r)_{r\in\N})$
a strongly structured ascending chain of families of subgroups of $G$. Let also
the equivalence relations $\sim_r$ as defined in Definition~\ref{def_eqrel_r}
and for each $r\in\N$ let $\mathcal{H}_r$ be a set of representatives of
the classes under $\sim_r$. For each $H\in\mathcal{H}_r$, let the strongly
structured ascending chain $(\RS{i}{G}{H})_{i = 0}^r$ of subgroups of
$\Nzerr{r}{G}{H}$ be as defined in Definition~\ref{defRS}. Take for every
$H\in\mathcal{H}_r$ and every $0 < i < r$ the set of representatives $\mathcal{K}_i(H)$
of the classes under $\sim_i$ restricted to $\dRS{i}{G}{H}$, i.e.,
$\mathcal{K}_i(H) = \mathcal{H}_i \cap \RS{i}{G}{H}$.

\subsection{Bredon geometric dimensions}

In the case of the auxiliary chain of families $\left(\RS{i}{G}{H}\right)_{i=0}^r$
of subgroups of $\Nzerr{r}{G}{H}$, we have:

\begin{Proposition}\label{gdRSi}
    Let $H \in \dF{F}{r}$ for $r>0$ and $0 < i < r$. Then,
    \begin{equation*}
        \begin{aligned}
            \gd_{\RS{i}{G}{H}} (\Nzerr{r}{G}{H} ) \leq
            \underset{[K]_i \in \mathcal{K}_i(H)}{\operatorname{max}} \{
            & \gd_{\RS{i-1}{\Nzerr{i}{G}{K}}{H}} (\Nzerr{i}{\Nzerr{r}{G}{H}}{K} ) + 1 \, , \\
            & \gd_{\RS{i-1}{\Nzerr{r}{G}{H}}{K}} (\Nzerr{i}{\Nzerr{r}{G}{H}}{K} ) + 1 \, , \\
            & \gd_{\RS{i}{\Nzerr{r}{G}{H}}{K}} (\Nzerr{i}{\Nzerr{r}{G}{H}}{K} )  \, , \\
            & \gd_{\RS{i-1}{G}{H}} (\Nzerr{r}{G}{H} )  \}.
        \end{aligned}
    \end{equation*}
\end{Proposition}

\begin{proof}
    From Corollaries \ref{gdRSi_first} and \ref{gdRSKi}.
\end{proof}

And this, as we mentioned in the last paragraph of Section~\ref{sec:class_for_RS}, implies
the following result about finite-dimensionality:

\begin{Corollary}\label{finite_gd_RSi}
    Let $H \in\dF{F}{r}$ for $r>0$ and $0 < i < r$. If there are finite-dimensional
    models for $\EFG{\RS{i}{\Nzerr{r}{G}{H}}{K}}{\Nzerr{i}{\Nzerr{r}{G}{H}}{K}}$ for
    every $0 < i < r$ and $K\in\mathcal{K}_i(H)$ and there is a finite-dimensional model
    for $\EFG{\FF_0}{G}$, then there is a finite-dimensional model for
    $\EFG{\RS{i}{G}{H}}{\Nzerr{r}{G}{H}}$ for all $i \in\N$ with $0 \leq i < r$.
\end{Corollary}
\begin{proof}
    It is the result of Proposition~\ref{gdRSi} and recursion over $i$.
\end{proof}

Analogously, for the chain $\left(\FF_r\right)_{r\in\N}$:

\begin{Proposition}\label{gdFr}
    Let $H \in \dF{F}{r}$ for $r>0$. Then,
    \begin{equation*}
        \begin{aligned}
            \gd_{\FF_r} G \leq \underset{[H]_r \in \mathcal{H}_r}{\operatorname{max}} \{
            & \gd_{\FF_{r-1} \cap \Nzerr{r}{G}{H}} (\Nzerr{r}{G}{H}) + 1 \, , \\
            & \gd_{\RS{r}{G}{H}} (\Nzerr{r}{G}{H})\, , \\
            & \gd_{\RS{r-1}{G}{H}} (\Nzerr{r}{G}{H}) + 1  \, , \\
            & \gd_{\FF_{r-1}} G  \}.
        \end{aligned}
    \end{equation*}
\end{Proposition}

\begin{proof}
    It is a consequence of Proposition~\ref{modelforFr} and Corollary~\ref{gdFrH}.
\end{proof}

\begin{Corollary}\label{finite_gd_Fr}
    Let $n>0$. If there are finite-dimensional models for
    $\EFG{\RS{r}{G}{H}}{\Nzerr{r}{G}{H}}$ for every $0 < r \leq n$ and
    $H\in\mathcal{H}_r$ and there is a finite-dimensional model
    for $\EFG{\FF_0}{G}$, then there is a finite-dimensional model
    for $\EFG{\FF_r}{G}$ for all $r \in\N$ with $0 \leq r \leq n$.
\end{Corollary}
\begin{proof}
    It is the result of Proposition~\ref{gdFr} and recursion over $r$.
\end{proof}

\subsection{Bredon cohomological dimensions}

Analogous results to those in the previous section can be proven for the Bredon
cohomological dimensions. We also present Mayer-Vietoris long exact sequences that
can help improve the upper bounds or provide lower bounds for the Bredon dimensions
when applied to particular groups and chains of families of subgroups.

\begin{Proposition}\label{cd_mv_rsHK}
    Let $K\in\mathcal{K}_i$ and let $M\in\RMod{\RS{i}{G}{H}[K]}{\Nzerr{i}{\Nzerr{r}{G}{H}}{K}}$.
    Let $F_1 : \RS{i}{\Nzerr{r}{G}{H}}{K} \to \RS{i}{G}{H}[K]$,
    $F_2 : \RS{i-1}{\Nzerr{i}{G}{K}}{H} \to \RS{i}{G}{H}[K]$ and
    $F_\cap : \RS{i-1}{\Nzerr{r}{G}{H}}{K} \to \RS{i}{G}{H}[K]$ be the
    inclusion functors.
    Then, the following sequence in Bredon cohomology is exact
    \begin{multline*}
        \cdots \rightarrow \Cohom{n}{\RS{i}{G}{H}[K]}{\Nzerr{i}{\Nzerr{r}{G}{H}}{K}}{M}
        \rightarrow \\
        \Cohom{n}{\RS{i}{\Nzerr{r}{G}{H}}{K}}{\Nzerr{i}{\Nzerr{r}{G}{H}}{K}}{\modres{F_1} M} \oplus
        \Cohom{n}{\RS{i-1}{\Nzerr{i}{G}{K}}{H}}{\Nzerr{i}{\Nzerr{r}{G}{H}}{K}}{\modres{F_2}M}
        \rightarrow \\
        \Cohom{n}{\RS{i-1}{\Nzerr{r}{G}{H}}{K}}{\Nzerr{i}{\Nzerr{r}{G}{H}}{K}}{\modres{F_\cap} M}
        \rightarrow \Cohom{n+1}{\RS{i}{G}{H}[K]}{\Nzerr{i}{\Nzerr{r}{G}{H}}{K}}{M} \to \cdots
    \end{multline*}

    and hence
    \begin{equation*}
        \begin{aligned}
            \cd_{\RS{i}{G}{H}[K]} (\Nzerr{i}{\Nzerr{r}{G}{H}}{K}) \leq \operatorname{max}\{
            & \cd_{\RS{i}{\Nzerr{r}{G}{H}}{K}} (\Nzerr{i}{\Nzerr{r}{G}{H}}{K})\, , \\
            & \cd_{\RS{i-1}{\Nzerr{i}{G}{K}}{H}} (\Nzerr{i}{\Nzerr{r}{G}{H}}{K})\, , \\
            & \cd_{\RS{i-1}{\Nzerr{r}{G}{H}}{K}} (\Nzerr{i}{\Nzerr{r}{G}{H}}{K}) + 1 \}.
        \end{aligned}
    \end{equation*}
\end{Proposition}
\begin{proof}
    Consequence of Lemmas \ref{cdE_FUG}, \ref{RSKiUnion} and \ref{RSKiInter}.
\end{proof}

\begin{Theorem}\label{cd_mv_rs_i_H}
    Let $H\in\mathcal{H}_r$ and let $i\in\N$ with $0 < i < r$. Let
    $M$ be a right $\OFG{\RS{i}{G}{H}}{\Nzerr{r}{G}{H}}$-module. Let
    $F_1 : \RS{i-1}{G}{H} \to \RS{i}{G}{H}$ and $F_{[K]} : \RS{i}{G}{H}[K] \to \RS{i}{G}{H}$
    and $I_{[K]} : \RS{i-1}{\Nzerr{i}{G}{K}}{H} \to \RS{i}{G}{H}$ the inclusion functors for each $K\in\mathcal{K}_i$.
    Then, the following sequence in Bredon Cohomology is exact
    \begin{multline*}
        \cdots \rightarrow \Cohom{n-1}{\RS{i}{G}{H}}{\Nzerr{r}{G}{H}}{M} \rightarrow \\
        \left( \underset{[K]\in \mathcal{K}_i}{\prod}
        \Cohom{n-1}{\RS{i}{G}{H}[K]}{\Nzerr{i}{\Nzerr{r}{G}{H}}{K}}{\modres{F_{[K]}} M}\right)
        \oplus  \Cohom{n-1}{\RS{i-1}{G}{H}}{G}{\modres{F_1} M}
        \rightarrow \\
        \underset{[K]\in\mathcal{K}_i}{\prod}
        \Cohom{n-1}{\RS{i-1}{\Nzerr{i}{G}{K}}{H}}{\Nzerr{i}{\Nzerr{r}{G}{H}}{K}}{\modres{I_{[K]}} M}
        \rightarrow \Cohom{n}{\RS{i}{G}{H}}{\Nzerr{r}{G}{H}}{M}\rightarrow \cdots
    \end{multline*}

    and hence
    \begin{equation*}
        \begin{aligned}
            \cd_{\RS{i}{G}{H}} (\Nzerr{r}{G}{H}) \leq
            \underset{[K]_i \in \mathcal{K}}{\operatorname{max}} \{
            & \cd_{\RS{i-1}{\Nzerr{i}{G}{K}}{H}} (\Nzerr{i}{\Nzerr{r}{G}{H}}{K} ) + 1 \, , \\
            & \cd_{\RS{i-1}{\Nzerr{r}{G}{H}}{K}} (\Nzerr{i}{\Nzerr{r}{G}{H}}{K} ) + 1 \, , \\
            & \cd_{\RS{i}{\Nzerr{r}{G}{H}}{K}} (\Nzerr{i}{\Nzerr{r}{G}{H}}{K} )  \, , \\
            & \cd_{\RS{i-1}{G}{H}} (\Nzerr{r}{G}{H} ) \}.
        \end{aligned}
    \end{equation*}
\end{Theorem}
\begin{proof}
    The Mayer-Vietoris sequence exists and is exact due to Corollary~\ref{modelforRSi} and
    Proposition~\ref{MV_LW}. The upper bound for Bredon cohomological dimension is
    consequence of Corollary~\ref{cdRSi_first} and Proposition~\ref{cd_mv_rsHK}.
\end{proof}

\begin{Proposition}\label{cd_mv_FrH}
    Let $H\in\mathcal{H}_r$ and let $M\in\RMod{\FF_r[H]}{\Nzerr{r}{G}{H}}$.
    Let $F_1 : \RS{r}{G}{H} \to \FF_r[H]$,
    $F_2 : \FF_{r-1}\cap\Nzerr{r}{G}{H} \to \FF_r[H]$ and
    $F_\cap : \RS{r-1}{G}{H} \to \FF_r[H]$ be the
    inclusion functors.
    Then, the following sequence in Bredon cohomology is exact
    \begin{multline*}
        \cdots \longrightarrow \Cohom{n}{\FF_r[H]}{\Nzerr{r}{G}{H}}{M}
        \longrightarrow \\
        \Cohom{n}{\RS{r}{G}{H}}{\Nzerr{r}{G}{H}}{\modres{F_1} M} \oplus
        \Cohom{n}{\FF_{r-1}\cap\Nzerr{r}{G}{H}}{\Nzerr{r}{G}{H}}{\modres{F_2} M}
        \longrightarrow \\
        \Cohom{n}{\RS{r-1}{G}{H}}{\Nzerr{r}{G}{H}}{\modres{F_\cap} M}
        \longrightarrow \Cohom{n+1}{\FF_r[H]}{\Nzerr{r}{G}{H}}{M} \longrightarrow \cdots
    \end{multline*}

    and hence
    \begin{equation*}
        \begin{aligned}
            \cd_{\FF_r[H]} (\Nzerr{r}{G}{H}) \leq \operatorname{max}\{
            & \cd_{\RS{r}{G}{H}} (\Nzerr{r}{G}{H})\, , \\
            & \cd_{\FF_{r-1}\cap\Nzerr{r}{G}{H}} (\Nzerr{r}{G}{H})\, , \\
            & \cd_{\RS{r-1}{G}{H}} (\Nzerr{r}{G}{H}) + 1 \}.
        \end{aligned}
    \end{equation*}
\end{Proposition}
\begin{proof}
    Consequence of Lemmas \ref{cdE_FUG} and \ref{ingrforE_FUG}.
\end{proof}

\begin{Theorem}\label{cd_mv_F_r}
    Let $G$ be a group and $\left(\FF_r\right)_{r\in\N}$ a strongly structured
    ascending chain of families of subgroups of $G$. Let $r > 0$ and
    $M\in\RMod{\FF_r}{G}$. Let $\mathcal{H}_r$ be a set of representatives of the
    equivalence classes in $\dF{F}{r}$ with respect to $\sim_r$ (as defined in
    Definition~\ref{def_eqrel_r}).
    Let $F_1 : \FF_{r-1} \to \FF_r$ and $F_{[H]} : \FF_r[H] \to \FF_r$
    and $I_{[H]} : \FF_{r-1}\cap\Nzerr{r}{G}{H} \to \FF_r$ the inclusion functors for
    each $H\in\mathcal{H}_r$.
    Then, the following sequence in Bredon Cohomology is exact
    \begin{multline*}
        \cdots \longrightarrow \Cohom{n-1}{\FF_r}{G}{M} \longrightarrow \\
        \left( \underset{[H]\in \mathcal{H}_r}{\prod}
        \Cohom{n-1}{\FF_r[H]}{\Nzerr{r}{G}{H}}{\modres{F_{[H]}} M}\right)
        \oplus  \Cohom{n-1}{\FF_{r-1}}{G}{\modres{F_1} M}
        \longrightarrow \\
        \underset{[H]\in \mathcal{H}_r}{\prod}
        \Cohom{n-1}{\FF_{r-1}\cap \Nzerr{r}{G}{H}}{\Nzerr{r}{G}{H}}{\modres{I_{[H]}} M}
        \longrightarrow \Cohom{n}{\FF_r}{G}{M}\longrightarrow \cdots
    \end{multline*}

    and hence
    \begin{equation*}
        \begin{aligned}
            \cd_{\FF_r} G \leq \underset{[H]_r \in \mathcal{H}_r}{\operatorname{max}} \{
            & \cd_{\FF_{r-1} \cap \Nzerr{r}{G}{H}} (\Nzerr{r}{G}{H}) + 1 \, , \\
            & \cd_{\RS{r}{G}{H}} (\Nzerr{r}{G}{H})\, , \\
            & \cd_{\RS{r-1}{G}{H}} (\Nzerr{r}{G}{H}) + 1  \, , \\
            & \cd_{\FF_{r-1}} G  \}.
        \end{aligned}
    \end{equation*}
\end{Theorem}
\begin{proof}
    Propositions \ref{modelforFr} and \ref{MV_LW} gives us the Mayer-Vietoris long exact
    sequence and Corollary~\ref{lw-dimensions} and Proposition~\ref{cd_mv_FrH}, the upper
    bound on Bredon cohomological dimension.
\end{proof}

\section{Classifying spaces for the families $\RS{r}{G}{H}$}

In this section, we give some additional conditions to the initial set-up that will
guarantee a finite-dimensional model for $\EFG{\RS{r}{G}{H}}{\Nzerr{r}{G}{H}}$. From this point
on, we adopt the commonly used notation $\ucd G, \ugd G, \UEG$ to represent $\cd_{\Fin (G)} G$,
$\gd_{\Fin (G)} G$ and $\EFG{\Fin (G)}{G}$, respectively.

\begin{Theorem}\label{cd_nzerr_nzr_general_case}
    Let $H\in\dF{F}{r}$ be such that $\Nzr{G}{H} = \Nzerr{r}{G}{H}$. Then,
    $\HHb = \{ LH/H \st L\in\RS{r}{G}{H}\}$ is a full family of subgroups of
    $\Nzr{G}{H}/H$. Moreover, if there is $n\in\N$ such that $\cd_{\RS{r}{G}{H}\cap LH} LH \leq n$
    for all $L\in\RS{r}{G}{H}$, then
    $$\cd_{\RS{r}{G}{H}} \Nzr{G}{H} \leq \cd_{\HHb} \left(\Nzr{G}{H}/H\right) + n.$$
\end{Theorem}
\begin{proof}
    First of all, if $L\in\RS{r}{G}{H}$, since $\Nzerr{r}{G}{H} = \Nzr{G}{H}$, $L\leq \Nzr{G}{H}$, so
    $H \lhd LH$ and $LH/H$ is a subgroup of $\Nzr{G}{H}/H$.

    To see that $\HHb$ is closed under
    conjugation, we need to prove that $\left(LH/H\right)^{gH}$ belongs to $\HHb$ for $gH\in\Nzr{G}{H}/H$
    and $L\in\RS{r}{G}{H}$. Take $kH \in LH/H$. Then $(gH)(kH)(gH)^{-1} = gkg^{-1}H \in L^gH/H$,
    given the fact that $g\Nzr{G}{H}$ and hence $(LH)^g = L^gH$. And since $L\in\RS{r}{G}{H}$,
    $g\in\Nzr{G}{H}$ and $\RS{r}{G}{H}$ is closed under conjugation, $L^gH/H \in \HHb$.

    Let $L\in\RS{r}{G}{H}$ and $S/H\leq LH/H$. We need to find $S^\prime\in\RS{r}{G}{H}$ such that
    $S^\prime H/H = S/H$. Take $S^\prime = S\cap L$. Since $S\leq LH$, then $S\cap L \leq L$, so
    $S^\prime \in \FF_r$. We only need to see that $S^\prime \cap H \sim_j S^\prime$ for some $j\leq r$.
    But $S^\prime \cap H = S \cap L \cap H = S \cap H \cap L = S \cap L = S^\prime$, since $H\leq S$.
    In particular, $S^\prime \cap H \sim_j S$ for some $j\leq r$, so $\HHb$ is closed under taking subgroups.

    Now that we proved that $\HHb$ is a full family of subgroups of $\Nzr{G}{H}/H$, as
    direct consequence of Corollary~\ref{cd_for_quotients}, we obtain the rest
    of the theorem.
\end{proof}

\begin{Theorem}\label{gd_nzerr_nzr_general_case}
    Let $H\in\dF{F}{r}$ be such that $\Nzr{G}{H} = \Nzerr{r}{G}{H}$. Then, if there is $n\in\N$
    such that $\gd_{\RS{r}{G}{H}\cap LH} LH \leq n$
    for all $L\in\RS{r}{G}{H}$, then
    $$\gd_{\RS{r}{G}{H}} \Nzr{G}{H} \leq \gd_{\HHb} \left(\Nzr{G}{H}/H\right) + n.$$
\end{Theorem}
\begin{proof}
    As we saw in Theorem~\ref{cd_nzerr_nzr_general_case}, $\HHb$ is a full family of subgroups of $\Nzr{G}{H}/H$.
    Since $\RS{r}{G}{H}$ is a full family of subgroups of $\Nzr{G}{H}$ and $LH/H \in\HHb$ for every $L\in\RS{r}{G}{H}$,
    Theorem~\ref{gd_for_quotients} yields the inequality we wanted to show.
\end{proof}

\begin{Corollary}\label{cd_nzerr_nzr_LH_in_same_family}
    Let $H\in\dF{F}{r}$ be such that $\Nzr{G}{H} = \Nzerr{r}{G}{H}$ and $LH\in\RS{r}{G}{H}$
    for all $L\in\RS{r}{G}{H}$. Then,
    $$\cd_{\RS{r}{G}{H}} \Nzr{G}{H} \leq \cd_{\HHb} \left(\Nzr{G}{H}/H\right).$$
\end{Corollary}
\begin{proof}
    By Theorem~\ref{cd_nzerr_nzr_general_case}, we know $\HHb$ is a full family
    of subgroups of $\Nzr{G}{H}$ and hence we can apply Corollary~\ref{cd_for_quotients}.
    Let $\HH = \{ S\leq G \st N \leq S \text{ and } S/N \in \HHb\}$ and let $S\in\HH$.
    By definition of $\HHb$ and $\HH$, we know that $S = LH$ for some $L\in\RS{r}{G}{H}$.
    Then, $S\in\RS{r}{G}{H}$, by hypothesis, so $\cd_{\RS{r}{G}{H}\cap S} S = 0$, as we needed
    to see.
\end{proof}

\begin{Corollary}\label{gd_nzerr_nzr_LH_in_same_family}
    Let $H\in\dF{F}{r}$ be such that $\Nzr{G}{H} = \Nzerr{r}{G}{H}$ and $LH\in\RS{r}{G}{H}$
    for all $L\in\RS{r}{G}{H}$. Then,
    $$\gd_{\RS{r}{G}{H}} \Nzr{G}{H} \leq \gd_{\HHb} \left(\Nzr{G}{H}/H\right).$$
\end{Corollary}
\begin{proof}
    Let $\HH = \{ S\leq G \st N \leq S \text{ and } S/N \in \HHb\}$.
    By Theorem~\ref{gd_for_quotients}, we only need to see that $\gd_{\RS{r}{G}{H}\cap S}S = 0$
    for all $S\in\HH$. And that is true given that $S\in\RS{r}{G}{H}$ for all $S\in\HH$, as we
    saw in the proof of Corollary~\ref{cd_nzerr_nzr_LH_in_same_family}.
\end{proof}

\begin{Corollary}\label{cd_top_rank_space_for_rs_comm}
    Let $\left( \FF_n \right)_{n\in\N}$ be a strongly structured ascending chain of families
    of subgroups of $G$ such that the equivalence relation $\sim_r$ in $\dF{F}{r}$ is
    commensurability, i.e., if $H, K \in \dF{F}{r}$ then $H\cap K \in \dF{F}{r}$
    if and only if $| H : H\cap K| < \infty$ and $| K : H\cap K | < \infty$.
    Then, if $[H]_r \in [\dF{F}{r}]_r$ is such that $\Nzerr{r}{G}{H} = \Nzr{G}{H}$, we have
    $$\cd_{\RS{r}{G}{H}} \Nzerr{r}{G}{H} \leq \ucd \left(\Nzr{G}{H}/H \right).$$
\end{Corollary}
\begin{proof}
    First, we need to see that the family $\HHb$ defined in Theorem~\ref{cd_nzerr_nzr_general_case}
    is the family of finite subgroups of $\Nzr{G}{H}/H$.

    For a subgroup of $L\leq \Nzr{G}{H}$, $LH/H\in\Fin(\Nzr{G}{H}/H)$ if and only if $|LH : H| < \infty$.
    Also, as $\sim_r$ is commensurability, $L\leq \Nzr{G}{H}$ belongs to $\RS{r}{G}{H}$ if and only if $|H : H\cap L| < \infty$.
    By the Second Isomorphism Theorem, we know $|LH : H| = |H : H\cap L|$, so $LH/H \in \Fin(\Nzr{G}{H}/H)$ if and only if
    $LH/H\in \HHb$, as we wanted to see.

    By Corollary~\ref{cd_nzerr_nzr_LH_in_same_family}, it only remains to prove that given
    $H\in\dF{F}{r}$ such that $\Nzerr{r}{G}{H} = \Nzr{G}{H}$ and $L\in\FF_r \cap \Nzr{G}{H}$ such that
    $|L : L\cap H| < \infty$, then $|LH : H| < \infty$. But that is consequence of $L\leq\Nzr{G}{H}$ and
    the Second Isomorphism Theorem.
\end{proof}

\begin{Corollary}\label{gd_top_rank_space_for_rs_comm}
    Let $\left( \FF_n \right)_{n\in\N}$ be a strongly structured ascending chain of families
    of subgroups of $G$ such that the equivalence relation $\sim_r$ in $\dF{F}{r}$ is
    commensurability. Assume that $[H]_r \in [\dF{F}{r}]_r$ is such that $\Nzerr{r}{G}{H} = \Nzr{G}{H}$
    and let $X$ be a model for $\UEGG{\Nzr{G}{H}/H}$. Then, $X$ is a model for $\EFG{\RS{r}{G}{H}}{\Nzerr{r}{G}{H}}$
    and
    $$\gd_{\RS{r}{G}{H}} \Nzerr{r}{G}{H} \leq \ugd \left(\Nzr{G}{H}/H \right).$$
\end{Corollary}
\begin{proof}
    Let $\HH = \{ L \leq G \st LH/H \in \Fin (\Nzr{G}{H}/H) \}$. If we show that $\HH = \RS{r}{G}{H}$, by
    Lemma~\ref{class_space_for_quotients}, we will reach the desired conclusions.

    But by the same argumentation we used to prove $\HHb = \Fin(\Nzr{G}{H}/H)$ in Corollary~\ref{cd_top_rank_space_for_rs_comm},
    $L\in\HH$ if and only if $L\in\RS{r}{G}{H}$.

    The inequality of Bredon geometric dimensions can be also proven as in the proof of Corollary~\ref{cd_top_rank_space_for_rs_comm},
    using Corollary~\ref{gd_nzerr_nzr_LH_in_same_family}.
\end{proof}

We would like to relax the condition of $\Nzr{G}{H} = \Nzerr{r}{G}{H}$ in order to provide
upper bounds for the Bredon dimension with respect to the families $\RS{r}{G}{H}$ of subgroups
of $\Nzerr{r}{G}{H}$ in a more general set-up. For that, the following
equivalence will be helpful:

\begin{Lemma}\label{cond_are_equivalent_RSrGH}
    Let $H\in\dF{F}{r}$. Then, the following conditions are equivalent:
    \begin{enumerate}[label = {\emph{(\roman*)}}, noitemsep]
        \item there is $H^\prime \in [H]_r$ such that $\Nzr{G}{H^\prime} = \Nzerr{r}{G}{H}$;
        \item for all $K\leq \Nzerr{r}{G}{H}$ there is $H_K \in [H]_r$ such that
        $K \leq \Nzr{G}{H_K}.$
    \end{enumerate}
\end{Lemma}
\begin{proof}
    For the first implication ((i) $\Rightarrow$ (ii)), let $K\leq\Nzerr{r}{G}{H}$. If we
    take $H_K = H^\prime$, we are done, since $\Nzr{G}{H^\prime} = \Nzerr{r}{G}{H}$.

    For the other implication, take $K = \Nzerr{r}{G}{H}$. Then, by hypothesis,
    there is $H_{\Nzerr{r}{G}{H}}$ such that
    $\Nzerr{r}{G}{H} \leq \Nzr{G}{H_{\Nzerr{r}{G}{H}}}$. But since $H_{\Nzerr{r}{G}{H}} \in [H]_r$,
    we also have $\Nzr{G}{H_{\Nzerr{r}{G}{H}}} \leq \Nzerr{r}{G}{H_{\Nzerr{r}{G}{H}}} = \Nzerr{r}{G}{H}$. Hence, taking
    $H^\prime = H_{\Nzerr{r}{G}{H}}$ completes the proof.
\end{proof}

Condition $(ii)$ in Lemma~\ref{cond_are_equivalent_RSrGH} (and hence condition $(i)$) can be relaxed by restricting
$K$ to belong to some set of subgroups of $\Nzerr{r}{G}{H}$, instead of it being any subgroup.
The following results apply this idea to a decomposition of $\Nzerr{r}{G}{H}$
as a direct union of subgroups.

We can see a similar methodology applied to the particular case of CAT(0) groups and the chains
of families $\Fin \subseteq \Vcy$ in \cite{degrijsepetrosyan_cat0} and
$\left(\AA_r\right)_{r\in\N}$ in \cite{prytula_cat0}, where $\AA_r$ is the family of virtually
abelian subgroups of torsion-free rank less than or equal to $r$. In the following results we relax the hypothesis
used in the aforementioned publications, while also making the proofs independent of the particular class of groups
and families of subgroups considered.

\begin{Theorem}\label{NGHr_direct_union_dim_RrGH}
    Let $G$ be a group and $\left(\FF_r\right)_{r\in\N}$ be a SSACFS of $G$. Let $H\in\dF{F}{r}$ such that
    $\Nzerr{r}{G}{H}$ is the direct union of $\{N_\lambda \st \lambda\in\Lambda\}$ with:
    \begin{enumerate}[label= {\emph{(\roman*)}}, noitemsep]
        \item $\Lambda$ is countable;
        \item for all $K\in\RS{r}{G}{H}$ there is $\lambda\in\Lambda$ such that $K\leq N_\lambda$; and
        \item for all $\lambda\in\Lambda$ there is $H_\lambda\in\dF{F}{r}$ such that $H\sim_r H_\lambda$
        and $N_\lambda \leq \Nzr{G}{H_\lambda}.$
    \end{enumerate}
    Then, if $s = \underset{\lambda\in\Lambda}{\operatorname{sup}}\{
    \cd_{\RS{r}{G}{H}\cap\Nzr{G}{H_\lambda}}\Nzr{G}{H_\lambda}\}$,
    $$ s \leq \cd_{\RS{r}{G}{H}} \Nzerr{r}{G}{H} \leq s + 1$$
    and
    $$ s \leq \gd_{\RS{r}{G}{H}} \Nzerr{r}{G}{H} \leq s + 1.$$
\end{Theorem}
\begin{proof}
    By Proposition~\ref{every_sg_in_one_summand_implies_compatible}, by assumption $(ii)$ and
    since $\RS{r}{G}{H}$ is a full family, we know that
    $\RS{r}{G}{H} \cap N_\lambda$ for $\lambda\in\Lambda$ and $\RS{r}{G}{H}$ are compatible with the
    direct union. Then, by Theorems \ref{cd_direct_unions} and \ref{gd_direct_unions},
    $$\cd_{\RS{r}{G}{H}} \Nzerr{r}{G}{H} \leq
    \underset{\lambda\in\Lambda}{\operatorname{sup}}\{
    \cd_{\RS{r}{G}{H}\cap N_\lambda} N_\lambda\} + 1$$
    and
    $$\gd_{\RS{r}{G}{H}} \Nzerr{r}{G}{H} \leq
    \underset{\lambda\in\Lambda}{\operatorname{sup}}\{
    \gd_{\RS{r}{G}{H}\cap N_\lambda} N_\lambda\} + 1,$$
    respectively.

    Consider now $H_\lambda$ as in assumption $(iii)$. $N_\lambda \leq \Nzr{G}{H_\lambda}$ and
    $\RS{r}{G}{H} \cap \Nzr{G}{H_\lambda}$ is a full family of subgroups, so by
    Theorem~\ref{Bredon:cd_subgroup}, $\cd_{\RS{r}{G}{H}\cap N_\lambda} N_\lambda \leq
    \cd_{\RS{r}{G}{H} \cap \Nzr{G}{H_\lambda}} \Nzr{G}{H_\lambda}$. That proves $\cd_{\RS{r}{G}{H}} \Nzerr{r}{G}{H} \leq s + 1$.
    Using Theorem~\ref{Bredon:gd_subgroup} instead, we obtain $\gd_{\RS{r}{G}{H}} \Nzerr{r}{G}{H} \leq s + 1$.

    Since $H_\lambda \sim_r H$, we know $\Nzr{G}{H_\lambda} \leq \Nzerr{r}{G}{H}$. For that reason,
    by Theorems \ref{Bredon:cd_subgroup} and \ref{Bredon:gd_subgroup},
    we have for all $\lambda\in\Lambda$
    $$\cd_{\RS{r}{G}{H} \cap \Nzr{G}{H_\lambda}} \Nzr{G}{H_\lambda} \leq
    \cd_{\RS{r}{G}{H}} \Nzerr{r}{G}{H}$$
    and
    $$\gd_{\RS{r}{G}{H} \cap \Nzr{G}{H_\lambda}} \Nzr{G}{H_\lambda} \leq
    \gd_{\RS{r}{G}{H}} \Nzerr{r}{G}{H}.$$
    And since the supremum of a set is the smallest of its upper bounds, that finishes the proof.
\end{proof}

We can now deduce similar results to \ref{cd_nzerr_nzr_general_case}, \ref{gd_nzerr_nzr_general_case},
\ref{cd_nzerr_nzr_LH_in_same_family}, \ref{gd_nzerr_nzr_LH_in_same_family}, \ref{cd_top_rank_space_for_rs_comm}
and \ref{gd_top_rank_space_for_rs_comm} substituting the condition that $\Nzr{G}{H} = \Nzerr{r}{G}{H}$
by the hypotheses in Theorem~\ref{NGHr_direct_union_dim_RrGH}.

\begin{Theorem}\label{cd_NGHr_direct_union_general_case}
    Let $H\in\dF{F}{r}$ such that $\Nzerr{r}{G}{H}$ is the direct union of
    $\{N_\lambda \st \lambda\in\Lambda\}$ and conditions $(i)-(iii)$ in
    Theorem~\ref{NGHr_direct_union_dim_RrGH} hold. Then, for every $\lambda\in\Lambda$,
    $\HHb_\lambda = \{ LH_\lambda/H_\lambda \st L\in\RS{r}{\Nzr{G}{H_\lambda}}{H_\lambda}\}$
    is a full family of subgroups of $\Nzr{G}{H_\lambda}/H_\lambda$.
    Moreover, if there is $n\in\N$ such that
    $\cd_{\RS{r}{\Nzr{G}{H_\lambda}}{H_\lambda}\cap LH_\lambda} LH_\lambda \leq n$
    for all $L\in\RS{r}{\Nzr{G}{H_\lambda}}{H_\lambda}$ and for all $\lambda\in\Lambda$, then
    $$\cd_{\RS{r}{G}{H} \Nzerr{r}{G}{H}} \Nzerr{r}{G}{H} \leq
    \underset{\lambda\in\Lambda}{\operatorname{sup}}\{\cd_{\HHb_\lambda} \left(\Nzr{G}{H_\lambda}/H_\lambda\right)\} + n + 1.$$
\end{Theorem}
\begin{proof}
    First, note that $\RS{r}{\Nzr{G}{H_\lambda}}{H_\lambda} = \RS{r}{G}{H} \cap \Nzr{G}{H_\lambda},$
    since $H_\lambda \sim_r H$ for all $\lambda\in\Lambda$. In particular, for all $L\in\RS{r}{\Nzr{G}{H_\lambda}}{H_\lambda}$
    we have $H_\lambda \lhd LH$, so $\HHb_\lambda$ is a well defined and, following the reasoning
    in Theorem~\ref{cd_nzerr_nzr_general_case}, full family of subgroups of $\Nzr{G}{H_\lambda}$.

    By Theorem~\ref{NGHr_direct_union_dim_RrGH}, $$\cd_{\RS{r}{G}{H}} \Nzerr{r}{G}{H} \leq \underset{\lambda\in\Lambda}{\operatorname{sup}}\{
    \cd_{\RS{r}{G}{H}\cap\Nzr{G}{H_\lambda}}\Nzr{G}{H_\lambda}\} + 1.$$

    That means that we only need to prove that $\cd_{\RS{r}{G}{H}\cap\Nzr{G}{H_\lambda}}\Nzr{G}{H_\lambda}
    \leq \cd_{\HHb_\lambda} \left(\Nzr{G}{H_\lambda}/H_\lambda\right) + n$ for every $\lambda\in\Lambda$.
    And that is consequence of Theorem~\ref{cd_nzerr_nzr_general_case}.
\end{proof}

\begin{Theorem}\label{gd_NGHr_direct_union_general_case}
    Let $H\in\dF{F}{r}$ such that $\Nzerr{r}{G}{H}$ is the direct union of
    $\{N_\lambda \st \lambda\in\Lambda\}$ and conditions $(i)-(iii)$ in
    Theorem~\ref{NGHr_direct_union_dim_RrGH} hold. Then, if there is $n\in\N$ such that
    $\gd_{\RS{r}{\Nzr{G}{H_\lambda}}{H_\lambda}\cap LH_\lambda} LH_\lambda \leq n$
    for all $L\in\RS{r}{\Nzr{G}{H_\lambda}}{H_\lambda}$ and for all $\lambda\in\Lambda$,
    $$\gd_{\RS{r}{G}{H} \Nzerr{r}{G}{H}} \Nzerr{r}{G}{H} \leq
    \underset{\lambda\in\Lambda}{\operatorname{sup}}\{\gd_{\HHb_\lambda} \left(\Nzr{G}{H_\lambda}/H_\lambda\right)\} + n + 1.$$
\end{Theorem}
\begin{proof}
    As in proof of Theorem~\ref{cd_NGHr_direct_union_general_case} but using Theorem~\ref{gd_nzerr_nzr_general_case} instead of Theorem~\ref{cd_nzerr_nzr_general_case}.
\end{proof}

\begin{Corollary}\label{cd_NGHr_direct_union_LH_in_same_family}
    Let $H\in\dF{F}{r}$ be such that $\Nzerr{r}{G}{H}$ is the direct union of
    $\{N_\lambda \st \lambda\in\Lambda\}$ and conditions $(i)-(iii)$ in
    Theorem~\ref{NGHr_direct_union_dim_RrGH} hold. Assume $LH_\lambda\in\RS{r}{\Nzr{G}{H_\lambda}}{H_\lambda}$
    for all $L\in\RS{r}{\Nzr{G}{H_\lambda}}{H_\lambda}$ and for all $\lambda\in\Lambda$.
    Then,
    $$\cd_{\RS{r}{G}{H}} \Nzerr{r}{G}{H} \leq
    \underset{\lambda\in\Lambda}{\operatorname{sup}}\{\cd_{\HHb_\lambda} \left(\Nzr{G}{H_\lambda}/H_\lambda\right)\} + 1.$$
\end{Corollary}
\begin{proof}
    As in proof of Corollary~\ref{cd_nzerr_nzr_LH_in_same_family}, but using Theorem~\ref{cd_NGHr_direct_union_general_case}
    instead of Theorem~\ref{cd_nzerr_nzr_general_case}.
\end{proof}

\begin{Corollary}\label{gd_NGHr_direct_union_LH_in_same_family}
    Let $H$ as in Corollary~\ref{cd_NGHr_direct_union_LH_in_same_family}. Then,
    $$\gd_{\RS{r}{G}{H}} \Nzerr{r}{G}{H} \leq
    \underset{\lambda\in\Lambda}{\operatorname{sup}}\{\gd_{\HHb_\lambda} \left(\Nzr{G}{H_\lambda}/H_\lambda\right)\} + 1.$$
\end{Corollary}
\begin{proof}
    As in proof of Corollary~\ref{gd_nzerr_nzr_LH_in_same_family}, but using Theorem~\ref{gd_NGHr_direct_union_general_case}
    instead of Theorem~\ref{gd_nzerr_nzr_general_case}.
\end{proof}

\begin{Corollary}\label{cd_direct_union_top_rank_space_for_rs_comm}
    Let $\left( \FF_n \right)_{n\in\N}$ be a strongly structured ascending chain of families
    of subgroups of $G$ such that the equivalence relation $\sim_r$ in $\dF{F}{r}$ is
    commensurability. Let $[H]_r \in [\dF{F}{r}]_r$ be such that $\Nzerr{r}{G}{H}$ is the direct union of
    $\{N_\lambda \st \lambda\in\Lambda\}$ and conditions $(i)-(iii)$ in
    Theorem~\ref{NGHr_direct_union_dim_RrGH} hold. Then,
    $$\cd_{\RS{r}{G}{H}} \Nzerr{r}{G}{H} \leq
    \underset{\lambda\in\Lambda}{\operatorname{sup}}\{
        \ucd \left( \Nzr{G}{H_\lambda}/H_\lambda \right)\} + 1.$$
\end{Corollary}
\begin{proof}
    Proceeding as we did in proof of Corollary~\ref{cd_top_rank_space_for_rs_comm}, but using Corollary~\ref{cd_NGHr_direct_union_LH_in_same_family}
    instead of Corollary~\ref{cd_nzerr_nzr_LH_in_same_family}.
\end{proof}

\begin{Corollary}\label{gd_direct_union_top_rank_space_for_rs_comm}
    Let $\left( \FF_n \right)_{n\in\N}$ and $[H]_r \in [\dF{F}{r}]_r$ as in
    Corollary~\ref{cd_direct_union_top_rank_space_for_rs_comm}. Then,
    $$\gd_{\RS{r}{G}{H}} \Nzerr{r}{G}{H} \leq
    \underset{\lambda\in\Lambda}{\operatorname{sup}}\{
        \ugd \left(\Nzr{G}{H_\lambda} \right)\} + 1.$$
\end{Corollary}
\begin{proof}
    As in proof of Corollary~\ref{gd_top_rank_space_for_rs_comm}, but using Corollary~\ref{gd_NGHr_direct_union_LH_in_same_family}
    instead of Corollary~\ref{gd_nzerr_nzr_LH_in_same_family}.
\end{proof}

%% file: ClassForPoly.tex
\chapter{Classifying spaces for families of virtually polycyclic subgroups}\label{ch:class_poly}
\minitoc

In this chapter we will use the constructions and results presented in Chapters \ref{ch:related-families} and \ref{ch:chains}
to study the Bredon dimensions of certain groups $G$ with respect to families of virtually
polycyclic subgroups. We will first focus on groups $G$ that are themselves virtually polycyclic to
then extend those results to groups $G$ belonging to a wider class of groups.

In \cite{scott} and \cite{segal}, many of the basic properties of polycyclic and virtually polycyclic groups can be found.
In the first source mentioned, virtually polycyclic groups are referred to as M-groups.

\begin{Defn}\label{def_poly-something}
    Given a property or class of groups $\mathfrak{X}$, we say that a group $G$ is
    \emph{poly-}~$\mathfrak{X}$ if it admits a subnormal series
    $$\{1\} = G_0 \lhd G_1 \lhd \ldots \lhd G_k = G$$
    where $G_{i+1}/G_i$ has the property $\mathfrak{X}$ for all
    $0 \leq i < k$.
\end{Defn}

Polycyclic groups, hence, are those that admit a finite subnormal series with cyclic
factors. Polycyclicity is preserved under taking subgroups and quotients.

Note that being virtually polycyclic, polycyclic-by-finite, poly-$\Z$-by-finite
or poly-($\Z$~or~finite) are equivalent group properties, and all of them are
preserved by taking subgroups, quotients and extensions.

We will use these facts throughout the current chapter
without mention.

As an invariant for polycyclic-by-finite groups to define an indexed ascending
chain of subgroups (as torsion-free rank does in \cite{nucinkisetal}), we take
the Hirsch length of the group:

\begin{Defn}
    Given a polycyclic-by-finite group $G$, its \emph{Hirsch length} $h(G)$
    is the number of infinite cyclic factors in any of its subnormal
    series with infinite cyclic or finite factors.
\end{Defn}

These and many other properties of this class of groups and the Hirsch length can be found
in \cite{segal}. The following one is crucial in many of the proofs:

\begin{Lemma}\label{hirschl} Let $G$ be virtually polycyclic, $H \leq G$ and $N \lhd G$, then
    \begin{enumerate}[label={\emph{(\roman*)}}]
        \item $h(H) \leq h(G)$
        \item $h(H) = h(G) \Leftrightarrow |H : G| < \infty$
        \item $h(G) = h(N) + h(G/N)$
    \end{enumerate}
\end{Lemma}

\begin{Defn}\label{def_Hr}
    Let $G$ be any group.
    We define $\HH_r$ as the family of virtually polycyclic subgroups of $G$ of Hirsch
    length less than or equal to $r$, where $r\in \N$.
\end{Defn}

\begin{Corollary}
    Let $G$ be any group, then $(\HH_r)_{r\in\N}$ is a strongly structured
    ascending chain of subgroups of $G$.
\end{Corollary}
\begin{proof}
    It is direct consequence of Lemma~\ref{bounded_rank_families_are_SSCFS} and
    Lemma~\ref{hirschl}.
\end{proof}

In this case, given $H\in\dF{\HH}{r}$ and $i \leq r$, we have
$$\RS{i}{G}{H} = \{ K \leq \Nzerr{r}{G}{H} \st h(K) = h(H\cap K) \leq i\}$$
and $\Nzerr{r}{G}{H} = \comm_G (H)$.

\section{Virtually polycyclic groups}

Let us focus our interest in the chain of families $(\HH_r)_{r\in\N}$ of subgroups of
a virtually polycyclic group $G$.

Note that in general the family $\HH_0$ is the family of finite subgroups of $G$ and the family
$\HH_1$ is the family of virtually cyclic subgroups of $G$. Models for $E_\Fin G$ and
$E_\Vcy G$ and the dimensions $\gd_\Fin G$, $\gd_\Vcy G$, $\cd_\Fin G$ and $\cd_\Vcy G$
for any virtually polycyclic group $G$ can be found in \cite{luecksurvey} and
\cite{lueckweiermann}.

For this reason, we dispose of the base case spaces in both recursions for the construction described in previous chapter, given the fact that $\RS{0}{G}{H}$ is the family of finite subgroups
of $\Nzerr{h(H)}{G}{H}$, which is also virtually polycyclic.

But as we saw in last chapter, we also require (finite-dimensional) models for
$\EFG{\RS{r}{G}{H}}{\Nzerr{r}{G}{H}}$ for $H\in\dF{\FF}{r}$ and for all
$r > 0$, where the families $\RS{r}{G}{H}$ are as in Definition~\ref{defRS}.
We will use Theorems \ref{cd_top_rank_space_for_rs_comm} and \ref{gd_top_rank_space_for_rs_comm}
to build them, and for that we need the result that follows:

\begin{Lemma}\label{normcomm}\cite[Corollary 10]{capracekropholler}
    The following assertions are equivalent for any finitely generated virtually
    soluble group $G$:
    \begin{enumerate}[label={(\roman*)}, noitemsep]
        \item $G$ is polycyclic-by-finite.
        \item Every $H\leq G$ contains a finite index subgroup $K$ such that $\Nzr{G}{K} = \comm_G (H).$
        \item For all subnormal subgroups $H$ of $G$ and all finitely generated $J \leq \comm_G (H)$,
        there exists a finite index subgroup $K\leq H$ which is normal in $\langle J \cup K \rangle.$
    \end{enumerate}
\end{Lemma}

We can take for each $r>0$ a set of representatives $\mathcal{H}_r$ of the classes
with respect to $\sim_r$ such that if $H \in \mathcal{H}_r$ then $\Nzr{G}{H} = \comm_G (H)$.

\begin{Corollary}\label{gd_top_rank_RS_poly}
    Let $G$ be a virtually polycyclic group and for every $r\in\N$ let $\HH_r$ be the family
    of subgroups of $G$ of Hirsch length smaller than or equal to $r$. Let $H\in\mathcal{H}_r$.
    If $X$ is a model for $\EFG{\Fin}{\Nzr{G}{H}/H}$, then $X$ is also a model for
    $\EFG{\RS{r}{G}{H}}{\Nzr{G}{H}}$. In particular, we have
    $$\gd_{\RS{r}{G}{H}}\left(\Nzr{G}{H}\right) \leq h(\Nzr{G}{H}) - h(H).$$
\end{Corollary}
\begin{proof}
    Consequence of Corollary~\ref{gd_top_rank_space_for_rs_comm}, since $\ugd (\Nzr{G}{H}/H) =
    h(\Nzr{G}{H}/H) = h(\Nzr{G}{H}) - h(H)$.
\end{proof}

\begin{Corollary}\label{cd_top_rank_RS_poly}
    Let $G$ be a virtually polycyclic group and for every $r\in\N$ let $\HH_r$ be the family
    of subgroups of $G$ of Hirsch length smaller than or equal to $r$. Let $H\in\mathcal{H}_r$.
    Then,
    $$\cd_{\RS{r}{G}{H}} \Nzr{G}{H} \leq h(\Nzr{G}{H}) - h(H).$$
\end{Corollary}
\begin{proof}
    Since $\ucd (\Nzr{G}{H}/H) = h(\Nzr{G}{H}/H) = h(\Nzr{G}{H}) - h(H)$, by Theorem~\ref{cd_top_rank_space_for_rs_comm}
    the inequality we had to prove holds.
\end{proof}

Now we have all necessary ingredients to apply results from Chapter~\ref*{ch:chains}. Let us first study
the chain of families $(\RS{i}{G}{H})_{i = 0}^{r}$ for $H\in\mathcal{H}_r$.

As a direct consequence of Corollary~\ref{gdRSKi} and Proposition~\ref{gdRSi}, we have:

\begin{Corollary}\label{gd_RS_as_max_poly}
    Let $G$ be a virtually polycyclic group and for every $r\in\N$ let $\HH_r$ be the family
    of subgroups of $G$ of Hirsch length smaller than or equal to $r$. Let $H\in\mathcal{H}_r$
    and let $\mathcal{K}_i = \mathcal{H}_i \cap \RS{i}{G}{H}$ for $i = 0, \ldots, r - 1$. Then,

    \begin{equation*}
        \begin{aligned}
            \gd_{\RS{i}{G}{H}[K]} (\Nzr{\Nzr{G}{H}}{K}) \leq
            \underset{[K]_i \in \mathcal{K}_i}{\operatorname{max}}\{ &
            h(\Nzr{\Nzr{G}{H}}{K}) - h(K)\, , \\ &
            \gd_{\RS{i-1}{\Nzr{G}{K}}{H}} (\Nzr{\Nzr{G}{H}}{K})\, , \\ &
            \gd_{\RS{i-1}{\Nzr{G}{H}}{K}} (\Nzr{\Nzr{G}{H}}{K}) + 1 \}
        \end{aligned}
    \end{equation*}

    and hence

    \begin{equation*}
        \begin{aligned}
            \gd_{\RS{i}{G}{H}} (\Nzr{G}{H} ) \leq
            \underset{[K]_i \in \mathcal{K}_i}{\operatorname{max}} \{ &
            \gd_{\RS{i-1}{\Nzr{G}{K}}{H}} (\Nzr{\Nzr{G}{H}}{K} ) + 1 \, , \\ &
            \gd_{\RS{i-1}{\Nzr{G}{H}}{K}} (\Nzr{\Nzr{G}{H}}{K} ) + 1 \, , \\ &
            h(\Nzr{\Nzr{G}{H}}{K}) - h(K)\, , \\ &
            \gd_{\RS{i-1}{G}{H}} (\Nzr{G}{H} )  \}.
        \end{aligned}
    \end{equation*}

\end{Corollary}
\begin{proof}
    Corollaries \ref{gdRSKi} and \ref{gd_top_rank_RS_poly} give us the first inequality. The second
    inequality is consequence of the first one and Proposition~\ref{gdRSi}.
\end{proof}

\begin{Corollary}\label{cd_RS_as_max_poly}
    Under the same assumptions than the previous result, we have

    \begin{equation*}
        \begin{aligned}
            \cd_{\RS{i}{G}{H}[K]} (\Nzr{\Nzr{G}{H}}{K}) \leq
            \underset{[K]_i \in \mathcal{K}_i}{\operatorname{max}}\{ &
            h(\Nzr{\Nzr{G}{H}}{K}) - h(K)\, , \\ &
            \cd_{\RS{i-1}{\Nzr{G}{K}}{H}} (\Nzr{\Nzr{G}{H}}{K})\, , \\ &
            \cd_{\RS{i-1}{\Nzr{G}{H}}{K}} (\Nzr{\Nzr{G}{H}}{K}) + 1 \}
        \end{aligned}
    \end{equation*}

    and hence

    \begin{equation*}
        \begin{aligned}
            \cd_{\RS{i}{G}{H}} (\Nzr{G}{H} ) \leq
            \underset{[K]_i \in \mathcal{K}_i}{\operatorname{max}} \{ &
            \cd_{\RS{i-1}{\Nzr{G}{K}}{H}} (\Nzr{\Nzr{G}{H}}{K} ) + 1 \, , \\ &
            \cd_{\RS{i-1}{\Nzr{G}{H}}{K}} (\Nzr{\Nzr{G}{H}}{K} ) + 1 \, , \\ &
            h(\Nzr{\Nzr{G}{H}}{K}) - h(K)\, , \\ &
            \cd_{\RS{i-1}{G}{H}} (\Nzr{G}{H} )  \}.
        \end{aligned}
    \end{equation*}

\end{Corollary}
\begin{proof}
    The first inequality is given by applying Corollaries \ref{cdRSKi} and \ref{cd_top_rank_RS_poly}.
    The second, by applying Theorem~\ref{cd_mv_rs_i_H}.
\end{proof}

We can extend these results to produce upper bounds for the Bredon dimensions
of $\Nzr{G}{H}$ with respect to the family $\RS{i}{G}{H}$, where $H\in\mathcal{H}_r$ and $i \leq r$:

\begin{Proposition}\label{upperbound_dim_RS_polycyclic}
    Let $G$ be a virtually polycyclic group and for every $r\in\N$ let $\HH_r$ be the family
    of subgroups of $G$ of Hirsch length smaller than or equal to $r$. Let $H\in\mathcal{H}_r$
    and $i \leq r - 1$. Then,
    $$\gd_{\RS{i}{G}{H}} (\Nzr{G}{H}) \leq h(\Nzr{G}{H}) + i$$
    and
    $$\cd_{\RS{i}{G}{H}} (\Nzr{G}{H}) \leq h(\Nzr{G}{H}) + i.$$
\end{Proposition}
\begin{proof}
    We will prove it by induction on $i$. For the base case, since $\RS{0}{G}{H}$ coincides with
    the family of finite subgroups of $\Nzr{G}{H}$, we already know
    $\gd_{\RS{0}{G}{H}} \Nzr{G}{H} = h(\Nzr{G}{H})$. Note that this equality holds for all
    virtually polycyclic group $L$ and $S\leq L$ such that $\comm_L S = \Nzr{L}{S}$.

    For the induction step, we assume that for all $L$ virtually polycyclic and all $S\leq L$
    such that $\comm_L S = \Nzr{L}{S}$ we have
    $\gd_{\RS{i-1}{L}{S}} \Nzr{L}{S} \leq h(\Nzr{L}{S}) + i - 1$ and we need to prove that
    $\gd_{\RS{i}{G}{H}} \Nzr{G}{H} \leq h(\Nzr{G}{H}) + i$. Given Corollary~\ref{gd_RS_as_max_poly},
    if is sufficient to prove the following for all $K\in\mathcal{K}_i = \mathcal{H}_r \cap \RS{i}{G}{H}$:
    \begin{enumerate}[label = {(\roman*)}, noitemsep]
        \item $\gd_{\RS{i-1}{\Nzr{G}{K}}{H}} (\Nzr{\Nzr{G}{H}}{K}) \leq h(\Nzr{G}{H}) + i - 1$;
        \item $\gd_{\RS{i-1}{\Nzr{G}{H}}{K}} (\Nzr{\Nzr{G}{H}}{K}) \leq h(\Nzr{G}{H}) + i - 1$ and
        \item $\gd_{\RS{i-1}{G}{H}} (\Nzr{G}{H}) \leq h(\Nzr{G}{H}) + i$.
    \end{enumerate}

    And the three inequalities are true by induction hypothesis applied to
    $H\cap\Nzr{G}{K} \leq \Nzr{G}{K}$, $K \leq \Nzr{G}{H}$ and $H \leq G$, respectively,
    given that if $A \leq B$ then $h(A) \leq h(B)$ for all $A, B$ virtually polycyclic.

    The proof for the Bredon cohomological dimension is the same as the one for the Bredon geometric
    dimension, but using Corollary~\ref{cd_RS_as_max_poly} instead of \ref{gd_RS_as_max_poly}.
\end{proof}

We can now derive similar results for the families $\HH_r [H]$ and $\HH_r$ of subgroups of $\Nzr{G}{H}$
and $G$, respectively.

\begin{Corollary}\label{gd_HH_as_max_poly}
    Let $G$ be a virtually polycyclic group and for every $r\in\N$ let $\HH_r$ be the family
    of subgroups of $G$ of Hirsch length smaller than or equal to $r$. Then,

    \begin{equation*}
        \begin{aligned}
            \gd_{\HH_r[H]} \Nzr{G}{H} \leq \underset{H \in \mathcal{H}_r}{\operatorname{max}}\{ &
            \gd_{\HH_{r-1}\cap \Nzr{G}{H}} \Nzr{G}{H}\, , \,
            h(\Nzr{G}{H}) - h(H)\, , \\ &
            \gd_{\RS{r-1}{G}{H}} \Nzr{G}{H} + 1 \}.
        \end{aligned}
    \end{equation*}

    and hence

    \begin{equation*}
        \begin{aligned}
            \gd_{\HH_r} G \leq \underset{H \in \mathcal{H}_r}{\operatorname{max}} \{ &
            \gd_{\HH_{r-1} \cap \Nzr{G}{H}} (\Nzr{G}{H}) + 1 \, , \\ &
            h(\Nzr{G}{H}) - h(H)\, , \\ &
            \gd_{\RS{r-1}{G}{H}} (\Nzr{G}{H}) + 1  \, , \\ &
            \gd_{\HH_{r-1}} G  \}.
        \end{aligned}
    \end{equation*}
\end{Corollary}
\begin{proof}
    We get the first inequality from Corollaries \ref{gdFrH} and \ref{gd_top_rank_RS_poly}.
    For the second inequality, we use the first and Proposition~\ref{gdFr}.
\end{proof}

\begin{Corollary}\label{cd_HH_as_max_poly}
    Let $G$ and $\HH_r$ as in the previous result. Then,

    \begin{equation*}
        \begin{aligned}
            \cd_{\HH_r[H]} \Nzr{G}{H} \leq \underset{H \in \mathcal{H}_r}{\operatorname{max}}\{ &
            \cd_{\HH_{r-1}\cap \Nzr{G}{H}} \Nzr{G}{H}\, , \,
            h(\Nzr{G}{H}) - h(H)\, , \\ &
            \cd_{\RS{r-1}{G}{H}} \Nzr{G}{H} + 1 \}.
        \end{aligned}
    \end{equation*}

    and hence

    \begin{equation*}
        \begin{aligned}
            \cd_{\HH_r} G \leq \underset{H \in \mathcal{H}_r}{\operatorname{max}} \{ &
            \cd_{\HH_{r-1} \cap \Nzr{G}{H}} (\Nzr{G}{H}) + 1 \, , \\ &
            h(\Nzr{G}{H}) - h(H)\, , \\ &
            \cd_{\RS{r-1}{G}{H}} (\Nzr{G}{H}) + 1  \, , \\ &
            \cd_{\HH_{r-1}} G  \}.
        \end{aligned}
    \end{equation*}
\end{Corollary}
\begin{proof}
    Using Proposition~\ref{cd_mv_FrH}, Corollary~\ref{cd_top_rank_RS_poly} and Theorem~\ref{cd_mv_F_r}.
\end{proof}

Finally, we get the following upper bounds for the Bredon dimensions of $G$ with respect to
the family $\HH_r$, for $r \leq h(G)$:

\begin{Theorem}\label{upperbound_dim_HH_polycyclic}
    Let $G$ be a virtually polycyclic group and, for every $r\in\N$, let $\HH_r$ be the family
    of subgroups of $G$ of Hirsch length smaller than or equal to $r$. Then,
    $$\gd_{\HH_r} G \leq h(G) + r$$
    and
    $$\cd_{\HH_r} G \leq h(G) + r.$$
\end{Theorem}
\begin{proof}
    We proceed, as in the proof of Proposition~\ref{upperbound_dim_RS_polycyclic}, by induction
    over $r$. For the base case, since $\HH_0$ is the family of finite subgroups of $G$, we know that
    $\gd_{\HH_0} G = h(G)$.

    For the inductive step, let $H\in\dF{H}{r}$ be such that $\Nzerr{r}{G}{H} = \Nzr{G}{H}$ and
    assume that for all $L$ virtually polycyclic we have
    $\gd_{\HH_{r-1}\cap L} L \leq h(L) + i$.
    By hypothesis of induction, we have:
    \begin{enumerate}[label={(\roman*)}, noitemsep]
        \item $\gd_{\HH_{r-1}\cap\Nzr{G}{H}} \Nzr{G}{H} \leq h(\Nzr{G}{H}) + r - 1 \leq
            h(G) + r - 1$;
        \item $\gd_{\RS{r-1}{G}{H}} (\Nzr{G}{H}) \leq h(\Nzr{G}{H}) + r - 1$ and
        \item $\gd_{\HH_{r-1}} G \leq h(G) + r - 1$.
    \end{enumerate}

    By Corollary~\ref{gd_HH_as_max_poly}, we have $\gd_{\HH_r} G \leq h(G) + r$, as we
    needed.

    The proof for the Bredon cohomological dimension is the same as the one for the Bredon geometric
    dimension, using Corollary~\ref{cd_HH_as_max_poly} instead of \ref{gd_HH_as_max_poly}.
\end{proof}

We can also give lower bounds for the Bredon dimensions of $G$ with respect to $\HH_r$ that will
be useful in the next section.

\begin{Corollary}\label{lower_bound_vpc}
    Let $G$ be a virtually polycyclic group and, for every $r\in\N$, let $\HH_r$ be the family
    of subgroups $H\leq G$ such that $h(H) \leq r$. Then,
    $$\cd_{\HH_r} G \geq h(G) - r$$
    and
    $$\gd_{\HH_r} G \geq h(G) - r.$$
\end{Corollary}
\begin{proof}
    Consider the families $\HH_0 \subseteq \HH_r$ and $\pi : \HH_0 \to \HH_r$ to be the inclusion.
    Given $H\in\HH_r$, we know that $\cd_{\HH_0 \cap H} H = \gd_{\HH_0 \cap H} H = h(H) \leq r$.
    By Corollary~\ref{inclusion_cd_goes_up_boundedly} and Proposition~\ref{inclusion_gd_goes_up_boundedly},
    since $\HH_0\subseteq\HH_r$ are full families, we can conclude that
    $h(G)  = \cd_{\HH_0} G \leq \cd_{\HH_r} G + r$ and $h(G) = \gd_{\HH_0} G \leq \gd_{\HH_r} G + r$.
\end{proof}

\section{Locally virtually polycyclic groups}

In this section we use Theorems \ref{gd_direct_unions} and \ref{cd_direct_unions} to widen
the class of groups to which the ambient group belongs in the results in the previous section.

\begin{Defn}
    Let $G$ be a locally virtually polycyclic group. Then, we define its Hirsch length
    as $h(G) = \operatorname{sup}\{h(H) \st H \leq G \emph{ finitely generated}\}$.
\end{Defn}

Note that this extension of the definition of the Hirsch length is consistent with that for
virtually polycyclic groups given in the previous section and also with that for elementary amenable groups
given in \cite{hillman1991elementary}.

\begin{Theorem}\label{dim_locally_vpc}
    Let $G$ be a locally virtually polycyclic countable group such that $h(G) < \infty$. Then,
    $$\cd_{\HH_r} G \leq h(G) + r + 1$$
    and
    $$\gd_{\HH_r} G \leq h(G) + r + 1,$$
    for $0 \leq r < h(G)$ and $\cd_{\HH_r} G \leq 1$ and $\gd_{\HH_r} G \leq 1$ for $r \geq h(G)$.
\end{Theorem}
\begin{proof}
    Let $\{G_\lambda \st \lambda\in\Lambda\}$ be the set of finitely generated subgroups of $G$.
    Let $r\in\N$. Since $\HH_r$ is a full family of finitely generated subgroups of $G$, by
    Proposition~\ref{fg_implies_compatible}, the families $\HH_r$ and
    $\{\HH_r \cap G_\lambda \st \lambda\in\Lambda\}$
    are compatible with the direct union. Given that $G$ is a countable group, $\Lambda$ is
    also countable. Hence, by Theorems \ref{cd_direct_unions} and \ref{gd_direct_unions}, we have
    $$\underset{\lambda\in\Lambda}{\operatorname{sup}}\{\cd_{\HH_{r,\lambda}} G_\lambda\} \leq
    \cd_{\HH_r} G \leq
    \underset{\lambda\in\Lambda}{\operatorname{sup}}\{\cd_{\HH_{r,\lambda}} G_\lambda\} + 1$$
    and
    $$\underset{\lambda\in\Lambda}{\operatorname{sup}}\{\gd_{\HH_{r,\lambda}} G_\lambda\} \leq
    \gd_{\HH_r} G \leq
    \underset{\lambda\in\Lambda}{\operatorname{sup}}\{\gd_{\HH_{r,\lambda}} G_\lambda\} + 1,$$
    where $\HH_{r,\lambda} = \HH_r \cap G_\lambda$.

    Let $\lambda\in\Lambda$ and $r\in\N$. Since $G$ is locally virtually polycyclic and
    $G_\lambda$ is finitely generated, $G_\lambda$ is virtually polycyclic.
    Therefore, by Theorem~\ref{upperbound_dim_HH_polycyclic},
    $\cd_{\HH_{r,\lambda}} G_\lambda \leq h(G_\lambda) + r$ and
    $\gd_{\HH_{r,\lambda}} G_\lambda \leq h(G_\lambda) + r$, respectively.
    Note that if $r \geq h(G_\lambda)$, $G_\lambda \in \HH_{r,\lambda}$, so
    $\cd_{\HH_{r,\lambda}} G_\lambda = \gd_{\HH_{r,\lambda}} G_\lambda = 0$, which concludes the
    proof, since $h(G) = \underset{\lambda\in\Lambda}{\operatorname{sup}}\{h(G_\lambda)\}$.
\end{proof}

\begin{Theorem}\label{cd_finite_iff_h_finite}
    Let $G$ be a locally virtually polycyclic countable group and $r\in\N$. Then,
    $\cd_{\HH_r} G < \infty$ if and only if $h(G) < \infty$.
\end{Theorem}
\begin{proof}
    We only need to prove the left-to-right implication, as the other implication is proven
    in Theorem~\ref{dim_locally_vpc}. We want to see that $h(G) < \infty$ assuming that
    $\cd_{\HH_r} G < \infty$. In order to achieve that, we will proceed by contrapositive, i.e.,
    we assume that $h(G) = \infty$ and see that then $\cd_{\HH_r} G$ can not be finite.

    Note that since $G$ is countable, $G$ is the direct union of $\{G_\lambda \st \lambda\in\Lambda\}$,
    where $G_\lambda$ is finitely generated (and hence virtually polycyclic). Proceeding as we did in
    in proof of Theorem~\ref{dim_locally_vpc}, we get that
    $$\cd_{\HH_r} G \geq
    \underset{\lambda\in\Lambda}{\operatorname{sup}}\{\cd_{\HH_{r,\lambda}} G_\lambda\}.$$

    Let $M\in\N$. We want to find $\lambda\in\Lambda$ such that $\cd_{\HH_{r,\lambda}} G_\lambda > M$.
    Since $G_\lambda$ is virtually polycyclic for all $\lambda\in\Lambda$, by
    Corollary~\ref{lower_bound_vpc}, $\cd_{\HH_r\cap G_\lambda} G_\lambda \geq h(G_\lambda) - r$.
    As $h(G) = \infty$ and $h(G) = \underset{\lambda\in\Lambda}{\operatorname{sup}}\{h(G_\lambda)\}$,
    for each $n\in\N$ there is $\lambda(n)\in\Lambda$ such that $h(G_{\lambda(n)}) > n$. If we
    take $\lambda = \lambda(M + r)$, we get that $\cd_{\HH_{r,\lambda}} G_\lambda > M + r - r = M$.
    Therefore, $\cd_{\HH_r} G$ can not be finite, as we wanted to see.
\end{proof}

Analogously, in the case of the Bredon geometric dimension:

\begin{Theorem}\label{gd_finite_iff_h_finite}
    Let $G$ be a locally virtually polycyclic countable group and $r\in\N$. Then,
    $\gd_{\HH_r} G < \infty$ if and only if $h(G) < \infty$.
\end{Theorem}
\begin{proof}
    Exchanging $\cd$ by $\gd$ in the previous theorem's proof yields the desired result.
\end{proof}